\documentclass[12pt,reqno]{amsbook}
\usepackage{amssymb}
\usepackage{amsmath}
\usepackage{graphicx}
\usepackage{subfigure}
\usepackage{multicol}
\usepackage{mathptmx}

\catcode`\@=11
\renewcommand \thesection {\thechapter.\@arabic\c@section}
\catcode`\@=12

\newtheorem{theorem}{Theorem}

\newtheorem{lemma}[theorem]{Lemma}
\newtheorem{corollary}[theorem]{Corollary}

\newtheorem{remark}[theorem]{Remark}

\newtheorem{note}[theorem]{Note}
\newtheorem{example}[theorem]{Example}

\numberwithin{theorem}{section}
\numberwithin{equation}{chapter}

\begin{document}

\def\thechapter{10}


\title{Lattice Path Enumeration}



\author{Christian Krattenthaler}

\address{Fakult\"at f\"ur Mathematik, Universit\"at Wien,
Oskar-Mor\-gen-\break stern-Platz~1, A-1090 Vienna, Austria.\newline\indent
WWW: \tt http://www.mat.univie.ac.at/\~{}kratt.}

\maketitle

{
\chapter{Lattice Path Enumeration}

\def\Z{{\mathbb Z}}
\def\R{{\mathbb R}}
\def\C{{\mathbb C}}
\def\Q{{\mathbb Q}}
\def\al{\alpha}
\def\be{\beta}
\def\de{\delta}
\def\De{\Delta}
\def\et{\eta}
\def\ga{\gamma}
\def\ep{\varepsilon}
\def\la{\lambda}
\def\th{\vartheta}
\def\om{\omega}
\def\si{\sigma}
\def\ph{\varphi}
\def\rh{\rho}
\def\sgn{\operatorname{sgn}}
\def\maj{\operatorname{maj}}
\def\({\left(}
\def\){\right)}

\def\suml{\sum\limits}
\def\bvert{\big\vert}
\def\Bvert{\Big\vert}
\def\area{\operatorname{area}}
\def\run{\operatorname{run}}
\def\Expr{\operatorname{Expr}}
\def\less{\operatorname{less}}
\def\xmaj{\operatorname{xmaj}}
\def\ymaj{\operatorname{ymaj}}
\def\stat{\operatorname{stat}}
\let\hacek\v
\def\vv#1{\left\vert #1\right\vert}
\def\V#1{\left\Vert #1\right\Vert}
\def\fl#1{\left\lfloor #1\right\rfloor}
\def\cl#1{\left\lceil #1\right\rceil}
\def\LL#1{L\big(#1\big)}
\def\P{{\mathbf P}}
\def\A{{\mathbf A}}
\def\E{{\mathbf E}}
\def\Sch{{\mathcal S}}
\def\SS{{\mathbb S}}
\def\Al{{\boldsymbol\alpha}}
\def\Be{{\boldsymbol\beta}}
\def\La{{\boldsymbol\lambda}}
\def\Mu{{\boldsymbol\mu}}
\def\x{{\mathbf x}}
\def\y{{\mathbf y}}
\def\coef#1{\langle #1\rangle}
\def\floor#1{\lfloor#1\rfloor}
\def\GF#1{GF\big(#1\big)}
\def\Pf{\operatorname{Pf}}
\def\NSEW{\text{NSEW}}
\def\Bau{\mathcal B}
\def\Rel{\mathcal R}
\def\Mo{\operatorname{\it M\!o}}
\def\Monomer{\mathcal M}
\def\Dimer{\mathcal D}

\def\qbinom#1#2#3{\mathchoice
   {\begin{bmatrix} #1\\#2\end{bmatrix}_#3}
          {\left[\begin{smallmatrix} #1\\#2\end{smallmatrix}\right]_#3}
          {\left[\begin{smallmatrix} #1\\#2\end{smallmatrix}\right]_#3}
          {\left[\begin{smallmatrix} #1\\#2\end{smallmatrix}\right]_#3}}

\catcode`\@=11
\def\iddots{\mathinner{\mkern1mu\raise\p@\hbox{.}\mkern2mu
    \raise4\p@\hbox{.}\mkern2mu\raise7\p@\vbox{\kern7\p@\hbox{.}}\mkern1mu}}
\catcode`\@=12

\def\DDSSchritt{\leavevmode\raise-.4pt\hbox to0pt{%
  \hbox to0pt{\hss.\hss}\hskip.4\Einheit
  \raise.2\Einheit\hbox to0pt{\hss.\hss}\hskip.4\Einheit
  \raise.4\Einheit\hbox to0pt{\hss.\hss}\hskip.4\Einheit
  \raise.6\Einheit\hbox to0pt{\hss.\hss}\hskip.4\Einheit
  \raise.8\Einheit\hbox to0pt{\hss.\hss}\hss}}
\def\ddSSchritt{\leavevmode\raise-.4pt\hbox to0pt{%
  \hbox to0pt{\hss.\hss}\hskip.4\Einheit
  \raise-.2\Einheit\hbox to0pt{\hss.\hss}\hskip.4\Einheit
  \raise-.4\Einheit\hbox to0pt{\hss.\hss}\hskip.4\Einheit
  \raise-.6\Einheit\hbox to0pt{\hss.\hss}\hskip.4\Einheit
  \raise-.8\Einheit\hbox to0pt{\hss.\hss}\hss}}
\def\SDPfad(#1,#2),#3\endSDPfad{\unskip\leavevmode
  \xcoord#1 \ycoord#2 \ZeichneSDPfad#3\endSDPfad}
\def\ZeichneSDPfad#1{\ifx#1\endSDPfad\let\next\relax
  \else\let\next\ZeichneSDPfad
    \ifnum#1=1
      \raise\ycoord \Einheit\hbox to0pt{\hskip\xcoord \Einheit
         \hSSchritt\hss}%
      \advance\xcoord by 1
    \else\ifnum#1=2
      \raise\ycoord \Einheit\hbox to0pt{\hskip\xcoord \Einheit
        \hbox{\hskip-2pt \vSSchritt}\hss}%
      \advance\ycoord by 1
    \else\ifnum#1=3
      \raise\ycoord \Einheit\hbox to0pt{\hskip\xcoord \Einheit
         \DDSSchritt\hss}
      \advance\xcoord by 1
      \advance\ycoord by 1
    \else\ifnum#1=4
      \raise\ycoord \Einheit\hbox to0pt{\hskip\xcoord \Einheit
         \ddSSchritt\hss}
      \advance\xcoord by 1
      \advance\ycoord by -1
    \fi\fi\fi\fi
  \fi\next}

\newcounter{saveeqn}
\newcommand{\alphaeqn}{\setcounter{saveeqn}{\value{equation}}%
\setcounter{equation}{0}%
\global\def\theequation{\mbox{\thechapter.\arabic{saveeqn}\alph{equation}}}}
\newcommand{\reseteqn}{\setcounter{equation}{\value{saveeqn}}%
\global\def\theequation{\thechapter.\arabic{equation}}}

\catcode`\@=11
\font\tenln    = line10
\font\tenlnw   = linew10

\thinlines
\newskip\Einheit \Einheit=0.6cm
\newcount\xcoord \newcount\ycoord
\newdimen\xdim \newdimen\ydim \newdimen\PfadD@cke \newdimen\Pfadd@cke
\PfadD@cke1pt \Pfadd@cke0.5pt
\def\PfadDicke#1{\PfadD@cke#1 \divide\PfadD@cke by2 \Pfadd@cke\PfadD@cke \multiply\PfadD@cke by2}
\long\def\LOOP#1\REPEAT{\def\BODY{#1}\ITERATE}
\def\ITERATE{\BODY \let\next\ITERATE \else\let\next\relax\fi \next}
\let\REPEAT=\fi
\def\Punkt{\hbox{\raise-2pt\hbox to0pt{\hss\scriptsize$\bullet$\hss}}}
\def\DuennPunkt(#1,#2){\unskip
  \raise#2 \Einheit\hbox to0pt{\hskip#1 \Einheit
          \raise-2.5pt\hbox to0pt{\hss\normalsize$\bullet$\hss}\hss}}
\def\NormalPunkt(#1,#2){\unskip
  \raise#2 \Einheit\hbox to0pt{\hskip#1 \Einheit
          \raise-3pt\hbox to0pt{\hss\large$\bullet$\hss}\hss}}
\def\DickPunkt(#1,#2){\unskip
  \raise#2 \Einheit\hbox to0pt{\hskip#1 \Einheit
          \raise-4pt\hbox to0pt{\hss\Large$\bullet$\hss}\hss}}
\def\Kreis(#1,#2){\unskip
  \raise#2 \Einheit\hbox to0pt{\hskip#1 \Einheit
          \raise-4pt\hbox to0pt{\hss\Large$\circ$\hss}\hss}}
\def\Diagonale(#1,#2)#3{\unskip\leavevmode
  \xcoord#1\relax \ycoord#2\relax
      \raise\ycoord \Einheit\hbox to0pt{\hskip\xcoord \Einheit
         \unitlength\Einheit
         \line(1,1){#3}\hss}}
\def\AntiDiagonale(#1,#2)#3{\unskip\leavevmode
  \xcoord#1\relax \ycoord#2\relax 
      \raise\ycoord \Einheit\hbox to0pt{\hskip\xcoord \Einheit
         \unitlength\Einheit
         \line(1,-1){#3}\hss}}
\def\Pfad(#1,#2),#3\endPfad{\unskip\leavevmode
  \xcoord#1 \ycoord#2 \thicklines\ZeichnePfad#3\endPfad\thinlines}
\def\ZeichnePfad#1{\ifx#1\endPfad\let\next\relax
  \else\let\next\ZeichnePfad
    \ifnum#1=1
      \raise\ycoord \Einheit\hbox to0pt{\hskip\xcoord \Einheit
         \vrule height\Pfadd@cke width1 \Einheit depth\Pfadd@cke\hss}%
      \advance\xcoord by 1
    \else\ifnum#1=2
      \raise\ycoord \Einheit\hbox to0pt{\hskip\xcoord \Einheit
        \hbox{\hskip-\PfadD@cke\vrule height1 \Einheit width\PfadD@cke depth0pt}\hss}%
      \advance\ycoord by 1
    \else\ifnum#1=3
      \raise\ycoord \Einheit\hbox to0pt{\hskip\xcoord \Einheit
         \unitlength\Einheit
         \line(1,1){1}\hss}
      \advance\xcoord by 1
      \advance\ycoord by 1
    \else\ifnum#1=4
      \raise\ycoord \Einheit\hbox to0pt{\hskip\xcoord \Einheit
         \unitlength\Einheit
         \line(1,-1){1}\hss}
      \advance\xcoord by 1
      \advance\ycoord by -1
    \else\ifnum#1=5
      \advance\xcoord by -1
      \raise\ycoord \Einheit\hbox to0pt{\hskip\xcoord \Einheit
         \vrule height\Pfadd@cke width1 \Einheit depth\Pfadd@cke\hss}%
    \else\ifnum#1=6
      \advance\ycoord by -1
      \raise\ycoord \Einheit\hbox to0pt{\hskip\xcoord \Einheit
        \hbox{\hskip-\PfadD@cke\vrule height1 \Einheit width\PfadD@cke depth0pt}\hss}%
    \else\ifnum#1=7
      \advance\xcoord by -1
      \advance\ycoord by -1
      \raise\ycoord \Einheit\hbox to0pt{\hskip\xcoord \Einheit
         \unitlength\Einheit
         \line(1,1){1}\hss}
    \else\ifnum#1=8
      \advance\xcoord by -1
      \advance\ycoord by +1
      \raise\ycoord \Einheit\hbox to0pt{\hskip\xcoord \Einheit
         \unitlength\Einheit
         \line(1,-1){1}\hss}
    \fi\fi\fi\fi
    \fi\fi\fi\fi
  \fi\next}
\def\hSSchritt{\leavevmode\raise-.4pt\hbox to0pt{\hss.\hss}\hskip.2\Einheit
  \raise-.4pt\hbox to0pt{\hss.\hss}\hskip.2\Einheit
  \raise-.4pt\hbox to0pt{\hss.\hss}\hskip.2\Einheit
  \raise-.4pt\hbox to0pt{\hss.\hss}\hskip.2\Einheit
  \raise-.4pt\hbox to0pt{\hss.\hss}\hskip.2\Einheit}
\def\vSSchritt{\vbox{\baselineskip.2\Einheit\lineskiplimit0pt
\hbox{.}\hbox{.}\hbox{.}\hbox{.}\hbox{.}}}
\def\DSSchritt{\leavevmode\raise-.4pt\hbox to0pt{%
  \hbox to0pt{\hss.\hss}\hskip.2\Einheit
  \raise.2\Einheit\hbox to0pt{\hss.\hss}\hskip.2\Einheit
  \raise.4\Einheit\hbox to0pt{\hss.\hss}\hskip.2\Einheit
  \raise.6\Einheit\hbox to0pt{\hss.\hss}\hskip.2\Einheit
  \raise.8\Einheit\hbox to0pt{\hss.\hss}\hss}}
\def\dSSchritt{\leavevmode\raise-.4pt\hbox to0pt{%
  \hbox to0pt{\hss.\hss}\hskip.2\Einheit
  \raise-.2\Einheit\hbox to0pt{\hss.\hss}\hskip.2\Einheit
  \raise-.4\Einheit\hbox to0pt{\hss.\hss}\hskip.2\Einheit
  \raise-.6\Einheit\hbox to0pt{\hss.\hss}\hskip.2\Einheit
  \raise-.8\Einheit\hbox to0pt{\hss.\hss}\hss}}
\def\SPfad(#1,#2),#3\endSPfad{\unskip\leavevmode
  \xcoord#1 \ycoord#2 \ZeichneSPfad#3\endSPfad}
\def\ZeichneSPfad#1{\ifx#1\endSPfad\let\next\relax
  \else\let\next\ZeichneSPfad
    \ifnum#1=1
      \raise\ycoord \Einheit\hbox to0pt{\hskip\xcoord \Einheit
         \hSSchritt\hss}%
      \advance\xcoord by 1
    \else\ifnum#1=2
      \raise\ycoord \Einheit\hbox to0pt{\hskip\xcoord \Einheit
        \hbox{\hskip-2pt \vSSchritt}\hss}%
      \advance\ycoord by 1
    \else\ifnum#1=3
      \raise\ycoord \Einheit\hbox to0pt{\hskip\xcoord \Einheit
         \DSSchritt\hss}
      \advance\xcoord by 1
      \advance\ycoord by 1
    \else\ifnum#1=4
      \raise\ycoord \Einheit\hbox to0pt{\hskip\xcoord \Einheit
         \dSSchritt\hss}
      \advance\xcoord by 1
      \advance\ycoord by -1
    \else\ifnum#1=5
      \advance\xcoord by -1
      \raise\ycoord \Einheit\hbox to0pt{\hskip\xcoord \Einheit
         \hSSchritt\hss}%
    \else\ifnum#1=6
      \advance\ycoord by -1
      \raise\ycoord \Einheit\hbox to0pt{\hskip\xcoord \Einheit
        \hbox{\hskip-2pt \vSSchritt}\hss}%
    \else\ifnum#1=7
      \advance\xcoord by -1
      \advance\ycoord by -1
      \raise\ycoord \Einheit\hbox to0pt{\hskip\xcoord \Einheit
         \DSSchritt\hss}
    \else\ifnum#1=8
      \advance\xcoord by -1
      \advance\ycoord by 1
      \raise\ycoord \Einheit\hbox to0pt{\hskip\xcoord \Einheit
         \dSSchritt\hss}
    \fi\fi\fi\fi
    \fi\fi\fi\fi
  \fi\next}
\def\Koordinatenachsen(#1,#2){\unskip
 \hbox to0pt{\hskip-.5pt\vrule height#2 \Einheit width.5pt depth1 \Einheit}%
 \hbox to0pt{\hskip-1 \Einheit \xcoord#1 \advance\xcoord by1
    \vrule height0.25pt width\xcoord \Einheit depth0.25pt\hss}}
\def\Koordinatenachsen(#1,#2)(#3,#4){\unskip
 \hbox to0pt{\hskip-.5pt \ycoord-#4 \advance\ycoord by1
    \vrule height#2 \Einheit width.5pt depth\ycoord \Einheit}%
 \hbox to0pt{\hskip-1 \Einheit \hskip#3\Einheit 
    \xcoord#1 \advance\xcoord by1 \advance\xcoord by-#3 
    \vrule height0.25pt width\xcoord \Einheit depth0.25pt\hss}}
\def\Gitter(#1,#2){\unskip \xcoord0 \ycoord0 \leavevmode
  \LOOP\ifnum\ycoord<#2
    \loop\ifnum\xcoord<#1
      \raise\ycoord \Einheit\hbox to0pt{\hskip\xcoord \Einheit\Punkt\hss}%
      \advance\xcoord by1
    \repeat
    \xcoord0
    \advance\ycoord by1
  \REPEAT}
\def\Gitter(#1,#2)(#3,#4){\unskip \xcoord#3 \ycoord#4 \leavevmode
  \LOOP\ifnum\ycoord<#2
    \loop\ifnum\xcoord<#1
      \raise\ycoord \Einheit\hbox to0pt{\hskip\xcoord \Einheit\Punkt\hss}%
      \advance\xcoord by1
    \repeat
    \xcoord#3
    \advance\ycoord by1
  \REPEAT}
\def\Label#1#2(#3,#4){\unskip \xdim#3 \Einheit \ydim#4 \Einheit
  \def\lo{\advance\xdim by-.5 \Einheit \advance\ydim by.5 \Einheit}%
  \def\llo{\advance\xdim by-.25cm \advance\ydim by.5 \Einheit}%
  \def\loo{\advance\xdim by-.5 \Einheit \advance\ydim by.25cm}%
  \def\o{\advance\ydim by.25cm}%
  \def\ro{\advance\xdim by.5 \Einheit \advance\ydim by.5 \Einheit}%
  \def\rro{\advance\xdim by.25cm \advance\ydim by.5 \Einheit}%
  \def\roo{\advance\xdim by.5 \Einheit \advance\ydim by.25cm}%
  \def\l{\advance\xdim by-.30cm}%
  \def\r{\advance\xdim by.30cm}%
  \def\lu{\advance\xdim by-.5 \Einheit \advance\ydim by-.6 \Einheit}%
  \def\llu{\advance\xdim by-.25cm \advance\ydim by-.6 \Einheit}%
  \def\luu{\advance\xdim by-.5 \Einheit \advance\ydim by-.30cm}%
  \def\u{\advance\ydim by-.30cm}%
  \def\ru{\advance\xdim by.5 \Einheit \advance\ydim by-.6 \Einheit}%
  \def\rru{\advance\xdim by.25cm \advance\ydim by-.6 \Einheit}%
  \def\ruu{\advance\xdim by.5 \Einheit \advance\ydim by-.30cm}%
  #1\raise\ydim\hbox to0pt{\hskip\xdim
     \vbox to0pt{\vss\hbox to0pt{\hss$#2$\hss}\vss}\hss}%
}

\catcode`\@=12

\def\Line(#1,#2)(#3,#4)#5{\unskip\leavevmode
  \xcoord#1\relax \ycoord#2\relax
      \raise\ycoord \Einheit\hbox to0pt{\hskip\xcoord \Einheit
  \line(#3,#4){#5}\hss}
  }


\section{Introduction}\label{CKsec:intro}

A {\em lattice path\index{lattice 
path}\index{path}\/} 
({\em path\/} for short) is what the name says: a path (walk) in a lattice
in some $d$-dimensional Euclidean space.
Formally, a lattice path $P$ is a sequence $P=(P_0,P_1,\dots,P_l)$ of
points $P_i$ in $\Z^d$. 
Figure~\ref{CKF0.1} shows the lattice path $((0,0),
(1,1),\break (2,1),(3,1),(3,2),(4,3))$.
The point $P_0$ is called the {\em starting point\index{starting
point of a path}\/} and $P_l$ is called the {\em end 
point\index{end point of a path}\/} of $P$. The vectors
$\overrightarrow {P_0P_1},\overrightarrow {P_1P_2}, \dots, \overrightarrow
{P_{l-1}P_{l}}$ are called the {\em steps\index{steps of a 
path}\/} of $P$.

\begin{figure} 
$$
\Gitter(6,5)(0,0)
\Koordinatenachsen(6,5)(0,0)
\Pfad(0,0),31123\endPfad
\DickPunkt(0,0)
\DickPunkt(4,3)
\DickPunkt(1,1)
\DickPunkt(2,1)
\DickPunkt(3,1)
\DickPunkt(3,2)
\hbox{\hskip4cm}
$$
\caption{}
\label{CKF0.1}
\end{figure}

Lattice paths have been studied for a very long time, 
explicitly at least since the second half of the 19th 
century. At the beginning stand the investigations
concerning the two-candidate ballot problem 
\index{Bertrand, Joseph Louis Fran{\c c}ois}%
\index{Whitworth, William Allen}\cite{BertAA,WhitAA}
(see the paragraph below Corollary~\ref{CKT1.2} in
Section~\ref{CKsec:lin1}) and
the gambler's ruin problem \cite{HuygAA} (see
\cite[Ch.~XIV, Sec.~2]{FellAA} and Example~\ref{CKex:gambl}
in Section~\ref{CKs3.6}). Since then, lattice paths have
penetrated many fields of mathematics, computer science,
and physics. The reason for their ubiquity is, on the
one hand, that they are well-suited to encode various
(combinatorial) objects and their properties, and, thus,
problems in various fields can be solved by solving 
lattice path problems.
On the other hand, since lattice paths
are --- at the outset --- reasonably simple combinatorial
objects, the study of physical, probabilistic, or statistical models
is attractive in its own right.
In particular, the importance of lattice path enumeration in
non-parametric statistics seems to explain that the only
books which are entirely devoted to lattice path combinatorics
that I am aware of, namely \cite{MohaAE} and \cite{NaraAB},
are written by statisticians.

The aim of this chapter is to provide an overview of results and
methods in lattice path enumeration. Since, in view of the vast
literature on the subject, comprehensiveness is hopeless, 
I have made a personal selection of topics that I consider of
importance in the theory, the same applying to the methods which
I present here. 

Clearly, when one talks of ``enumeration," this
comes in two different ``flavours": exact and asymptotic.
In this chapter, I only rarely touch asymptotics, but rather
concentrate on exact enumeration results. In most cases, 
corresponding asymptotic results are easily derivable from 
the exact formulas
by using standard methods from asymptotic analysis. See
\cite{FlSeAA} for {\it the} standard text on asymptotic methods
in combinatorial enumeration.

In many cases, I omit proofs. The proofs which are given are either
reasonably short, or they serve to illustrate a key method or idea
in lattice path enumeration. If one attempts to make a list of
the important methods in lattice path enumeration, then this will
include:

\begin{enumerate} 
\item {\em generating functions} (of course), in combination with
the {\em Lagrange inversion formula} and/or {\em residue calculus}
(see the second proof of Theorem~\ref{CKT1.4}, the proof of
Theorem~\ref{CKT1.3a}, and the proof of Theorem~\ref{CKthm:BF1}
for examples);
\item {\em bijections} (they appear explicitly or implicitly at many places);
\item\index{reflection principle}{\em reflection principle} 
(see the proof of Theorem~\ref{CKT1.1} and Section~\ref{CKsec:refl});
\item\index{cycle lemma}{\em cycle lemma} (see Section~\ref{CKsec:rat});
\item\index{transfer matrix method}\index{method, transfer matrix} 
{\em transfer matrix method} (see the proof of Theorem~\ref{CKT3.11});
\item\index{kernel method}\index{method, kernel}{\em kernel method\/}
(see the proof of Theorem~\ref{CKthm:BF2} and the paragraphs 
thereafter);
\item\index{non-intersecting lattice path}\index{path,
  non-intersecting}\index{lattice path,
  non-intersecting}the {\em path switching involution 
for non-intersecting lattice paths} (see Section~\ref{CKsec:nonint});
\item\index{two-rowed array}\index{array, two-rowed}manipulation of
  {\em two-rowed arrays} for turn enumeration (see Section~\ref{CKsec:turn});
\item\index{orthogonal polynomial}\index{polynomial, orthogonal}%
\index{continued fraction}\index{fraction, continued}%
{\em orthogonal polynomials}, {\em continued fractions} (see
Sections~\ref{CKs3.3}--\ref{CKs3.6}).
\end{enumerate}

We start with some simple results on the enumeration of
paths in the $d$-dimen\-sio\-nal integer lattice in
Section~\ref{CKsec:simple}.
The sections which follow, Sections~\ref{CKsec:lin1}--\ref{CKs1.6},
discuss so-called 
\index{simple lattice path}\index{simple path}%
\index{lattice path, simple}\index{path, simple}{\em simple lattice paths}
in the plane integer lattice $\Z^2$; these are paths in $\Z^2$
consisting of horizontal and vertical unit steps in the positive direction. 
While still staying in the plane integer lattice,
beginning from Section~\ref{CKs3.1}, we allow three kinds of steps:
changing the geometry slightly by a rotation about $45^\circ$, these
are up-, down-, and level-steps. The case of Motzkin paths is
intimately related to the theory of orthogonal polynomials and
continued fractions. This link is explained in Sections~\ref{CKs3.3}--\ref{CKs3.6}.
Section~\ref{CKsec:var} provides a loose collection of further results for
lattice paths in the plane integer lattice, with many pointers
to the literature. The subsequent section,
Section~\ref{CKsec:nonint}, is devoted to the theory of non-intersecting
lattice paths, which is an extremely useful enumeration theory
with many applications --- particularly in the enumeration of tilings,
plane partitions, and tableaux ---,
but is also of great interest in its own right.
Turn statistics are investigated in Section~\ref{CKsec:turn}.
Again, the original motivation comes from statistics, but more
recent work, most importantly work on counting non-intersecting
lattice paths by their number of turns, arose from problems
in commutative algebra. Then we move into higher-dimensional
space. Sections~\ref{CKsec:multi}--\ref{CKs7.5} present
standard results for lattice paths in higher-dimensional
lattices. How far one can go with the reflection principle
is explained in Section~\ref{CKsec:refl}. The brief
Section~\ref{CKsec:RR} gives some glimpses of $q$-analogues,
including pointers to the connections of lattice path enumeration
with the Rogers--Ramanujan identities. 

\medskip
We conclude this introduction by fixing some notation which will be
used consistently in this chapter.
(It is in part inspired by standard probability notation.) 
Given lattice points $A$ and
$E$, a set $\SS$ of steps (vectors), a set of restrictions $R$,
and a non-negative integer $m$, we write
\begin{equation} \label{CKeq:LL} 
L_m\big(A\to E;\SS\mid R\big)
\end{equation}
for the set of all lattice paths from $A$ to $E$ with $m$ steps,
all of which from $\SS$, which obey the restrictions in $R$.
The lattice itself in which these paths are considered will be
always clear from the context and is therefore not included in
the notation. For example, the path in Figure~\ref{CKF0.1} is
in 
$$L_5\big((0,0)\to(4,3);\{(1,0),\,(0,1),\,(1,1)\}\mid x\ge y\big),$$
where $x\ge y$ indicates the restriction that the $x$-coordinate
of any lattice point of the path is at least as large as its
$y$-coordinate, or, equivalently, obeys the restriction to
stay weakly below the diagonal $x=y$.

Parts in \eqref{CKeq:LL} may be left out if we do not intend to
require the corresponding restriction, or if that restriction is
clear from the context. For example, the set of lattice paths
in $\Z^2$ from $A$ to $E$ with horizontal and vertical unit steps in
the positive direction without further restriction will be
denoted by $\LL{A\to E;\{(1,0),\,(0,1)\}}$, or sometimes even
shorter, if the step set is clear from the context,
$\LL{A\to E}$.

When we consider \index{weighted counting}{\em weighted
counting}, then we shall also use a uniform notation. 
Given a set $\mathcal M$ and a \index{weight
function}weight
function $w$ on $\mathcal M$, we denote by 
$GF(\mathcal M;w)$\glossary{$GF(\mathcal M;w)$} the
\index{generating function}{\em generating function\/} 
for $\mathcal M$ with respect to $w$, i.e.,
\begin{equation}\label{CKe0.1}
GF(\mathcal M;w):=\sum _{x\in \mathcal M} ^{}w(x).
\end{equation}

Finally, by convention, whenever we write a binomial coefficient $\binom nk$,
it is assumed to be zero if $k$ is not an integer
satisfying $0\leq k \leq n$.

\section{Lattice paths without restrictions}\label{CKsec:simple}

In this short section, we briefly cover the simplest enumeration
problems for lattice paths. If we are given a set of steps $S$,
then the number of paths starting from the origin and using
$n$ steps from $S$ is $\vert S\vert^n$. If we are also fixing
the end point, then we cannot expect a reasonable formula in
this generality. 

However, in the case of (positive) unit steps such formulae
are available. Namely, the number of paths 
in the plane
integer lattice $\Z^2$ from $(a,b)$ to $(c,d)$ 
consisting of horizontal and vertical unit steps in the positive
direction is
\begin{equation} \label{CKe1.1}
\vv{\LL{(a,b)\to (c,d)}}=\binom {c+d-a-b}{c-a},
\end{equation}
since each path from $(a,b)$ to $(c,d)$ can be identified with a
sequence of $(c-a)$ horizontal steps and $(d-b)$ vertical steps, the
number of those sequences being given by the binomial coefficient in
\eqref{CKe1.1}. 

More generally, for the same reason, the number of paths 
in the $d$-dimensional
integer lattice $\Z^d$ from 
$\mathbf a=(a_1,a_2,\dots,a_d)$ to $ \mathbf e=(e_1,e_2,\dots,e_d)$ 
consisting of positive unit steps in the direction
of some coordinate axis
is given by a multinomial coefficient, namely
\begin{equation} \label{CKe1.1d}
\vv{\LL{\mathbf a\to \mathbf e}}=
\binom {\sum_{i=1}^d(e_i-a_i)} 
{e_1-a_1,e_2-a_2,\dots,e_d-a_d}:=
\frac {\left(\sum_{i=1}^d(e_i-a_i)\right)!} 
{(e_1-a_1)!\,(e_2-a_2)!\cdots (e_d-a_d)!}.
\end{equation}

There is another special case, in which one can write down a
closed form expression for the number of paths between two
given points with a fixed number of steps. Namely, 
the number of paths with $n$ horizontal and vertical unit steps
(in the positive {\it or} negative direction) from $(a,b)$ to $(c,d)$
is given by
\begin{equation} \label{CKeq:1.1-4} 
\vv{L_n\big((a,b)\to(c,d);\{(\pm1,0),\,(0,\pm1)\}\big)}
=\binom n {\frac{n+c+d-a-b}{2}}\binom n {\frac{n+c-d-a+b}{2}}.
\end{equation}
See \cite{GuKSAA} and the references given there.

If one considers other step sets then it may often be possible to
obtain (non-closed) formulae by ``mixing" steps. A typical example
is the case where we consider lattice paths in the plane allowing 
three types of steps, namely horizontal
unit steps $(1,0)$, vertical
unit steps $(0,1)$, and diagonal steps $(1,1)$.
Let $S=\{(1,0),(0,1),\break (1,1)\}$ be this step set. If we want to
know how many lattice paths there exist from $(a,b)$ to $(c,d)$
consisting of steps from $S$, then we find
\begin{equation} \label{CKe1.1-5} 
\vv{\LL{(a,b)\to (c,d);S}}=
\sum_{k=0}^{c-a}\binom {c+d-a-b-k} {k,c-a-k,d-b-k},
\end{equation}
since, if we fix the number of diagonal steps to $k$, then the
number of ways to mix $k$ diagonal steps,
$c-a-k$ horizontal steps, and $d-b-k$ vertical steps is given
by the multinomial coefficient which represents the summand
in \eqref{CKe1.1-5}. In the special case where $(a,b)=(0,0)$,
the corresponding numbers are called 
\index{Delannoy numbers}\index{numbers, Delannoy}%
\index{Delannoy, Henri-Auguste}{\it Delannoy numbers}, and, if $(c,d)=(n,n)$,
\index{central Delannoy numbers}\index{numbers, central Delannoy}%
{\it central Delannoy numbers}.

\begin{figure} 
$$
\Gitter(6,5)(0,0)
\Koordinatenachsen(6,5)(0,0)
\Pfad(0,0),21221121\endPfad
\DickPunkt(0,0)
\DickPunkt(1,1)
\DickPunkt(1,2)
\DickPunkt(1,3)
\DickPunkt(2,3)
\DickPunkt(3,3)
\DickPunkt(3,4)
\DickPunkt(4,4)
\hbox{\hskip6cm}
\raise1cm\hbox{%
\Gitter(5,3)(-1,-2)
\Koordinatenachsen(5,3)(-1,-2)
\Pfad(-1,-2),21122121\endPfad
\DickPunkt(-1,-2)
\DickPunkt(-1,-1)
\DickPunkt(0,-1)
\DickPunkt(1,-1)
\DickPunkt(1,0)
\DickPunkt(1,1)
\DickPunkt(2,1)
\DickPunkt(2,2)
\DickPunkt(3,2)
\hskip3cm}
$$
\caption{}
\label{CKF0.2}
\end{figure}

\medskip
As a first excursion to weighted counting, we consider the
generating function for lattice paths in $\Z^2$ from $A=(a,b)$
to $E=(c,d)$ consisting of horizontal and vertical unit steps
in the positive direction, in which each path is weighted by
$q^{a(P)}$, where $a(P)$ denotes the area between the path
and the $x$-axis (with portions of the path which lie below
the $x$-axis contributing a negative area). More precisely,
the area $a(P)$ is the sum of the heights (abscissa) of the
horizontal steps of $P$. For example, for the left-hand path
in Figure~\ref{CKF0.2} we have $a(\,.\,)=1+3+3+4=11$, while for the
right-hand path we have $a(\,.\,)=(-1)+(-1)+1+2=1$.
It is then straightforward to check (by induction on the length
of paths) that
\begin{equation} \label{CKe1.1q}
  \GF{\LL{(a,b)\to (c,d)};q^{a(\,.\,)}}=q^{b(c-a)}\qbinom {c+d-a-b}{c-a}q,
\end{equation}
where $\qbinom {c+d-a-b}{c-a}q$ denotes the 
\index{$q$-binomial coefficient}\index{coefficient, $q$-binomial}{\em
  $q$-binomial coefficient\/} defined by
$$\qbinom nkq:=
\frac {(1-q^n)(1-q^{n-1})\cdots(1-q)} 
{(1-q^k)(1-q^{k-1})\cdots(1-q)(1-q^{n-k})(1-q^{n-k-1})\cdots(1-q)},
$$ 
and $\qbinom nkq=0$ if $k<0$.
This result connects lattice path enumeration with the theory of
\index{integer partition}\index{partition, integer}integer
partitions. What we have computed in \eqref{CKe1.1q} is equivalent
to the classical result that the generating function for integer
partitions with at most $k$ parts, each of which is bounded above
by $n$ is given by $\qbinom {n+k}kq$. We shall say a little bit
more about $q$-counting in Section~\ref{CKsec:RR}. The reader is
referred to \cite{AndrAF} for an excellent survey of the
theory of partitions.

\section{Linear boundaries of slope $1$}\label{CKsec:lin1}
Next we want to count paths from $(a,b)$ to $(c,d)$, where $a\ge b$ and
$c\ge d$, which stay weakly below the main diagonal $y=x$.  
So, what we want to know is the
number $\vv{\LL{(a,b)\to (c,d)\mid x\ge y}}$. This problem is most
conveniently solved by 
the so-called {\em reflection principle\/}\index{reflection principle} 
most often attributed to
\index{Andr\'e, D\'esir\'e}Andr\'e \cite{AndDAA}. However, while
Andr\'e did solve the ballot problem, he did not use the reflection 
principle. Its origin lies most likely in the method of images
of electrostatics, see Sections~2.3--2.6 in \cite{HumKAA}.

\begin{theorem} \label{CKT1.1}%
Let $a\ge b$ and $c\ge d$. The number of
all paths from $(a,b)$ to $(c,d)$ staying weakly below the line $y=x$ is
given by
\begin{equation} \label{CKe1.2}
\vv{\LL{(a,b)\to (c,d)\mid x\ge y}}=\binom {c+d-a-b} {c-a} -
\binom {c+d-a-b}{c-b+1}. 
\end{equation}
\end{theorem}

\begin{proof}First we observe that the number in question is the number
of all paths from $(a,b)$ to $(c,d)$ minus the number of those paths
which cross the line $y=x$,
\begin{multline} \label{CKe1.3}
\vv{\LL{(a,b)\to (c,d)\mid x\ge y}}=\vv{\LL{(a,b)\to (c,d)}}
\\
-\vv{\LL{(a,b)\to (c,d)\mid x\not\ge y\text{ at least once}}}.
\end{multline}
By \eqref{CKe1.1} we already know $\vv{\LL{(a,b)\to (c,d)}}$. The
reflection principle shows that paths from $(a,b)$ to $(c,d)$ which
cross $y=x$ are in bijection with paths from\break $(b-1,a+1)$ to $(c,d)$.
This implies
$$\vv{\LL{(a,b)\to (c,d)\mid x\not\ge y\text{ at least once}}}
=\vv{\LL{(b-1,a+1)\to (c,d)}}.$$
Hence, using \eqref{CKe1.1} again, we establish
\eqref{CKe1.2}.

\begin{figure} 
$$
\Gitter(12,10)(-2,-1)
\Koordinatenachsen(12,10)(-2,-1)
\Pfad(1,-1),22221111212222111112\endPfad
\SPfad(-2,2),111122221211112\endSPfad
\Diagonale(-2,-2){12}
\SPfad(-2,-1),33333333333\endSPfad
\DickPunkt(1,-1)
\DickPunkt(-2,2)
\DickPunkt(7,8)
\DickPunkt(11,9)
\Label\ru{(a,b)}(1,-1)
\Label\l{(b-1,a+1)}(-4,2)
\Label\ro{(c,d)}(11,9)
\Label\ro{P}(5,3)
\Label\ro{ y=x}(2,1)
\Label\o{S}(7,8)
\Label\ro{P'}(-1,2)
\Label\u{y=x+1}(-3,0)
\hskip4cm
$$
\caption{}
\label{CKF1.1}
\end{figure}

The claimed bijection is obtained as follows. Consider a path $P$ from
$(a,b)$ to $(c,d)$ crossing the line $y=x$. See Figure~\ref{CKF1.1} for an
example. Then $P$ must meet the line $y=x+1$. Among all the meeting
points of $P$ and $y=x+1$ choose the right-most. Denote this point by
$S$. Now reflect the portion of $P$ from $(a,b)$ to $S$ in the line
$y=x+1$, leaving the portion from $S$ to $(c,d)$ invariant. Thus we
obtain a new path $P'$ from $(b-1,a+1)$ to $(c,d)$. To construct the
reverse mapping we only have to observe that any path from $(b-1,a+1)$
to $(c,d)$ must meet $y=x+1$ since $(b-1,a+1)$ and $(c,d)$ lie on
different sides of $y=x+1$. Again we choose the right-most meeting
point, denote it by $S$, and reflect the portion from $(b-1,a+1)$ to
$S$ in the line $y=x+1$, thus obtaining a path from $(a,b)$ to $(c,d)$
that meets the line $y=x+1$, or, equivalently, crosses the line $y=x$.
\end{proof}

In particular, for $a=b=0$ we obtain the following compact formula.

\begin{corollary} \label{CKT1.2}%
If $c\ge d$ we have
\begin{equation} \label{CKe1.4}
\vv{\LL{(0,0)\to (c,d)\mid x\ge y}}=\frac {c+1-d} {c+d+1}\binom {c+d+1}
{d}
\end{equation}
and
\begin{equation} \label{CKe1.5}
\vv{\LL{(0,0)\to (n,n)\mid x\ge y}}= \frac {1} {n+1}\binom {2n}{n}.
\end{equation}
\end{corollary}
The numbers $\frac {c+1-d} {c+d+1}\binom {c+d+1}{d}$ are called {\em
\index{ballot numbers}ballot numbers\/} 
since they give the answer to the classical ballot
problem\index{ballot problem}, which is usually attributed to 
\index{Bertrand, Joseph Louis Fran{\c c}ois}Bertrand 
\cite{BertAA}, but was actually first stated and solved by
Whitworth\index{Whitworth, William Allen} \cite{WhitAA}. 
The problem is stated as follows:
in an election
candidate $A$ received $c$ votes and candidate $B$ received $d$
votes;
how many ways of counting the votes are there such that at each stage
during the counting candidate $A$ has at least as many votes as
candidate $B$? By representing a vote for $A$ by a horizontal step and a
vote for $B$ by a vertical step, it is seen that the number  in
question is the same as the number of lattice paths from $(0,0)$ to
$(c,d)$ staying weakly below $y=x$. This number is given in \eqref{CKe1.4}.
More about ballot problems appears in Sections~\ref{CKsec:var} and
\ref{CKsec:refl}.

The numbers $\frac {1} {n+1}\binom {2n}{n}$ are called {\em \index{Catalan
number}\index{Catalan, Eug\`ene Charles}Catalan
numbers\/} \cite{CataAA,CataAB}. 
However, they have been considered earlier by 
\index{Segner, Johann Andreas}Segner \cite{SegnAA} and 
\index{Euler, Leonhard}Euler \cite{EuleAA}, and independently even 
earlier in China; see the historical remarks in \cite{PakIAA}
and \cite[p.~212]{StanBI}.
They appear in numerous places; see \cite[Ex.~6.19]{StanBI},
with many more occurrences in the addendum \cite{Stadd}. 

\medskip
An iterated reflection argument will give the number
of paths between two diagonal lines. 

\begin{theorem} \label{CKT1.3}%
Let $a+t\ge b\ge a+s$ and $c+t\ge d\ge c+s$. The number of all paths
from $(a,b)$ to $(c,d)$ staying weakly below the line $y=x+t$ and above the
line $y=x+s$ is given by
\begin{multline} \label{CKe1.6}
\vv{\LL{(a,b)\to (c,d)\mid x+t\ge y\ge x+s}}\\
=\sum _{k\in\Z} ^{}\(\binom
{c+d-a-b} {c-a-k(t-s+2)}-\binom {c+d-a-b}{c-b-k(t-s+2)+t+1}\). 
\end{multline}
\end{theorem}

Since this is (as well as Theorem~\ref{CKT1.1}) an instance of
the general formula \eqref{CKe7.13} for the number of paths staying in 
regions defined by hyperplanes, we omit the proof.

\medskip
The formula in Theorem~\ref{CKT1.3} is very convenient for computing
the number of paths as long as the parameters are not too large.
On the other hand, it is of no use if one is interested in asymptotic
information, because the summands on the right-hand side of \eqref{CKe1.6}
alternate in sign so that there is considerable cancellation.
However, with the help of little residue calculus, the formula 
can be transformed into a surprising formula featuring
cosines and sines, from which asymptotic information can be easily
read off.

\begin{theorem} \label{CKT1.3a}%
Let $a+t\ge b\ge a+s$ and $c+t\ge d\ge c+s$. The number of all paths
from $(a,b)$ to $(c,d)$ staying weakly below the line $y=x+t$ and above the
line $y=x+s$ is given by
\begin{multline} \label{CKe1.6a}
\vv{\LL{(a,b)\to (c,d)\mid x+t\ge y\ge x+s}}\\
=\sum _{k=1} ^{\fl{(t-s+1)/2}}\frac {4} {t-s+2}\(2\cos\frac {\pi k}
{t-s+2}\)^{c+d-a-b}\kern3cm\\
\cdot
\sin\(\frac {\pi k(a-b+t+1)} {t-s+2}\)\cdot
\sin\(\frac {\pi k(c-d+t+1)} {t-s+2}\).
\end{multline}
\end{theorem}
\begin{proof} 
Trivially, the binomial coefficient $\binom nk$ is the
coefficient of $z^{-1}$ in the Laurent series
$$\frac {(1+z)^n} {z^{k+1}}.$$
Thus, the sum \eqref{CKe1.6} equals the
coefficient of $z^{-1}$ in
\begin{align} \notag
\sum _{k=0} ^{\infty}&\bigg(\frac {(1+z)^{c+d-a-b}\,z^{k(t-s+2)}} 
{z^{c-a+(c+d-a-b)(t-s+2)+1}}-
\frac {(1+z)^{c+d-a-b}\,z^{k(t-s+2)}} {z^{c-b+t+(c+d-a-b)(t-s+2)+2}}\bigg)\\
\notag
&\kern-1pt
=\frac {(1+z)^{c+d-a-b}} {z^{c-a+(c+d-a-b)(t-s+2)+1}(1-z^{t-s+2})}\\
\notag
&\kern4cm
-
\frac {(1+z)^{c+d-a-b}} {z^{c-b+t+(c+d-a-b)(t-s+2)+2}(1-z^{t-s+2})}\\
&\kern-1pt
=\frac {(1+z)^{c+d-a-b}\(z^{(-c+d+a-b)/2}-z^{(-c+d-a+b)/2-t-1}\)} 
{z^{(c+d-a-b)/2+(c+d-a-b)(t-s+2)+1}(1-z^{t-s+2})}. \label{CKeq:coef}
\end{align}
(In the second line we used the formula for the geometric series. It
can be either regarded as a summation in the formal sense, or else
one must assume that $\vert z\vert<1$.)
Equivalently, 
the sum \eqref{CKe1.6} equals the residuum of the Laurent series
\eqref{CKeq:coef} at $z=0$. Now consider the contour integral of \eqref{CKeq:coef}
(with respect to $z$, of course)
along a circle of radius $r$ around the origin. It is a standard fact
that in the limit $r\to\infty$ this integral vanishes, because the
integrand \eqref{CKeq:coef} is of the order $O(1/z^2)$. Therefore, by
the theorem of residues, the
sum of the residues of \eqref{CKeq:coef} must be 0, or, equivalently,
the residuum at $z=0$, which we are interested in, equals the negative of
the sum of the other residues. As the other poles of \eqref{CKeq:coef}
are the $(t-s+2)$-th roots of unity different from 1, we obtain
\begin{multline*} 
-\sum _{k=1} ^{t-s+1}
\frac {\(1+e^{\frac {2\pi ik} {t-s+2}}\)^{c+d-a-b}
\(e^{\frac {\pi ik} {t-s+2}(-c+d+a-b)}-
e^{\frac {\pi ik} {t-s+2}(-c+d-a+b-2t-2)}\)} 
{e^{\frac {2\pi ik} {t-s+2}(\frac
{c+d-a-b} {2}+1)}\(-(t-s+2)e^{\frac {2\pi ik} {t-s+2}(t-s+1)}\)}\\
=\sum _{k=1} ^{t-s+1}\frac {1} {t-s+2}\(2\cos\frac {\pi k}
{t-s+2}\)^{c+d-a-b}
e^{\frac {\pi ik} {t-s+2}(-c+d-t-1)}\\
\cdot \(e^{\frac {\pi ik} {t-s+2}(a-b+t+1)}-
e^{-\frac {\pi ik} {t-s+2}(a-b+t+1)}\)
\end{multline*}
for the sum \eqref{CKe1.6}.
Now, in the last line, 
we pair the $k$-th and the $(t-s+2-k)$-th summand. Thus, upon little
manipulation, the above sum turns into
\begin{multline*} 
\sum _{k=1} ^{\fl{(t-s+1)/2}}\frac {1} {t-s+2}\(2\cos\frac {\pi k}
{t-s+2}\)^{c+d-a-b}\\
\cdot
\(e^{-\frac {\pi ik} {t-s+2}(c-d+t+1)}-
e^{\frac {\pi ik} {t-s+2}(c-d+t+1)}\) 
{\(e^{\frac {\pi ik} {t-s+2}(a-b+t+1)}-
e^{-\frac {\pi ik} {t-s+2}(a-b+t+1)}\)}.
\end{multline*}
Clearly, this formula is equivalent to \eqref{CKe1.6a}.
\end{proof}

From the generating function formula given in Section~\ref{CKs3.6}
(see\break Example~\ref{CKex:Cheb}), 
one can see that this asymptotic formula comes from 
\index{Chebyshev, Pafnuty Lvovich}%
\index{polynomial, Chebyshev}\index{Chebyshev polynomial}{\em Chebyshev
polynomials}.

\section{Simple paths with linear boundaries of rational slope, I}
\label{CKsec:rat}

When we want to count simple lattice paths
(recall the meaning of ``simple" from the introduction) in the plane 
bounded by an arbitrary line
$y=kx+d$, $k,d\in\R$, the reflection principle obviously does not
help, since the reflection of a lattice path in a generic
line does not necessarily give a lattice path. In
fact, a solution in form of a determinant can be given when the
boundary is viewed as a special case of a set of general boundaries
(see Section~\ref{CKs1.6}, Theorem~\ref{CKT1.15}; another solution was
proposed by 
\index{Tak\'acs, L\'aszlo}Tak\'acs \cite{TakaAH}, which is of similar complexity as
it involves the solution of a large system of linear equations). 
However, there are cases where
simpler expressions can be obtained, and these are discussed in this
section. All of them can be derived 
from a very basic combinatorial lemma,
the so-called \index{cycle lemma}{\it``cycle lemma"}, which exists
in several variations.

The first case which we discuss is the enumeration of simple lattice paths
from the origin to a lattice point $(r,s)$, with $r$ and $s$ relatively
prime, which stay weakly below the line connecting the origin and $(r,s)$.

\begin{theorem} \label{CKT1.20}%
Let $r$ and $s$ be relatively prime positive integers. The number of
all paths from $(0,0)$ to $(r,s)$ staying weakly below the line $ry=sx$ is
given by
\begin{equation} \label{CKe1.40}
\vv{\LL{(0,0)\to (r,s)\mid sx\ge ry}}=\frac {1} {r+s}\binom {r+s} {r} .
\end{equation}
\end{theorem}

\begin{remark}\em
The numbers in \eqref{CKe1.40} are nowadays called
\index{rational Catalan numbers}{\it rational Catalan numbers}
(cf.\ \cite{ArRWAA}), the Catalan numbers being the special case
where $r=n$ and $s=n+1$.
\end{remark}

The above result follows easily from a form of the 
\index{cycle lemma}cycle lemma
which is known in the statistics literature as 
\index{Spitzer's lemma}\index{Spitzer, Frank Ludvig}Spitzer's lemma \cite{SpitAA}.

\begin{lemma}[\sc Spitzer's Lemma]\label{CKlem:Spitzer}%
Let $a_1,a_2,\dots,a_N$ be real numbers with the property that
$a_1+a_2+\dots+a_N=0$ and no other partial sum of consecutive 
$a_i$'s, read cyclically (by which we mean sums of the form
$a_j+a_{j+1}+\dots+a_k$ with $j\le k$ and $k-j<N$, where indices are 
interpreted modulo $N$), vanishes. Then there exists a unique
cyclic permutation $a_i,a_{i+1},\dots,a_N,a_1,\dots,a_{i-1}$
with the property that for all
$j=1,2,\dots,N$ the sum of the first $j$ letters of this
permuted array is non-negative.
\end{lemma}

\begin{remark}\em
This lemma could be further generalized by weakening the above assumption
to demanding that $K$ partial sums of consecutive $a_i$'s, read
cyclically, of {\it minimal length\/} vanish, with the
conclusion that there be $K$ cyclic permutations with the above
non-negativity property.
\end{remark}

\begin{proof}We interpret the real numbers $a_i$ as steps of
a path (although not necessarily of a {\it lattice} path), by concatenating
the steps $(1,a_1)$, $(1,a_2)$, \dots, $(1.a_N)$ to a path starting
at the origin. 
See the left half of Figure~\ref{CKF1.20} for a typical example.

\begin{figure} 
$$
\Einheit=.05cm
         \unitlength\Einheit
\Koordinatenachsen(100,50)(-2,-1)
\thicklines
  \Line(0,0)(1,4){10}
  \Line(10,40)(2,-1){10}
  \Line(20,35)(1,-1){10}
  \Line(30,25)(1,-3){10}
  \Line(40,-5)(1,1){10}
  \Line(50,5)(1,-2){10}
  \Line(60,-15)(2,1){10}
  \Line(70,-10)(1,3){10}
  \Line(80,20)(1,-2){10}
  \DuennPunkt(0,0)
  \DuennPunkt(10,40)
  \DuennPunkt(20,35)
  \DuennPunkt(30,25)
  \DuennPunkt(40,-5)
  \DuennPunkt(50,5)
  \DickPunkt(60,-15)
  \DuennPunkt(70,-10)
  \DuennPunkt(80,20)
  \DuennPunkt(90,0)
  \Label\u{A}(60,-15)
\hbox{\hskip6cm}
\Koordinatenachsen(100,50)(-2,-1)
\thicklines
  \Line(0,0)(2,1){10}
  \Line(10,5)(1,3){10}
  \Line(20,35)(1,-2){10}
  \Line(30,15)(1,4){10}
  \Line(40,55)(2,-1){10}
  \Line(50,50)(1,-1){10}
  \Line(60,40)(1,-3){10}
  \Line(70,10)(1,1){10}
  \Line(80,20)(1,-2){10}
  \DickPunkt(0,0)
  \DuennPunkt(10,5)
  \DuennPunkt(20,35)
  \DuennPunkt(30,15)
  \DuennPunkt(40,55)
  \DuennPunkt(50,50)
  \DuennPunkt(60,40)
  \DuennPunkt(70,10)
  \DuennPunkt(80,20)
  \DickPunkt(90,0)
  \Label\u{A}(0,0)
  \Label\u{A}(90,0)
\hskip5cm
$$
\caption{}
\label{CKF1.20}
\end{figure}

Since the sum of all $a_i$'s vanishes, the end point of the path
lies on the $x$-axis. We identify this end point with the starting
point (located at the origin), so that we consider this path as a
cyclic object.

By the non-vanishing of cyclic subsums, there is a unique point
of minimal height, $A$ say. (This may also be the starting/end point, which
we identified.) In the figure this point is marked by a thick dot.
Now ``permute" the path cyclically, that is, take the portion of
the path from $A$ to the end, and concatenate it with the initial
portion of the path until $A$. See the right half of
Figure~\ref{CKF1.20} for the result in our example. Obviously, the
new path always lies strictly above the $x$-axis, except at the
beginning and at the end. This identifies the cyclic permutation of
the $a_i$'s with the required property.
\end{proof}

\noindent
{\bf Proof of Theorem~\ref{CKT1.20}}\indent
We consider {\it all\/} paths from $(0,0)$ to $(r,s)$.
There are $\binom {r+s}s$ such paths. Given a path $P$ from
$(0,0)$ to $(r,s)$, we consider the sequence $a_1,a_2,\dots,a_{r+s}$,
where $a_i=s$ if the $i$-th step of the path is a horizontal
step, and $a_i=-r$ if the $i$-th step of the path is a vertical
step. Since $r$ and $s$ are relatively prime, no cyclic subsum of
the $a_i$'s, except the complete sum, can vanish. The cycle lemma
in Lemma~\ref{CKlem:Spitzer} then implies that, out of the $r+s$
cyclic ``permutations" of the path $P$, there is exactly one
which stays (weakly) below the line $sx=ry$. Thus, there are
in total $\frac {1} {r+s}\binom {r+s}r$ paths with that property.
\bigskip

The next case where a closed form formula can be obtained
(partially overlapping with the result in Theorem~\ref{CKT1.20}) is when counting
lattice paths from $(0,0)$ to $(c,d)$ which stay weakly below the line
$x=\mu y$, where $\mu$ is a positive integer. Of course we have to
assume $c\ge \mu d$. There are two conceptually different 
standard approaches to obtain
the corresponding result: application of another version of the
cycle lemma (see Lemma~\ref{CKT1.5}), respectively generating functions
combined with the use of the Lagrange inversion formula.

\begin{theorem} \label{CKT1.4}%
Let $\mu$ be a non-negative integer and
$c\ge \mu d$. The number of all lattice paths from the origin to
$(c,d)$ which lie weakly below $x=\mu y$ is given by
\begin{equation} \label{CKe1.8}
\vv{\LL{(0,0)\to (c,d)\mid x\ge\mu y}}=\frac {c-\mu d+1} {c+d+1}\binom
{c+d+1} {d}  .
\end{equation}
\end{theorem}

This result is essentially equivalent to the cycle lemma due to
\index{Dvoretzky, Aryeh}Dvoretzky and 
\index{Motzkin, Theodore S.}Motzkin
\cite{DvMoAA}. It has been rediscovered many times;
see \cite{DeZaAA} for a partial survey and many related references,
as well as \cite[Lemma~5.3.6 and Example~5,3,7]{StanBI}.

\begin{lemma}[\sc Cycle Lemma]\label{CKT1.5}%
Let $\mu$ be a non-negative integer. For any sequence $p_1p_2\dots
p_{m+n}$ of $m$ $1$'s and $n$ $2$'s, with $m\ge\mu n$, there exist exactly
$m-\mu n$ cyclic permutations $p_ip_{i+1}\dots p_{m+n}p_1\dots
p_{i-1}$, $1\le i\le m+n$, that have the property that for all
$j=1,2,\dots,m+n$ the first $j$ letters of this permutation contain
more 1's than $\mu$ times the number of 2's. 
\end{lemma}

\begin{proof} A sequence $p_1p_2\dots p_{m+n}$ of $m$ 1's and $n$ 2's can
be seen as a lattice path from $(0,0)$ to $(m,n)$ by interpreting the
1's as horizontal steps and the 2's as vertical steps. Cyclically
permuting $p_1p_2\dots p_{m+n}$ means to cut the corresponding lattice
path into two pieces and put them together in exchanged order, thus
obtaining a new lattice path from $(0,0)$ to $(m,n)$. Finally, the
property that in each initial string of a sequence the number of 1's
dominates (i.e., 
is larger than) $\mu$ times the number of 2's means that the corresponding
lattice path stays strictly below the line $x=\mu y$, with the
exception of the starting point $(0,0)$. 

For the proof of the lemma interpret $p_1p_2\dots p_{m+n}$ as a path,
as described before, and join a shifted copy of this path at the end
point $(m,n)$, another shifted copy at $(2m,2n)$, etc.
Figure~\ref{CKF1.4} shows an example with $\mu=2$, $m=9$, $n=3$. The
path $P$ corresponds to the sequence $121111122111$.
\begin{figure} 
$$
\Einheit.5cm
\Gitter(20,10)(-1,0)
\Koordinatenachsen(20,10)(-1,0)
\Pfad(0,0),121111122111 121111122111\endPfad
\SPfad(18,6),121\endSPfad
\PfadDicke{2.5pt}
\Pfad(1,1),1\endPfad
\Pfad(6,3),11\endPfad
\Pfad(10,4),1\endPfad
\Pfad(15,6),11\endPfad
\DickPunkt(0,0)
\DickPunkt(9,3)
\DickPunkt(18,6)
\SDPfad(-1,0),3\endSDPfad
\SDPfad(2,1),3\endSDPfad
\SDPfad(4,2),3\endSDPfad
\SDPfad(8,3),3\endSDPfad
\SDPfad(11,4),3\endSDPfad
\SDPfad(13,5),3\endSDPfad
\SDPfad(17,6),3\endSDPfad
\Label\u{(m,n)}(9,3)
\hbox to0pt{$\underbrace{\vrule height0pt width9 \Einheit
depth0pt}_{\textstyle P}\hss$}
\hskip9cm
$$
\caption{}
\label{CKF1.4}
\end{figure}
Cyclic permutations of $p_1p_2\dots p_{m+n}$ correspond to cutting a
piece of $m+n$ successive steps out of this lattice path structure.
Then imagine a sun to be located in direction $(\mu,1)$ illuminating
the lattice path structure. A cyclic permutation will satisfy the
dominance property for each initial string if and only if the first
step of the corresponding lattice path is illuminated. 
In Figure~\ref{CKF1.4} the illuminated steps are
indicated by thick lines. It is an easy
matter of fact that of any $m+n$ successive steps there are exactly $m-\mu n$
illuminated steps. Therefore out of the $m+n$ cyclic permutations of
$p_1p_2\dots p_{m+n}$ there are exactly $m-\mu n$ cyclic permutations
having the
dominance property for each initial string.
\end{proof}

\noindent
{\bf First proof of Theorem~\ref{CKT1.4}}\indent
We want to count paths from $(0,0)$ to $(c,d)$ staying weakly below $x=\mu
y$. To fit with the cycle lemma we adjoin a horizontal step at the
beginning and shift everything by one unit to the right. Thus we are
now asking for the number of paths from $(0,0)$ to $(c+1,d)$ staying
{\em strictly\/} below $x=\mu y$, except for the starting point $(0,0)$.
Now one applies Lemma~\ref{CKT1.5} with $m=c+1$ and $n=d$: 
given a path $P$ from $(0,0)$ to
$(c+1,d)$, exactly $m-\mu n=c+1-\mu d$ 
of its cyclic ``permutations" satisfy the property of staying
{\em strictly\/} below $x=\mu y$, except for the starting point
$(0,0)$. Thus, the total number of paths from $(0,0)$ to $(c,d)$
with that property is given by \eqref{CKe1.8}.
\bigskip

For instructional purposes, we also present the generating function
proof. 

\bigskip
\noindent
{\bf Second proof of Theorem~\ref{CKT1.4}}\indent
The generating function proof works in two steps. First, an
equation is found for the generating function of those paths which
return in the end to the boundary $x=\mu y$. Then, in a second
step, paths ending arbitrarily are decomposed into paths of the
former type, leading to a generating function expression in
terms of the earlier generating function to which the Lagrange
inversion formula is applicable.

Let $P$ be a path in $\LL{(0,0) \to (\mu d,d)\mid x\ge\mu
y}$ (see Figure~\ref{CKF1.2}). 
\begin{figure} 
$$
\Gitter(12,8)(-1,0)
\Koordinatenachsen(12,8)(-1,0)
\Pfad(0,0),111212111211122\endPfad
\raise-1 \Einheit\hbox to0pt{\hskip-2 \Einheit
         \unitlength\Einheit
         \line(2,1){13}\hss}
\DickPunkt(0,0)
\DickPunkt(4,2)
\DickPunkt(5,2)
\DickPunkt(7,3)
\DickPunkt(8,3)
\DickPunkt(10,4)
\DickPunkt(10,5)
\Label\ru{\raise10pt\hbox{\hskip4pt$s_h$}}(4,2)
\Label\ru{\raise10pt\hbox{\hskip4pt$s_h$}}(7,3)
\Label\ru{\hskip-2pt s_v}(10,5)
\Label\ro{x=\mu y}(12,5)
\Label\o{S_0}(4,2)
\Label\o{S_1}(7,3)
\SDPfad(5,2),3\endSDPfad
\SDPfad(8,3),3\endSDPfad
\hbox to0pt{$\underbrace{\vrule height0pt width4 \Einheit
depth0pt}_{\textstyle P_0}\hss$}
\hbox to0pt{$\vrule height0pt width5 \Einheit depth0pt
\underbrace{\vrule height0pt width2 \Einheit depth0pt}_{\textstyle P_1}\hss$}
\hbox to0pt{$\vrule height0pt width8 \Einheit depth0pt
\underbrace{\vrule height0pt width2 \Einheit depth0pt}_{\textstyle P_2}\hss$}
\hskip6cm
$$
\caption{}
\label{CKF1.2}
\end{figure}
For $l=0,1,\dots,\break\mu-1$, the path $P$ will
meet the line $x=\mu y+l$ (which is parallel to our boundary $x=\mu
y$) somewhere for the last time. Denote this point by $S_l$. Clearly, the path
$P$ must leave $S_l$ by a horizontal step, which we denote by $s_h$
for short. This gives us a unique decomposition of $P$ of the form
$$P=P_0s_hP_1s_h\dots s_hP_\mu s_v,$$
where $P_0$ is $P$'s portion from the origin up to $S_0$, $P_1$ is
$P$'s portion from the point immediately following $S_0$ up to $S_1$,
etc. By $s_v$ we denote the final vertical step. All the portions $P_i$ (when
shifted appropriately) belong to $\LL{(0,0)\to (\mu n,n)\mid x\ge\mu
y}$ for some $n$. Let
\begin{equation} \label{CKeq:L0} 
\mathcal L_0=\bigcup _{n\ge 0} ^{}\LL{(0,0)\to (\mu n,n)\mid
x\ge\mu y}.
\end{equation}
Then we have the following decomposition:
\begin{equation*} 
\mathcal L_0=\{\ep\}\cup \big((\mathcal L_0 s_h)^\mu \mathcal L_0 s_v\big).
\end{equation*}
Here, as always in the sequel, \glossary{$\ep$}$\ep$ denotes the empty path.

By elementary combinatorial principles, 
this immediately translates into a functional equation
for the generating function
$$F_0(z):=\sum _{n\ge 0} ^{}\vv{\LL{(0,0)\to (\mu n,n)\mid x\ge\mu
y}}z^n$$
for $\mathcal L_0$ (note that the summation index 
$n$ records the vertical height of the end point of paths), namely
\begin{equation} \label{CKe1.9}
F_0(z)=1+zF_0(z)^{\mu +1}. 
\end{equation}
If we write $F_0(z)=1+G_0(z)$, then Equation~\eqref{CKe1.9} in terms
of the series $G_0(z)$ reads
\begin{equation} \label{CKe1.10}
\frac {G_0(z)} {(1+G_0(z))^{\mu+1}}=z, 
\end{equation}
which simply says that $G_0(z)$ is the compositional inverse of
$z/(1+z)^{\mu+1}$. 

Turning to the more general problem,
consider a lattice path $P$ in $\big\vert\LL{(0,0)\to \linebreak[4](\mu d+k,d)\mid
x\ge \mu y}\big\vert$ (see Figure~\ref{CKF1.3}). 
\begin{figure} 
$$
\Gitter(13,8)(-1,0)
\Koordinatenachsen(13,8)(-1,0)
\Pfad(0,0),112112111211112\endPfad
\raise-1 \Einheit\hbox to0pt{\hskip-2 \Einheit
         \unitlength\Einheit
         \line(2,1){14}\hss}
\DickPunkt(0,0)
\DickPunkt(4,2)
\DickPunkt(5,2)
\DickPunkt(7,3)
\DickPunkt(8,3)
\DickPunkt(9,3)
\DickPunkt(11,4)
\Label\ru{\raise10pt\hbox{\hskip4pt$s_h$}}(4,2)
\Label\ru{\raise10pt\hbox{\hskip4pt$s_h$}}(7,3)
\Label\ru{\raise10pt\hbox{\hskip4pt$s_h$}}(8,3)
\Label\lo{x=\mu y}(10,5)
\Label\o{S_0}(4,2)
\Label\o{S_1}(7,3)
\Label\o{S_2}(8,3)
\Label\ro{(c,d)}(11,4)
\SDPfad(5,2),3\endSDPfad
\SDPfad(9,3),3\endSDPfad
\hbox to0pt{$\underbrace{\vrule height0pt width4 \Einheit
depth0pt}_{\textstyle P_0}\hss$}
\hbox to0pt{$\vrule height0pt width5 \Einheit depth0pt
\underbrace{\vrule height0pt width2 \Einheit depth0pt}_{\textstyle P_1}\hss$}
\hbox to0pt{$\vrule height0pt width7.5 \Einheit depth0pt
\underbrace{\vrule height0pt width0 \Einheit depth0pt}_{\textstyle P_2}\hss$}
\hbox to0pt{$\vrule height0pt width9 \Einheit depth0pt
\underbrace{\vrule height0pt width2 \Einheit depth0pt}_{\textstyle P_3}\hss$}
\hskip6.5cm
$$
\caption{}
\label{CKF1.3}
\end{figure}
For $l=0,1,\dots,k-1$ the
path will meet the line $x=\mu y+l$ somewhere
for the last time. Denote this point by $S_l$. Clearly, the path
$P$ must leave $S_l$ by a horizontal step, for which we again write
$s_h$. This gives us a decomposition of $P$ of the form
$$P=P_0s_hP_1s_h\dots s_h P_{k},$$
where $P_0$ is the portion of $P$ from the origin up to $S_0$, $P_1$ is
the portion of $P$ from the point immediately following $S_0$ up to $S_1$,
and so on. Observe that again the portions $P_l$ belong to 
$\mathcal L_0$ (being defined in \eqref{CKeq:L0}).
Let
$$\mathcal L_k=
\bigcup _{n\ge 0} ^{}\LL{(0,0)\to (\mu n +k,n)\mid x\ge \mu y}.$$
Then we have the following decomposition:
$$\mathcal L_k=(\mathcal L_0s_h)^{k}\mathcal L_0.$$
This translates again into an equation for the corresponding generating
function 
$$F_k(z):=\sum _{n\ge0} ^{}\vv{\LL{(0,0)\to (\mu n+k,n)\mid x\ge\mu y}}
z^n$$
for $\mathcal L_k$, namely into
\begin{equation*} 
F_k(z)=F_0(z)^{k+1}=(1+G_0(z))^{k+1}, 
\end{equation*}
We noted above that $G_0(z)$ is the compositional inverse of 
$z/(1+z)^{\mu+1}$. Therefore, we may apply
the Lagrange formula (see \cite[Corollary~5.4.3]{StanBI};
for the current purpose, we have to choose $H(z)=(1+z)^{k+1}$,
$f(z)=z/(1+z)^{\mu +1}$ there). This yields
\begin{align} \notag
\LL{(0,0)\to (\mu n+k,n)\mid x\ge \mu y}&=\frac {1} {n}\coef{z^{-1}}
(k+1) \,(1+z)^k\frac {(1+z)^{n(\mu +1)}} {z^n}\\
\notag
&\hskip-2cm = \frac {k+1} {n}\binom {\mu n+k +n} {n-1}\\
\notag
&\hskip-2cm = \frac {k+1} {\mu n+k +n+1} \binom {\mu n+k+n+1}{n},
\end{align}
which turns into \eqref{CKe1.8} once we replace $\mu n+k$ by $c$ and $n$
by $d$.
\bigskip

In particular, for $\mu=1$ the generating function $F_0(z)$ can be
explicitly evaluated from solving the quadratic equation
\eqref{CKe1.9}. 
In the case, where the paths return to the boundary $x=y$, i.e., 
where $(c,d)=(n,n)$,
this gives the familiar generating function for the Catalan numbers
(compare with the second paragraph after Corollary~\ref{CKT1.2})
\begin{equation} \label{CKe1.13}\sum _{n\ge0} ^{} C_n z^n= \frac
{1-\sqrt {1-4z}} {2z}. 
\end{equation}

More generally, if $\mu$ is kept generic and $(c,d)=(\mu n,n)$
(that is, we consider again paths which return to the boundary),
then the formula on the right-hand side of \eqref{CKe1.8} becomes
$\frac {1} {\mu n+1}\binom
{(\mu+1)n} {n}  .$ These numbers are now commonly called
\index{Fu\ss--Catalan number}\index{number, Fu\ss--Catalan}%
\index{Fu\ss, Nikolaus}\index{Catalan, Eug\`ene Charles}%
{\em Fu\ss--Catalan numbers}, cf.\ \cite[pp.~59--60]{ArmDAA}
for more information on their significance and historical
remarks.

\medskip
So far we only counted paths bounded by $x=\mu y$ where the starting
point lies on the boundary. If we drop this latter assumption
and now want to 
enumerate all paths from $(a,b)$ to $(c,d)$ staying weakly below $x=\mu y$,
there is still an answer, although
only in terms of a sum. In fact, we can offer two different
expressions. Which of these two is preferable depends on the
particular situation, to be more precise, on which of the numbers
$({a/\mu}-b)$ or $(d-{a/\mu})$ being larger (see
Figure~\ref{CKF1.5} for the pictorial significance of these numbers). 
While the
proof for the first expression is rather straightforward, the
proof for the second expression is more difficult.
The result below was first found by \index{Korolyuk, V. S.}Korolyuk
\cite{KoroAA}. It is a special case of an even more general
result of \index{Niederhausen, Heinrich}Niederhausen \cite[Sec.~2.2]{NiedAC} 
on the enumeration of
simple paths with piecewise linear boundaries, which we will discuss
in Section~\ref{CKs1.3}. 

\begin{figure} 
$$
\Diagonale(-2,0)8
\DickPunkt(0,0)
\DickPunkt(5,5)
\Label\o{(c,d)}(5,5)
\Label\u{(a,b)}(0,0)
\SPfad(0,0),22\endSPfad
\SPfad(0,2),11111\endSPfad
\SPfad(5,0),22222\endSPfad
\SPfad(0,0),11111\endSPfad
\Label\ro{\,d-{a/\mu}}(6.5,3)
\Label\r{\,{a/\mu}-b}(6.5,1)
\Label\ru{\hskip-20pt x=\mu y}(7,8)
\hskip5\Einheit
\raise0.8\Einheit
\hbox to0pt{$\left.\vrule height1.2\Einheit depth0pt width0pt\right\}$\hss}
\raise3.3\Einheit
\hbox to0pt{$\left.\vrule height1.7\Einheit depth0pt width0pt\right\}$\hss}
$$
\caption{}
\label{CKF1.5}
\end{figure}

\begin{theorem} \label{CKT1.6}%
Let $\mu$ be a non-negative integer, $a\ge \mu b$ and $c\ge\mu d$. The
number of all lattice paths from $(a,b)$ to $(c,d)$ staying weakly below
$x=\mu y$ is given by
\begin{multline} \label{CKe1.15}
\vv{\LL{(a,b)\to (c,d)\mid x\ge\mu y}} 
=\binom {c+d-a-b}{c-a}\\
-\sum _{i=\floor {a/\mu}+1} ^{d}\binom {i(\mu+1)-a-b-1}{i-b} 
\frac {c-\mu d+1} {c+d-i(\mu+1)+1}\binom {c+d-i(\mu+1)+1} {d-i},\\
\end{multline}
and also
\begin{multline} \label{CKe1.16}
\vv{\LL{(a,b)\to (c,d)\mid x\ge\mu y}}
=\sum _{i=0} ^{\floor{a/\mu}-b}
(-1)^{i} \binom {a-\mu (b+i)}i \\
\times\frac {c-\mu d +1} {c+d-(\mu+1)(b+i)+1}
\binom {c+d-(\mu+1)(b+i)+1} {d-b-i}.
\end{multline}
\end{theorem}
\noindent
{\bf Proof of \eqref{CKe1.15}}\indent
The number of paths in question equals the number of all paths from
$(a,b)$ to $(c,d)$ minus those paths which cross $x=\mu y$. To count
the latter observe that any path crossing $x=\mu y$ must meet the line
$x=\mu y-1$, and for the last time in some point $(\mu i-1,i)$ where
$\floor {a/\mu}+1\le i\le c$. Fix such an $i$, then the number of all
these paths is
$$\vv{\LL{(a,b)\to (\mu i-1,i)}}\cdot \vv{\LL{(\mu i,i)\to (c,d)\mid x\ge
\mu y}}.$$
We already know the first number due to \eqref{CKe1.1},
and we also know the second number due to \eqref{CKe1.8},
since a shift in direction $(-\mu i,-i)$ shows that the second number
equals $\vv{\LL{(0,0)\to (c-\mu i, d-i)\mid x\ge \mu y}}$.

\bigskip
\noindent
{\bf Proof of \eqref{CKe1.16}}\indent
This is the special case of Theorem~\ref{CKthm:piece} where 
$m=2$, $\mu_1=\nu_1=0$, $y_1=\fl{a/\mu}-b$, $\mu_2=\mu$, $\nu_2=\mu b-a$.

For a different, direct proof, 
in the sum in \eqref{CKe1.15} replace the index $i$
by $i+b$; the new index then ranges from $\fl{a/\mu}-b+1$ to
$d-b$; extend the sum to all $i$ between $0$ and $d-b$, thereby
adding a partial sum where $i$ ranges from $0$ to $\fl{a/^mu}-b$;
the former sum can be evaluated by means of a convolution
formula of the Hagen--Rothe type (cf.\ \cite[Eq.~(11)]{GOulAA}), and the result is
the binomial coefficient $\binom {c+d-a-b}{c-a}$.
\bigskip

Enumeration of lattice paths in the presence of several linear boundaries
can in the best cases be solved by an iterated application of the reflection principle;
see Section~\ref{CKsec:refl} for the most general situation where the reflection
principle applies. But, if it does not apply (which, in a random case, will
certainly be so), then the enumeration problem will be very challenging.
Usually, one cannot expect to find a useful exact formula
(but see Section~\ref{CKs1.6}), and will instead
investigate asymptotic behaviours. This is still quite challenging.
The reader is referred to \cite{BrORAA,JRPRAB,JRPRAA} for work in this direction.

\section{Simple paths with linear boundaries with rational slope, II}
 \label{CKs1.4}

In Section~\ref{CKsec:rat} we considered lattice paths bounded by a line
$x=\mu y$, with $\mu$ a non-negative integer. 
Now we want to consider a more general linear boundary of the form
$\nu x=\mu y$, where $\nu,\mu$ are non-negative integers. We describe
a generating function approach, due to \index{Sato, Masako}Sato \cite{SatoAA}, which
works for a large class of cases. Alternative solutions, which work
in all cases, in the form of a determinant, can be given as a special
case of a set of general boundaries. These are discussed in
Section~\ref{CKs1.6}, see in particular Theorem~\ref{CKT1.15}.

The problem that we want to attack here is to enumerate all lattice
paths from an arbitrary starting point to an arbitrary end point 
staying weakly below the line $\nu x=\mu y$, where $\nu$ and $\mu$
are positive integers.
A simple shift of the plane shows that this is equivalent to
enumerating paths from the origin to an arbitrary end point 
staying weakly below $\nu x= \mu
y-\rho$, for an appropriate $\rho$. Without loss of generality we
may assume in the sequel that $\nu<\mu$.
For the approach of Sato, this is the more convenient
formulation of the problem. The idea is to introduce $\nu\times\nu$
matrices which contain the path numbers that we are looking for. More
precisely, define the $\nu\times\nu$ matrix $W(z;c,\rh)$ by
\begin{equation} \label{CKe1.20a}
W(z;c,\rh):=\big(w(z;c+g,\rh+h)\big)_{0\le g,h\le \nu-1},
\end{equation}
where 
\begin{equation} \label{CKe1.21b}
w(z;c,\rh)=\sum _{\mu n+c\equiv\rh\ (\text{mod }\nu)} ^{}w(n;c,\rh)\,z^n,
\end{equation}
with
\begin{multline} \label{CKe1.21c}
w(n;c,\rh)=\begin{cases} \vv{\LL{(0,0)\to(\frac {\mu n+c-\rh}
{\nu},n)\mid \nu x\ge\mu y-\rh}},&\mu n+c\equiv \rh\ (\text{mod }\nu)\\
&\text {and }\mu n+c\ge\rh,\\
\binom {n+\frac {\mu n+c-\rh} {\nu}}{n},&\mu n+c\equiv \rh\ (\text{mod }\nu)\\
&\text {and }\mu n+c<\rh.
\end{cases}\\
\end{multline}
So, what the matrix $W(z;c,\rh)$ contains is generating functions
of the path numbers $w(n;c,\rh)$ 
that we want to know. The definition of $w(n;c,\rh)$ for $\mu n+c<\rh$
(in which case there cannot be any paths from $(0,0)$ to $(\frac {\mu n+c-\rh}
{\nu},n)$) is just for technical convenience. Basically, the matrix
$W(z;c,\rh)$ is
\begin{multline} \label{CKe1.22}
W(z;c,\rh)=\bigg(\sum _{\mu n+c+g-h\equiv \rh\ (\text{mod } \nu)} ^{}
\left\vert\LL{(0,0)\to\left(\frac {\mu n+c+g-\rh-h}
{\nu},n\right)\mid\right. \\
\nu x\ge\mu y-\rh-h}\Big\vert\, z^n\bigg)_{0\le g,h\le \nu-1}.
\end{multline}

The following theorem of \index{Sato, Masako}Sato \cite[Theorem~1]{SatoAA} tells us how to
compute\break $W(z;c,\rh)$.
\begin{theorem} \label{CKT1.7}%
Let
\begin{equation} \label{CKe1.23f}
M=\((-1)^{\nu-h-1} 
s_{(c+g+1,1^{\nu-h-1})}\big(u_0(z),\dots,u_{\nu-1}(z)\big)\)_{0\le
g,h\le \nu-1},
\end{equation}
where
\begin{multline} \label{CKe1.23g}
s_{(\al,1^\be)}(u_0,\dots,u_{\nu-1})\\
=\sum _{\nu-1\ge 
i_\al\ge i_{\al-1}\ge\dots\ge i_1<j_1<\dots<j_\be\le \nu-1} ^{}
u_{i_\al}(z)u_{i_{\al-1}}(z)\cdots u_{i_1}(z)u_{j_1}(z)\cdots u_{j_\be}(z),
\end{multline}
$u_l(z)$ being defined by
\begin{multline} \label{CKe1.23e}
u_l(z)=e^{(2\pi il/\nu)}\sum _{n\ge0} ^{}\frac {1}
{1+(\nu+\mu)n}\binom {\frac {1} {\nu}+(1+\frac {\mu}
{\nu})n}n\(ze^{2\pi il\mu/\nu}\)^n,\\
l=0,1,\dots,\nu-1.
\end{multline}
Furthermore, let
\begin{equation} \label{CKe1.23b}
\Phi(z;\rh)=\Bigg(\underset{\mu l\le \rh-g+h}{\sum _{\mu l\equiv
\rh-g+h\ (\text{\em mod }\nu)}
^{}}(-1)^l\binom {(\rh-g+h-\mu l)/\nu}lz^l\Bigg)_{0\le g,h\le \nu-1}.
\end{equation}
Then, 
for any non-negative integers $c$, $\rh$, $\nu$, $\mu$, with $\nu<\mu$, we have
\begin{equation} \label{CKe1.23a}
W(z;c,\rh)=M(z;c,\rh)\,\Phi(z;\rh).
\end{equation}
\end{theorem}
\begin{note}\em Note that $s_{(\al,1^\be)}(u_0,\dots,u_{\nu-1})$
is a {\em Schur function\/} of {\em hook shape} (cf.\ \cite[Ch.~I,
Sec.~3, Ex.~9]{MacdAC}).
\end{note}
It might be useful to 
discuss an example, in order to illustrate what this is all about.

\begin{example}%
\em
We take $\nu=2$, $\mu=3$. 
So, by \eqref{CKe1.21c}, the quantity $w(n;c,\rh)$ represents the
number of all lattice paths from $(0,0)$ to $((3n+c-\rh)/2,n)$ which
stay weakly below the line $2x= 3y-\rh$, where $c\equiv 3n-\rh\ (\text{mod }2)$, i.e.,
$c\equiv n+\rh\ (\text{mod }2)$.

By definition \eqref{CKe1.20a}, we have
\begin{align} \notag
W(z;c,\rh)&=\big(w(z;c+g,\rh+h)\big)_{0\le g,h\le 1}\\
&=\begin{pmatrix} w(z;c,\rh)&w(z;c,\rh+1)\\
w(z;c+1,\rh)&w(z;c+1,\rh+1)\end{pmatrix}.\notag
\end{align}
Using \eqref{CKe1.23a}, this can be written as 
$$W(z;c,\rh)=M(z;c,2)\,\Phi(z;\rh),$$
where 
\begin{multline*}
\Phi(z;\rh)\\=
\begin{pmatrix} 
\underset{\mu l\le \rh}{\sum _{\mu l\equiv \rh\ (\text{mod }\nu)}
^{}}(-1)^l\binom {(\rh-\mu l)/\nu}lz^l&
\underset{\mu l\le \rh+1}{\sum _{\mu l\equiv \rh+1\ (\text{mod }\nu)}
^{}}(-1)^l\binom {(\rh+1-\mu l)/\nu}lz^l\\
\underset{\mu l\le \rh-1}{\sum _{\mu l\equiv \rh-1\ (\text{mod }\nu)}
^{}}(-1)^l\binom {(\rh-1-\mu l)/\nu}lz^l&
\underset{\mu l\le \rh}{\sum _{\mu l\equiv \rh\ (\text{mod }\nu)}
^{}}(-1)^l\binom {(\rh-\mu l)/\nu}lz^l
\end{pmatrix}
\end{multline*}
by \eqref{CKe1.23b},
and 
$$
M(z;c,2)=\begin{pmatrix} 
-s_{(c+1,1)}\big(u_0(z),u_{1}(z)\big)&
s_{(c+1)}\big(u_0(z),u_{1}(z)\big)\\
-s_{(c+2,1)}\big(u_0(z),u_{1}(z)\big)&
s_{(c+2)}\big(u_0(z),u_{1}(z)\big)
\end{pmatrix}$$
by \eqref{CKe1.23f}, with $s_{(\al,1^\be)}(u_0(z),u_1(z))$ being defined
in \eqref{CKe1.23g}, and
$$u_l(z)=(-1)^l\sum _{n\ge0} ^{}\frac {(-1)^{ln}} {1+5n}\binom {\frac {1}
{2}+\frac {5} {2}n}nz^n,$$
as given in \eqref{CKe1.23e}.

So, in particular, in case that $c=\rh=0$ the matrix $\Phi(z;0)$ is
the $2\times 2$ identity matrix, and so we have
\begin{multline} \notag
W(z;0,0)=\begin{pmatrix} w(z;0,0)&w(z;0,1)\\
w(z;1,0)&w(z;1,1)\end{pmatrix}\\
=M(z;0,2)=
\begin{pmatrix} -u_0(z)u_1(z)&u_0(z)+u_1(z)\\
-u_0(z)u_1(z)(u_0(z)+u_1(z))&u_0^2(z)+u_0(z)u_1(z)+u_1^2(z)\end{pmatrix}.
\end{multline}
Whence, for even $n$ the number of all lattice paths from $(0,0)$ to
$(3n/2,n)$ which stay weakly below the line $2x\ge 3y$ equals
\begin{multline} \notag
\vv{\LL{(0,0)\to(\tfrac {3n} {2},n)\mid 2x\ge 3y}}=\coef{z^n}w(z;0,0)\\
=\sum _{l=0} ^{n}(-1)^l\frac {1} {1+5l}\binom {\frac {1} {2}+\frac {5}
{2}l}l\cdot \frac {1} {1+5(n-l)}\binom {\frac {1} {2}+\frac {5}
{2}(n-l)}{n-l},
\end{multline}
and for odd $n$ the number of all lattice paths from $(0,0)$ to
$((3n-1)/2,n)$ which stay weakly below the line $2x\ge 3y-1$ equals
$$
\vv{\LL{(0,0)\to(\tfrac {3n-1} {2},n)\mid 2x\ge
3y-1}}=\coef{z^n}w(z;0,1)=\frac {2} {1+5n}\binom {\frac {1} {2}+\frac
{5} {2}n}n.
$$
\end{example}

\index{Sato, Masako}Sato \cite{SatoAA} also derived a result of similar type
for two parallel linear boundaries with rational slope. To be precise,
we want to enumerate all lattice
paths from an arbitrary starting point to an arbitrary end point 
staying weakly below a given line $\nu x=\mu y-\rh$ and above 
another given line $\nu x=\mu y+\si$, where $\mu,\nu,\rh,\si$ are
non-negative integers. Again, without loss of generality we may assume
that $\nu<\mu$ and that the starting point is the origin.

Following the approach we have taken earlier, 
we define the $\nu\times\nu$ matrix\break $T(z;c,\rh,\si)$ by
\begin{equation} \label{CKe1.28a}
T(z;c,\rh,\si):=\big(t(z;c+g,\rh+h,\si-h)\big)_{0\le g,h\le \nu-1},
\end{equation}
where 
\begin{equation} \label{CKe1.28b}
t(z;c,\rh,\si)=\sum _{\mu n+c\equiv\rh\ (\text{mod }\nu)} ^{}t(n;c,\rh,\si)\,z^n,
\end{equation}
with
\begin{multline} \label{CKe1.28c}
t(n;c,\rh,\si)=\begin{cases} \vv{\LL{(0,0)\to(\frac {\mu n+c-\rh}
{\nu},n)\mid \mu y+\si\ge\nu x\ge\mu y-\rh}},\kern-2cm&\\
&\mu n+c\equiv \rh\ (\text{mod }\nu)\\
&\text {and }\mu n+c\ge\rh,\\
\binom {n+\frac {\mu n+c-\rh} {\nu}}{n},&\mu n+c\equiv \rh\ (\text{mod }\nu)\\
&\text {and }\mu n+c<\rh.
\end{cases}\\
\end{multline}
Similarly to the one boundary case, 
what the matrix $T(z;c,\rh,\si)$ contains is generating functions
of the path numbers $t(n;c,\rh,\si)$ 
that we want to compute. The definition of $t(n;c,\rh,\si)$ for $\mu n+c<\rh$
(in which case there cannot be any paths from $(0,0)$ to $(\frac {\mu n+c-\rh}
{\nu},n)$) is just for technical convenience. Basically, the matrix
$T(z;c,\rh,\si)$ is
\begin{multline} \label{CKe1.28d}
T(z;c,\rh,\si)=\bigg(\sum _{\mu n+c+g-h\equiv \rh\ (\text{mod } \nu)} ^{}
\left\vert\LL{(0,0)\to\left(\frac {\mu n+c+g-\rh-h}
{\nu},n\right)\mid\right. \\
\mu y+\si-h\ge\nu x\ge\mu y-\rh-g}\Big\vert\, z^n\bigg)_{0\le g,h\le \nu-1}.
\end{multline}

The following theorem of \index{Sato, Masako}Sato
\cite[Theorem~4]{SatoAA} tells us how to
compute\break $T(z;c,\rh,\si)$.

\begin{theorem} \label{CKT1.8}%
For any non-negative integers $c$, $\rh$, $\si$, 
$\nu$, $\mu$, with $\nu<\mu$, we have
\begin{equation} \label{CKe1.28e}
T(z;c,\rh,\si)=\Phi(z;\rh+\si+1-c-\mu)\, \Phi^{-1}(z;\rh+\si+1)\,
\Phi(z;\rh),
\end{equation}
where $\Phi(z;\th)$ is given by \eqref{CKe1.23b}.

\end{theorem}

\begin{example}%
\em
As an illustration, let us again consider the case
$\nu=2$, $\mu=3$. By \eqref{CKe1.28c}, the quantity $t(n;c,\rh,\si)$ 
represents the
number of all lattice paths from $(0,0)$ to $((3n+c-\rh)/2,n)$ which
stay weakly below the line $2x= 3y-\rh$ and above the line $2x=3y+\si$, 
where $c\equiv 3n-\rh\ (\text{mod }2)$, i.e.,
$c\equiv n+\rh\ (\text{mod }2)$.

By definition \eqref{CKe1.28a}, we have
\begin{align} \notag
T(z;c,\rh,\si)&=\big(t(z;c+g,\rh+h,\si-h)\big)_{0\le g,h\le 1}\\
&=\begin{pmatrix} t(z;c,\rh,\si)&t(z;c,\rh+1,\si-1)\\
t(z;c+1,\rh,\si)&t(z;c+1,\rh+1,\si-1)\end{pmatrix}.\notag
\end{align}
By Theorem~\ref{CKT1.8}, this can be written as
\begin{align} \notag
T(z;c,\rh,\si)=&\begin{pmatrix} \phi(z;\rh+\si-c-1)&
\phi(z;\rh+\si-c)\\
\phi(z;\rh+\si-c-2)&\phi(z;\rh+\si-c-1)\end{pmatrix}\\
\notag
&\times \begin{pmatrix} \phi(z;\rh+\si+1)&
\phi(z;\rh+\si+2)\\
\phi(z;\rh+\si)&\phi(z;\rh+\si+1)\end{pmatrix}^{-1}\\
&\times\begin{pmatrix} \phi(z;\rh)&
\phi(z;\rh+1)\\
\phi(z;\rh-1)&\phi(z;\rh)\end{pmatrix},\label{CKe1.29a}
\end{align}
where
\begin{equation} \notag
\phi(z;a)=\underset{\mu l\le a}{\sum _{\mu l\equiv
a\ (\text{mod }\nu)}
^{}}(-1)^l\binom {(a-\mu l)/\nu}lz^l.
\end{equation}

In particular, if $c=\si$ and $\rh=0$, so that we are 
e.g\@. interested in
the number of all lattice paths from $(0,0)$ to $((3n+c)/2,n)$ which
stay weakly below the line $2x= 3y$ and above the line $2x=3y+c$ (the reader
should observe that this means that the starting point is on the
first line whereas the end point is on the second line), then our
formula \eqref{CKe1.29a} reduces to
\begin{align} \notag
T(z;c,\rh)&=\begin{pmatrix} t(z;c,0,c)&t(z;c,1,c-1)\\
t(z;c+1,0,c)&t(z;c+1,1,c-1)\end{pmatrix}\\
\notag
&=\frac {1} {\phi^2(z;c+1)-\phi(z;c)\phi(z;c+2)}
\begin{pmatrix} \phi(z;c)&\phi(z;c+1)\\0&0\end{pmatrix}.
\end{align}
Thus we obtain
\begin{align} \notag
t(z;c,0,c)&=\sum _{ n\equiv c\ (\text{mod }2)} ^{}
\vv{\LL{(0,0)\to\left(\tfrac {3 n+c}
{2},n\right)\mid 3 y+c\ge2 x\ge3 y}}\,z^n\\
&=\frac {\phi(z;c)} {\phi^2(z;c+1)-\phi(z;c)\phi(z;c+2)},
\end{align}
with 
$$\phi(z;a)=\underset{3 l\le a}{\sum _{3 l\equiv
a\ (\text{mod }2)}
^{}}(-1)^l\binom {(a-3 l)/2}lz^l.
$$
\end{example}

\section{Simple paths with a piecewise linear boundary} \label{CKs1.3}

In this section we generalize the one-sided linear boundary
results in Corollary~\ref{CKT1.2} and Theorems~\ref{CKT1.4}, \ref{CKT1.6} 
to piecewise
linear boundaries. To be more precise, we want to count lattice paths
from the origin $(0,0)$ to $(c,d)$ staying weakly below the line segments
\begin{multline} \label{CKeq:R}
\{(x,y):x=\mu _1y+\nu_1,0=y_0\le y\le y_1\},\
\{(x,y):x=\mu _2y+\nu_2,y_1< y\le y_2\},\\ \dots,\
\{(x,y):x=\mu _my+\nu_m,y_{m-1}< y\le y_m=d\},
\end{multline}
for some sequence
$0=y_0<y_1<\dots<y_m=d$ of non-negative integers, non-negative integers
$\mu_1,\mu_2,\dots,\mu_m$, and integers $\nu_1,\nu_2,\dots,\nu_m$.
Let us denote this piecewise linear restriction by $R_m$.
See Figure~\ref{CKF1.5B} for an example.
By an
iteration argument it will be seen that the solution to this problem
can be given in form of an $m$-fold sum. 

The result below is due to 
\index{Niederhausen, Heinrich}Niederhausen \cite[Sec.~2.2 in connection
with (2.4) and (2.7)]{NiedAC}, but see also \cite{MoNaAA}.
In order to understand the statement below, it is important to observe
that the number of paths which we want to determine is a polynomial
in $c$, while keeping all other variables fixed. 
We shall not provide a detailed argument here but, instead, refer to
\cite[Sec.~2.2]{NiedAC}.
For convenience, let us denote this polynomial by $L_{R_m,d}(c)$.

\begin{theorem} \label{CKthm:piece}%
The number of lattice paths from $(0,0)$ to $(c,d)$ staying\break weakly
below the piecewise linear boundary $R_m$ given in \eqref{CKeq:R} is
equal to
\begin{multline} 
\vv{\LL{(0,0)\to (c,d)\mid R_m}}
=\sum _{i=0} ^{y_{m-1}}
L_{R_{m-1},i}(\mu_m i+\nu_m-1)\\
\cdot
\frac {c-\mu_md-\nu_m+1} {c+d-i(\mu_m+1)-\nu_m+1}
\binom{c+d-i(\mu_m+1)-\nu_m+1}{d-i}.
\label{CKe1.19}
\end{multline}
\end{theorem}

\begin{remark}\em
Clearly, we may now apply Theorem~\ref{CKthm:piece} to
$L_{R_{m-1},i}(\mu_m i+\nu_m-1)$
so that, iteratively, we obtain an $m$-fold sum. (In the last step,
one applies \eqref{CKe1.16}.)
\end{remark}

\begin{figure}
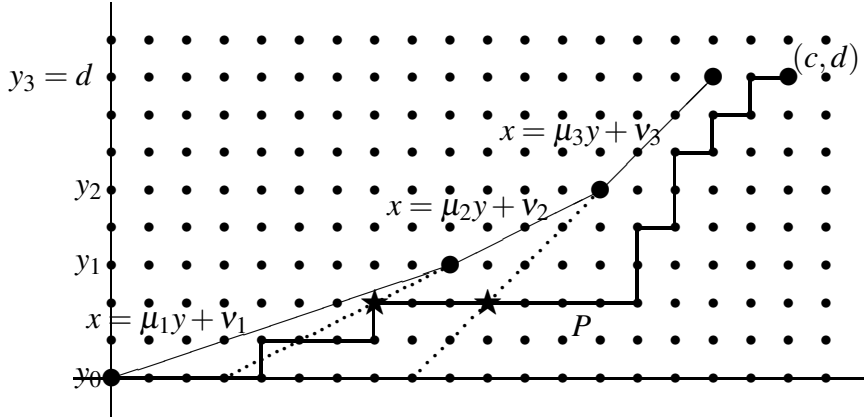
 
$$
\Einheit.5cm
\def\SternPunkt(#1,#2){\unskip
  \raise#2 \Einheit\hbox to0pt{\hskip#1 \Einheit
          \raise-3pt\hbox to0pt{\hss$\bigstar$\hss}\hss}}
\def\ZweiSchritt{\leavevmode\raise-.4pt\hbox to0pt{%
  \hbox to0pt{\hss.\hss}\hskip.2\Einheit
  \raise.1\Einheit\hbox to0pt{\hss.\hss}\hskip.2\Einheit
  \raise.2\Einheit\hbox to0pt{\hss.\hss}\hskip.2\Einheit
  \raise.3\Einheit\hbox to0pt{\hss.\hss}\hskip.2\Einheit
  \raise.4\Einheit\hbox to0pt{\hss.\hss}\hskip.2\Einheit
  \raise.5\Einheit\hbox to0pt{\hss.\hss}\hskip.2\Einheit
  \raise.6\Einheit\hbox to0pt{\hss.\hss}\hskip.2\Einheit
  \raise.7\Einheit\hbox to0pt{\hss.\hss}\hskip.2\Einheit
  \raise.8\Einheit\hbox to0pt{\hss.\hss}\hskip.2\Einheit
  \raise.9\Einheit\hbox to0pt{\hss.\hss}\hss}}
\Gitter(20,10)(0,0)
\Koordinatenachsen(20,10)(0,0)
\raise0\Einheit\hbox to0pt{\hskip0 \Einheit
         \unitlength\Einheit
         \line(3,1){9}\hss}
\raise3\Einheit\hbox to0pt{\hskip9 \Einheit
         \unitlength\Einheit
         \line(2,1){4}\hss}
\raise5\Einheit\hbox to0pt{\hskip13 \Einheit
         \unitlength\Einheit
         \line(1,1){3}\hss}
\raise0 \Einheit\hbox to0pt{\hskip3 \Einheit
         \ZweiSchritt\hss}
\raise1 \Einheit\hbox to0pt{\hskip5 \Einheit
         \ZweiSchritt\hss}
\raise2 \Einheit\hbox to0pt{\hskip7 \Einheit
         \ZweiSchritt\hss}
\SPfad(8,0),33333\endSPfad
\Pfad(0,0),11112111211111112212212121\endPfad
\DickPunkt(0,0)
\DickPunkt(18,8)
\DickPunkt(9,3)
\DickPunkt(13,5)
\DickPunkt(16,8)
\SternPunkt(7,2)
\SternPunkt(10,2)
\Label\ru{P}(12,2)
\Label\lo{\hskip0pt x=\mu_1 y+\nu_1}(2,1)
\Label\lo{\hskip0pt x=\mu_2 y+\nu_2}(10,4)
\Label\lo{\hskip0pt x=\mu_3 y+\nu_3}(13,6)
\Label\l{y_0}(0,0)
\Label\l{y_1}(0,3)
\Label\l{y_2}(0,5)
\Label\l{y_3=d}(-1,8)
\Label\o{(c,d)}(19,8)
\hskip9cm
$$
\caption{A piecewise linear boundary}
\label{CKF1.5B}
\end{figure}

\bigskip
\noindent
{\bf Idea of proof of Theorem~\ref{CKthm:piece}}\indent
To begin with, let us assume that the piecewise linear boundary be
convex. See Figure~\ref{CKF1.5B} for an example.
Evidently, any path from $(0,0)$ to $(c,d)$
has to touch $x=\mu_m y+\nu_m$ for the last time, say in $(\mu_m
i+\nu_m,i)$. In Figure~\ref{CKF1.5B} we have $m=3$, the last
touching point of the path $P$ with $x=\mu_m y+\nu_m$ is $(10,2)$, 
it is marked by a star.
Then we utilize the same idea which led to the formula \eqref{CKe1.15}  
to obtain the number in question being equal to
\begin{align} \notag
&\vv{\LL{(0,0)\to (c,d)\mid R_m}}\\
\notag
&\hskip1cm=
\sum _{i=0} ^{y_{m-1}}\vv{\LL{(0,0)\to (\mu_m i+\nu_m-1,i)\mid R_m}}\\
\notag
&\hskip3cm
\cdot
\vv{\LL{(\mu_m i+\nu_m,i)\to (c,d)\mid x\ge \mu_m y+\nu_m}}
\\\notag
&\hskip1cm
=\sum _{i=0} ^{y_{m-1}}\vv{\LL{(0,0)\to (\mu_m i+\nu_m-1,i)\mid R_{m-1}}}\\
\notag
&\hskip3cm
\cdot
\frac {c-\mu_md-\nu_m+1} {c+d-i(\mu_m+1)-\nu_m+1}
\binom{c+d-i(\mu_m+1)-\nu_m+1}{d-i},
\notag
\end{align}
by virtue of Theorem~\ref{CKT1.4}. In the summand, we were allowed to
replace $R_m$ by $R_{m-1}$ since the summation ends at $i=y_{m-1}$,
and thus the $m$-th segment does not come into play.

Evidently, if the piecewise linear boundary should not be convex,
then this argument breaks down.
However, 
\index{Niederhausen, Heinrich}Niederhausen 
shows in \cite[Sec.~2.2]{NiedAC}, using
the polynomiality of the path numbers (and some results from
umbral calculus), that the above formula continues to hold in
that case also, even if the substitution of $\mu_m i+\nu_m-1$ in
the argument of the polynomial $L_{R_{m-1},i}(\,.\,)$ has no
combinatorial meaning anymore. 
\bigskip

\section{Simple paths with general boundaries} \label{CKs1.6}
The most general problem to encounter is to count paths in a region
that is bounded by nonlinear upper and lower boundaries as exemplified
in Figure~\ref{CKF1.6}.
\begin{figure} 
$$
\Gitter(8,9)(0,0)
\Koordinatenachsen(0,9)(1,-1)
\Pfad(0,0),22111221221221\endPfad
\PfadDicke{.5pt}
\SPfad(0,3),12212212111\endSPfad
\SPfad(0,0),12112122211222\endSPfad
\Label\ro{a_1}(0,3)
\Label\ro{a_2}(1,5)
\Label\ro{a_3}(2,7)
\Label\ro{a_4}(3,8)
\Label\ro{\hskip18pt(n,a_n)}(6,8)
\Label\lu{(0,b_1)\hskip14pt}(0,0)
\Label\ru{b_1}(0,0)
\Label\ru{b_2}(1,1)
\Label\ru{b_3}(2,1)
\Label\ru{b_4}(3,2)
\Label\ro{\raise-10pt\hbox{$y_1$}}(0,2)
\Label\ro{\raise-10pt\hbox{$y_2$}}(1,2)
\Label\ro{\raise-10pt\hbox{$y_3$}}(2,2)
\Label\ro{\raise-10pt\hbox{$y_4$}}(3,4)
\DickPunkt(0,0)
\DickPunkt(6,8)
\hskip4cm
$$
\caption{}
\label{CKF1.6}
\end{figure}

To have a convenient notation, let $a_1\le a_2\le \dots\le a_n$ and
$b_1\le b_2\le \dots\le b_n$ be integers with $a_i\ge b_i$. We
abbreviate $\mathbf a=(a_1,a_2,\dots,a_n)$ and $\mathbf
b=(b_1,b_2,\dots,b_n)$. By $\LL{(0,b_1)\to (n,a_n)\mid \mathbf a\ge \mathbf
y\ge \mathbf b}$ we denote the set of all lattice paths from $(0,b_1)$
to $(n,a_n)$ that satisfy the property that for all $i=1,2,\dots,n$
the height $y_i$ of the $i$-th horizontal step is in the interval
$[b_i,a_i]$. If we also write $\mathbf y(P)=(y_1,y_2,\dots,y_n)$ for the
sequence of heights of horizontal steps of a path $P$, then the notation just
introduced explains itself. Pictorially (see Figure~\ref{CKF1.6}), the
described restriction means that we consider paths in a ladder-shaped
region, the upper ladder being determined by $\mathbf a$, the lower
ladder being determined by $\mathbf b$.
See Figure~\ref{CKF1.7}, which displays an example 
with $n=6$, $\mathbf a=(3,5,7,8,8,8)$,
$\mathbf b=(0,1,1,2,5,5)$, $\mathbf y(P_0)=(2,2,2,4,6,8)$.

Originally, the result below was derived by 
\index{Kreweras, Germain}Kreweras \cite{KrewAB}
using recurrence relations, but 
the most conceptual and most elegant way to attack
this problem is by the method of non-intersecting lattice
paths; see Section~\ref{CKsec:nonint} and \cite{SulaAC}.

\begin{figure}
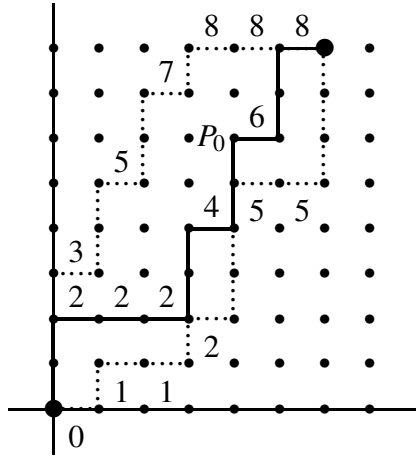
 
$$
{
\Gitter(8,9)(0,0)
\Koordinatenachsen(8,9)(0,0)
\Pfad(0,0),22111221221221\endPfad
\Label\l{P_0}(4,6)
\PfadDicke{.5pt}
\SPfad(0,3),12212212111\endSPfad
\SPfad(0,0),12112122211222\endSPfad
\Label\ro{3}(0,3)
\Label\ro{5}(1,5)
\Label\ro{7}(2,7)
\Label\ro{8}(3,8)
\Label\ro{8}(4,8)
\Label\ro{8}(5,8)
\Label\ru{0}(0,0)
\Label\ru{1}(1,1)
\Label\ru{1}(2,1)
\Label\ru{2}(3,2)
\Label\ru{5}(4,5)
\Label\ru{5}(5,5)
\Label\ro{2}(0,2)
\Label\ro{2}(1,2)
\Label\ro{2}(2,2)
\Label\ro{4}(3,4)
\Label\ro{6}(4,6)
\DickPunkt(0,0)
\DickPunkt(6,8)
\hskip5cm
}
$$
\caption{Lattice path with a general boundary}
\label{CKF1.7}
\end{figure}

\begin{theorem} \label{CKT1.15}%
Let $\mathbf a=(a_1,a_2,\dots,a_n)$ and $\mathbf b=(b_1,b_2,\dots,b_n)$ be
integer sequences with $a_1\le a_2\le \dots\le a_n$, $b_1\le b_2\le
\dots\le b_n$, and $a_i\ge b_i$, $i=1,2,\dots,n$. The number of all
paths from $(0,b_1)$ to $(n,a_n)$ satisfying the property that 
for all $i=1,2,\dots,n$ the height of the $i$-th horizontal step is
between $b_i$ and $a_i$ is given by
\begin{equation} \label{CKe1.30}
\vv{\LL{(0,b_1)\to (n,a_n)\mid \mathbf a\ge \mathbf y\ge \mathbf b}}=
\det_{1\le i,j\le n}\(\binom {a_i-b_j+1}{j-i+1}\). 
\end{equation}
\end{theorem}
\begin{proof}We apply Theorem~\ref{CKT2.4} with 
$\La=(1,1,\dots,1)$ and $\Mu=(0,0,\dots,0)$, both vectors
containing $n$ entries. This counts vectors
$(\pi_1,\pi_2,\dots,\pi_n)$ with $\pi_1<\pi_2<\dots<\pi_n$ with
a lower and an upper bound on each $\pi_i$. By replacing
$\pi_i$ by $\pi_i-i$, this counting problem is translated into
the counting problem we consider here, $\pi_i-i$ corresponding
to the height of the $i$-th horizontal step of a path. 
\end{proof}

Of course, with increasing $n$ this formula will become less
tractable. An alternative formula can be obtained by rotating the
whole picture by $90^\circ$ and applying formula \eqref{CKe1.30} to the
new situation. Now the size of the determinant is $a_n$, which is smaller 
than before if
$a_n<n$, i.e., if the difference between the $y$-coordinates of
end and starting point is less than the difference between the
respective $x$-coordinates. 

In some cases, a different type of
formula might be preferrable, which one may obtain by the so-called
{\em dummy path technique\/}\index{dummy path
technique}\index{path, dummy},
as proposed
in \index{Krattenthaler, Christian}Krattenthaler and \index{Mohanty,
Sri Gopal}Mohanty 
\cite{KrMoAA}. Again, it comes from
non-intersecting lattice paths. It is based on the following
observation (see \index{Stanley, Richard Peter}Stanley \cite[Ex.~2.7.2]{StanAP}). 

\begin{lemma} \label{CKT1.16}%
Let $C_1,C_2,\dots,C_n$ be pairwise distinct points in $\Z^2$. Then
the number of lattice paths from $(a,b)$ to $(c,d)$ which avoid
$C_1,C_2,\dots, C_n$ is given by
\begin{equation} \label{CKe1.31}
\det_{1\le i,j\le n+1}\big(\vv{L(A_j\to E_i)}\big) ,
\end{equation}
where $A_1=(a,b)$, $A_2=C_1$, \dots, $A_{n+1}=C_{n}$, $E_1=(c,d)$,
$E_2=C_1$, \dots, $E_{n+1}=C_n$.
\end{lemma}
\begin{proof}We reformulate our counting problem in that we
want to determine the number of families $(P_1,P_2,\dots,P_{n+1})$ of
non-intersecting lattice paths, where $P_1$ runs from $A_1=(a,b)$ to
$E_1=(c,d)$, and for $i=1,2,\dots,n$ the ``dummy path" $P_{i+1}$ runs
from $A_{i+1}=C_i$ to $E_{i+1}=C_i$. By Theorem~\ref{CKT2.3}, with
$G$ the directed graph with vertices $\Z^2$ and edges given by
horizontal and vertical unit steps in the positive direction,
all weights being $1$, $A_i$ and $E_i$ as above, this number 
equals the determinant in \eqref{CKe1.31}. 
\end{proof}

The idea now is that, given some (possibly two-sided) boundary,
one ``describes" this boundary by such ``dummy points" (paths)
and uses the above lemma to compute the number of paths which
avoid these, thus avoiding the boundary. In cases the boundary
can be ``described" by only very few ``dummy points", this may
lead to a useful formula.
Several formulae which appear in the literature are instances
of this idea (sometimes of minor variations), although it
may not be stated there that way; see \cite[Theorem~2 on p.~36]{MohaAE}
and \cite{KrMoAA} and the references given there.

\section{Elementary results on Motzkin and Schr\"oder paths}
\label{CKs3.1}

The subject of this section and the following three sections 
is lattice paths in $\Z^2$ which consist of
\index{up-step}{\em up-steps} $(1,1)$, \index{down-step}{\em down-steps} 
$(1,-1)$, and \index{level step}{\em level-steps} $(1,0)$ or
$(2,0)$, and which do not pass below the $x$-axis. If the only allowed
level-steps are unit steps $(1,0)$, then the corresponding paths are called
\index{path, Motzkin}\index{Motzkin path}\index{Motzkin, Theodore S.}{\em Motzkin paths}. 
If the only allowed
level-steps are double steps $(2,0)$, then the corresponding paths are called
\index{path, Schr\"oder}\index{Schr\"oder path}\index{Schr{\"o}der,
  Friedrich Wilhelm Karl Ernst}{\em Schr\"oder paths}. 
We call the special paths which consist of just up- and down-steps
(but contain no level-steps) 
\index{path, Catalan}\index{Catalan path}\index{Catalan, Eug\`ene Charles}{\em Catalan paths}. In the special case, where these paths start and end
on the $x$-axis, they are commonly called
\index{path, Dyck}\index{Dyck path}{\em Dyck paths}.

\begin{figure}
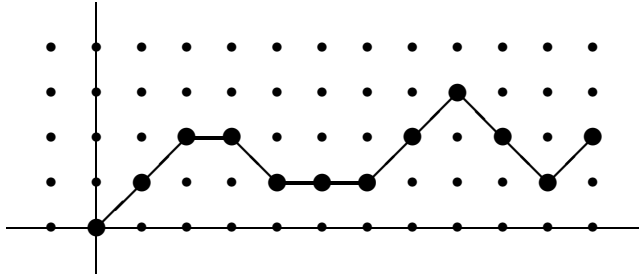
 
$$
\Gitter(12,5)(-1,0)
\Koordinatenachsen(12,5)(-1,0)
\Pfad(0,0),33141133443\endPfad
\DickPunkt(0,0)
\DickPunkt(1,1)
\DickPunkt(2,2)
\DickPunkt(3,2)
\DickPunkt(4,1)
\DickPunkt(5,1)
\DickPunkt(6,1)
\DickPunkt(7,2)
\DickPunkt(8,3)
\DickPunkt(9,2)
\DickPunkt(10,1)
\DickPunkt(11,2)
\hbox{\hskip6.5cm}
$$
\caption{A Motzkin path}
\label{CKF3.1}
\end{figure}

\begin{figure}
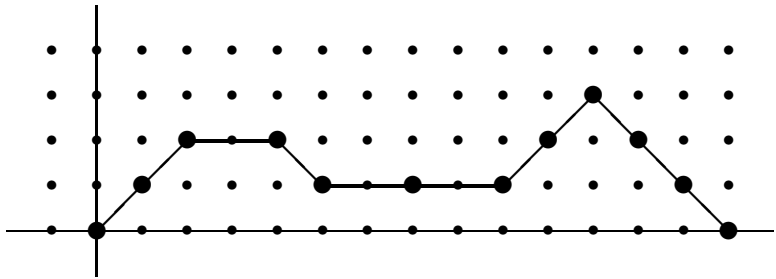
 
$$
\Gitter(15,5)(-1,0)
\Koordinatenachsen(15,5)(-1,0)
\Pfad(0,0),33114111133444\endPfad
\DickPunkt(0,0)
\DickPunkt(1,1)
\DickPunkt(2,2)
\DickPunkt(4,2)
\DickPunkt(5,1)
\DickPunkt(7,1)
\DickPunkt(9,1)
\DickPunkt(10,2)
\DickPunkt(11,3)
\DickPunkt(12,2)
\DickPunkt(13,1)
\DickPunkt(14,0)
\hbox{\hskip7.5cm}
$$
\caption{A Schr\"oder path}
\label{CKF3.2}
\end{figure}

\begin{figure}
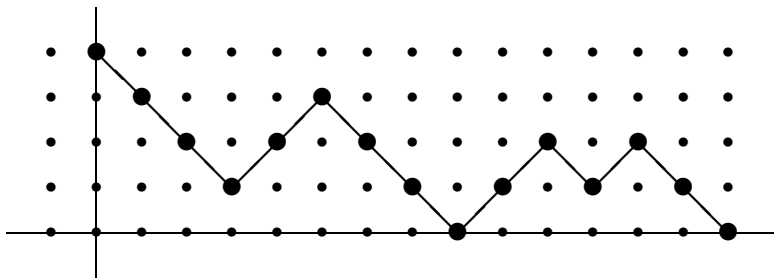
 
$$
\Gitter(15,5)(-1,0)
\Koordinatenachsen(15,5)(-1,0)
\Pfad(0,4),44433444334344\endPfad
\DickPunkt(0,4)
\DickPunkt(1,3)
\DickPunkt(2,2)
\DickPunkt(3,1)
\DickPunkt(4,2)
\DickPunkt(5,3)
\DickPunkt(6,2)
\DickPunkt(7,1)
\DickPunkt(8,0)
\DickPunkt(9,1)
\DickPunkt(10,2)
\DickPunkt(11,1)
\DickPunkt(12,2)
\DickPunkt(13,1)
\DickPunkt(14,0)
\hbox{\hskip6.5cm}
$$
\caption{A Catalan path}
\label{CKF3.C}
\end{figure}

Let $M=\{(1,1),(1,-1),(1,0)\}$ and 
$S=\{(1,1),(1,-1),(2,0)\}$, so that $M$ is the set of steps allowed
in Motzkin paths (see Figure~\ref{CKF3.1} for an example) 
and $S$ is the set of steps allowed
in Schr\"oder paths (see Figure~\ref{CKF3.2} for an example).  

A frequently used alternative way to view Schr\"oder paths is by
reflecting the picture with respect to the $x$-axis, rotating the
result by $45^\circ$, and finally scaling everything by a factor of
$1/\sqrt{2}$, so that the steps $(1,1),(1,-1),(2,0)$ 
are replaced by the steps $(1,0),(0,1),(1,1)$, in that order.
Figure~\ref{CKF3.3} shows the result of this translation when applied
to the Schr\"oder path in Figure~\ref{CKF3.2}. It translates Schr\"oder
paths into paths which consist of unit horizontal and vertical steps in
the positive direction and of upwards diagonal steps, and which stay 
weakly below the main diagonal $y=x$. Without the diagonal
restriction, the counting problem would be solved by the 
Delannoy numbers in \eqref{CKe1.1-5}.

\begin{figure}
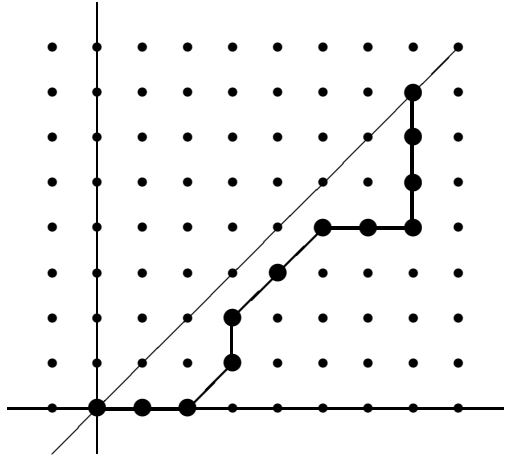
 
$$
\Gitter(9,9)(-1,0)
\Koordinatenachsen(9,9)(-1,0)
\Pfad(0,0),11323311222\endPfad
\DickPunkt(0,0)
\DickPunkt(1,0)
\DickPunkt(2,0)
\DickPunkt(3,1)
\DickPunkt(3,2)
\DickPunkt(4,3)
\DickPunkt(5,4)
\DickPunkt(6,4)
\DickPunkt(7,4)
\DickPunkt(7,5)
\DickPunkt(7,6)
\DickPunkt(7,7)
\thinlines
\Diagonale(-1,-1)9
\hbox{\hskip3.5cm}
$$
\caption{A Schr\"oder path rotated-reflected}
\label{CKF3.3}
\end{figure}

Nevertheless,
this translation, combined with Theorem~\ref{CKT1.1},
already tells us how to enumerate Motzkin and Schr\"oder paths with
given starting and end point.

\begin{theorem} \label{CKT3.1}%
Let $b\ge 0$ and $d\ge0$. The number of all paths from $(a,b)$ to
$(c,d)$ which consist of steps out of $M=\{(1,1),(1,-1),(1,0)\}$ and
do not pass below the $x$-axis 
\index{path, Motzkin}\index{Motzkin path}\index{Motzkin, Theodore S.}{\em (Motzkin paths)} is given by
\begin{multline} \label{CKe3.1}
\vv{\LL{(a,b)\to (c,d);M\mid y\ge0}}\\=
\sum _{k=0} ^{c-a}\binom {c-a}k\(\binom {c-a-k} {(c+d-k-a-b)/2} -
\binom {c-a-k}{(c+d-k-a+b+2)/2}\),
\end{multline}
where, by convention, a binomial coefficient is $0$ if its bottom
parameter is not an integer.

Furthermore, the number of all paths from $(a,b)$ to
$(c,d)$ which consist of steps out of $S=\{(1,1),(1,-1),(2,0)\}$ and
do not pass below the $x$-axis 
\index{path, Schr\"oder}\index{Schr\"oder path}\index{Schr\"oder,
  Friedrich Wilhelm Karl Ernst}{\em (Schr\"oder paths)} is given by
\begin{multline} \label{CKe3.2}
\vv{\LL{(a,b)\to (c,d);S\mid y\ge0}}=
\sum _{k=0} ^{(c-a)/2}\binom {c-a-k}k\\
\cdot\(\binom {c-a-2k} {(c+d-2k-a-b)/2} -
\binom {c-a-2k}{(c+d-2k-a+b+2)/2}\),
\end{multline}
with the same convention for binomial coefficients.
\end{theorem}

\begin{proof} By the above described translation (reflection +
rotation), a Motzkin path from $(a,b)$ to $(c,d)$ with exactly $k$
level-steps is translated into
a path from $\(\frac {a+b} {2},\frac {a-b} {2}\)$ to 
$\(\frac {c+d} {2},\frac {c-d} {2}\)$, which consists of steps from 
$\{(1,0),(0,1),(\frac {1} {2},\frac {1} {2})\}$, among them exactly $k$ diagonal steps
$(\frac {1} {2},\frac {1} {2})$, and which stays weakly below 
the main diagonal $y=x$. Clearly, if we remove the $k$ diagonal
steps and concatenate the resulting path pieces, we obtain a simple
path from $\(\frac {a+b} {2},\frac {a-b} {2}\)$ to
$\(\frac {c+d} {2}-\frac {k} {2},\frac {c-d} {2}-\frac {k} {2}\)$ which stays weakly below $y=x$. The
number of the latter paths was determined in Theorem~\ref{CKT1.1}. On
the other hand, there are $\binom {c-a}k$ ways to reinsert the $k$
diagonal steps. Thus, Eq.~\eqref{CKe3.1} is established.

The proof of \eqref{CKe3.2} is analogous.
\end{proof}

We will derive expressions for corresponding generating functions
in\break Section~\ref{CKs3.3}, see Theorem~\ref{CKT3.5}.

\medskip
It is worth stating the special case of Theorem~\ref{CKT3.1} 
where the paths start and
terminate on the $x$-axis separately.

\begin{corollary} \label{CKC3.2}%
The number of 
\index{path, Motzkin}\index{Motzkin path}\index{Motzkin, Theodore S.}Motzkin paths 
from $(0,0)$ to $(n,0)$ is given by
\begin{equation} \label{CKe3.3}
\vv{\LL{(0,0)\to (n,0);M\mid y\ge0}}=
\sum _{k=0} ^{\fl{n/2}}\binom {n}{2k}\frac {1} {k+1}\binom {2k} {k}.
\end{equation}
If $n$ is even, the number of 
\index{path, Schr\"oder}\index{Schr\"oder path}\index{Schr\"oder,
  Friedrich Wilhelm Karl Ernst}Schr\"oder paths 
from $(0,0)$ to $(n,0)$ is given by
\begin{equation} \label{CKe3.4}
\vv{\LL{(0,0)\to (n,0);S\mid y\ge0}}=
\sum _{k=0} ^{n/2}\binom {n/2+k}{2k}\frac {1} {k+1}\binom {2k}{k}.
\end{equation}
\end{corollary}

The numbers in \eqref{CKe3.3} are called
\index{number, Motzkin}\index{Motzkin number}\index{Motzkin, Theodore S.}{\em Motzkin numbers}. 
The numbers in \eqref{CKe3.4} are called
\index{number, large Schr\"oder}\index{Schr\"oder
  number}\index{Schr\"oder, Friedrich Wilhelm Karl Ernst}{\em large Schr\"oder numbers}. If $n\ge1$, the latter are all divisible by $2$
(which is easily seen by switching the first occurrence of a level-step
with a pair consisting of an up-step and a down-step, and vice versa).
Dividing these numbers by $2$, we obtain the
\index{number, little Schr\"oder}\index{Schr\"oder
  number}\index{Schr\"oder, Friedrich Wilhelm Karl Ernst}{\em little Schr\"oder numbers}. Similarly to Catalan numbers,
also Motzkin and Schr\"oder numbers appear in numerous contexts;
see \cite[Ex.~6.38 and 6.39]{StanBI}.

The summations in \eqref{CKe3.1} and \eqref{CKe3.2} do not simplify, not
even in the special cases given in \eqref{CKe3.3} and \eqref{CKe3.4}. 

\medskip
In concluding this section, we point out that Motzkin paths,
or, more precisely, 
\index{Motzkin path}\index{path, Motzkin}{\it decorated\/} Motzkin
paths, are of utmost importance for the enumeration of many other
combinatorial objects, most importantly for the enumeration
of permutations and (set) partitions. A decorated Motzkin
path (in the french literature: {\it``histoire"}) is a Motzkin
path in which each step carries a certain label. In terms
of enumeration, one may consider this as allowing several
different steps of the same kind: for example, several
different horizontal steps, etc. In terms of generating
functions, this labelling is reflected by appropriate
weights of the steps. The importance of decorated
Motzkin paths comes from the fact that several bijections
have been constructed between them and permutations or
partitions, which have the property that they ``transfer"
detailed information about permutations or partitions
to the world of (decorated) Motzkin paths, allowing for
very refined enumeration results for permutations and partitions. 
Such bijections have been constructed by
\index{Biane, Philippe}Biane \cite{BianAA},
\index{Foata, Dominique}Foata and 
\index{Zeilberger, Doron}Zeilberger \cite{FoZeAA}, 
\index{Fran\c con, Jean}Fran\c con and 
\index{Viennot, Xavier G\'erard}Viennot \cite{FrViAA}, 
\index{M\'edicis, Anne de}M\'edicis and
\index{Viennot, Xavier G\'erard}Viennot \cite{MeViAA}, and by
\index{Simion, R.}Simion and \index{Stanton, D.}Stanton
\cite{SiStAA}. See \cite{ClSZAA} for a unifying view.

\section{A continued fraction for the weighted counting of Motzkin
paths}
\label{CKs3.3}
We now assign a weight to each 
\index{path, Motzkin}\index{Motzkin path}\index{Motzkin, Theodore S.}Motzkin 
path which starts
and ends on the $x$-axis, and express
the corresponding generating function in terms of a 
\index{continued fraction}\index{fraction, continued}{\em continued
fraction}. The corresponding result is 
due to \index{Flajolet, Philippe}Flajolet
\cite{FlajAA}. The weight is so general that the result also covers 
Schr\"oder paths and Catalan paths.

Given a Motzkin path $P$, we define the weight $w(P)$ to be the
product of the weights of all its steps, where the weight of an
up-step is $1$ (hence, does not contribute anything to the weight), the
weight of a level-step at height $h$ is $b_h$, and the weight of
a down-step from height $h$ to $h-1$ is $\la_h$. Figure~\ref{CKF3.4}
shows a Motzkin path the steps of which are labelled by their
corresponding weights, so that the weight of the path is
$b_2\la_2b_1b_1\la_3\la_2\la_1=b_1^2b_2\la_1\la_2^2\la_3$.

\begin{figure} 
$$
\Gitter(12,5)(-1,0)
\Koordinatenachsen(12,5)(-1,0)
\Pfad(0,0),33141133444\endPfad
\DickPunkt(0,0)
\DickPunkt(1,1)
\DickPunkt(2,2)
\DickPunkt(3,2)
\DickPunkt(4,1)
\DickPunkt(5,1)
\DickPunkt(6,1)
\DickPunkt(7,2)
\DickPunkt(8,3)
\DickPunkt(9,2)
\DickPunkt(10,1)
\DickPunkt(11,0)
\Label\l{1\kern3pt}(1,1)
\Label\l{1\kern3pt}(2,2)
\Label\lo{b_2}(3,2)
\Label\l{\kern3pt\la_2}(4,2)
\Label\lo{b_1}(5,1)
\Label\lo{b_1}(6,1)
\Label\l{1\kern2pt}(7,2)
\Label\l{1\kern2pt}(8,3)
\Label\l{\kern3pt\la_3}(9,3)
\Label\l{\kern3pt\la_2}(10,2)
\Label\l{\kern3pt\la_1}(11,1)
\hbox{\hskip6.5cm}
$$
\caption{}
\label{CKF3.4}
\end{figure}

Then the following theorem is true.

\begin{theorem} \label{CKT3.4}%
With the weight $w$ defined as above, the generating function for
\index{path, Motzkin}\index{Motzkin path}\index{Motzkin, Theodore S.}Motzkin paths 
running from the origin back to the $x$-axis, 
which stay weakly below the line $y=k$, is given by
\begin{multline} \label{CKe3.5}
\GF{\LL{(0,0)\to(*,0);M\mid 0\le y\le k};w}\\
=\cfrac 1{1-b_0-
\cfrac {\la_1}{1-b_1-
\cfrac {\la_2}{1-b_2-\cdots-
\cfrac {\la_{k}}{1-b_k}}}}\ .
\end{multline}
In particular, the generating function for {\em all} Motzkin paths
running from the origin back to the $x$-axis is
given by the infinite continued fraction
\begin{equation} \label{CKe3.6}
\GF{\LL{(0,0)\to(*,0);M\mid 0\le y};w}=\cfrac 1{1-b_0-
\cfrac {\la_1}{1-b_1-
\cfrac {\la_2}{1-b_2-\cdots}}}\ .
\end{equation}

\end{theorem}

\begin{proof} Clearly, it suffices to prove \eqref{CKe3.5}.
Equation~\eqref{CKe3.6} then follows upon letting $k\to\infty$.

We prove \eqref{CKe3.5} by induction on $k$. For $k=0$,
Equation~\eqref{CKe3.5} is trivially true. Hence, let us assume the
truth of \eqref{CKe3.5} for $k$ replaced by $k-1$. For accomplishing
the induction step, we consider a Motzkin path starting at the
origin, staying weakly below $y=k$, and finally returning to the $x$-axis, 
see Figure~\ref{CKF3.5} for an example with $k=3$.

\begin{figure} 
$$
\Einheit.5cm
\Gitter(19,5)(-1,0)
\Koordinatenachsen(19,5)(-1,0)
\Pfad(0,0),13314411343133444\endPfad
\DickPunkt(0,0)
\DickPunkt(1,0)
\DickPunkt(2,1)
\DickPunkt(3,2)
\DickPunkt(4,2)
\DickPunkt(5,1)
\DickPunkt(6,0)
\DickPunkt(7,0)
\DickPunkt(8,0)
\DickPunkt(9,1)
\DickPunkt(10,0)
\DickPunkt(11,1)
\DickPunkt(12,1)
\DickPunkt(13,2)
\DickPunkt(14,3)
\DickPunkt(15,2)
\DickPunkt(16,1)
\DickPunkt(17,0)
\PfadDicke{.3pt}
\Pfad(-2,3),111111111111111111111\endPfad
\hbox{\hskip8.5cm}
$$
\caption{}
\label{CKF3.5}
\end{figure}

Such a path can be uniquely decomposed into
$$l^{e_0}uP_1d\,l^{e_1}uP_2d\,l^{e_2}\dots ,$$
where $l$ denotes a level-step at height~$0$, $u$ an up-step, 
and $d$ a down-step,
where $e_i$ are non-negative integers, and where, for any $i$, 
$P_i$ is some path between the lines $y=1$ and $y=k$ which
starts at and returns to the line $y=1$. For example, this
decomposition applied to the path in Figure~\ref{CKF3.5} yields
$$l^1uP_1d\,l^2uP_2d\,l^0uP_3d,$$
where $P_1=uld$, $P_2$ is the empty path, and $P_3=luudd$.
This implies immediately the
generating function equation
\begin{multline*} 
\GF{\LL{(0,0)\to(*,0);M\mid 0\le y\le k};w}\\=\frac {1} 
{1-b_0-\la_1 \cdot\GF{\LL{(0,1)\to(*,1);M\mid 1\le y\le k};w}}.
\end{multline*}
By induction, the generating function on the right-hand side is
known: it is given by \eqref{CKe3.5} with $k$ replaced by $k-1$, 
$b_i$ replaced by $b_{i+1}$, and $\la_i$ replaced by $\la_{i+1}$, for
all $i$. This completes the induction step.
\end{proof}

This result has numerous consequences. First of all, it allows us to
derive algebraic expressions for the generating functions $\sum
_{n\ge0} M_nz^n$ and $\sum _{n\ge0} S_nz^n$, where $M_n$ denotes the
number of all Motzkin paths from $(0,0)$ to $(n,0)$, and where $S_n$
denotes the number of all Schr\"oder paths from $(0,0)$ to $(2n,0)$.
By definition, $M_0=S_0=1$.
The numbers $M_n$ are called 
\index{number, Motzkin}\index{Motzkin number}\index{Motzkin, Theodore S.}{\em Motzkin
numbers}, while the numbers $S_n$ are called 
\index{number, Schr\"oder}\index{Schr\"oder
  number}\index{Schr{\"o}der, Friedrich Wilhelm Karl Ernst}{\em large Schr\"oder numbers}. 

\begin{theorem} \label{CKT3.5}%
We have
\begin{equation} \label{CKe3.7}
\sum _{n\ge0} ^{}M_nz^n=\frac {1-z-\sqrt{1-2z-3z^2}} {2z^2}
\end{equation}
and
\begin{equation} \label{CKe3.8}
\sum _{n\ge0} ^{}S_{n}z^n=\frac {1-z-\sqrt{1-6z+z^2}} {2z}.
\end{equation}
\end{theorem}
\begin{proof} 
By \eqref{CKe3.6} with $b_i=z$ and $\la_i=z^2$ for all $i$
(the reader should note that for any down-step there is a corresponding
up-step before),
we have 
$$\sum _{n\ge0} ^{}M_nz^n=\cfrac 1{1-z-
\cfrac {z^2}{1-z-
\cfrac {z^2}{1-z-\cdots}}}\ .
$$
Thus, in particular, we have $M(z)=1/(1-z-z^2M(z))$. The appropriate
solution of this quadratic equation is exactly the right-hand side of
\eqref{CKe3.7}. 

Similarly, by setting $b_i=\la_i=z^2$ in \eqref{CKe3.6} for all $i$, we
obtain 
$$\sum _{n\ge0} ^{}S_nz^{2n}=\cfrac 1{1-z^2-
\cfrac {z^2}{1-z^2-
\cfrac {z^2}{1-z^2-\cdots}}}\ ,
$$
and eventually \eqref{CKe3.8} after solving the analogous quadratic
equation.
\end{proof}

In Section~\ref{CKs3.6}, we will express the continued fraction
\eqref{CKe3.5} in numerator/denominator form, the numerator and
denominator being orthogonal polynomials.

\begin{figure}[h]
$$
\Gitter(17,6)(0,0)
\Koordinatenachsen(17,6)(0,0)
\Pfad(0,0),3343334344344434\endPfad
\DickPunkt(0,0)
\DickPunkt(16,0)
\hskip9cm
$$
\caption{}
\label{CKfig2}
\end{figure}

\medskip
We conclude this section with another continued fraction result,
due to 
\index{Roblet, Emmanuel}Roblet and 
\index{Viennot, Xavier G\'erard}Viennot \cite{RoViAA}.
We restrict our attention to Dyck paths, that is, to paths 
consisting of up- and down-steps, starting at the origin and
returning to the $x$-axis, and never running below the $x$-axis. 
We refine the earlier
defined weight $w$ in the following way, 
so that in addition it also takes into account peaks: 
given a Dyck path $P$, we define the weight $\hat w(P)$ of $P$ to be the
product of the weights of all its steps, where the weight of an
up-step is $1$, the weight of 
a down-step from height $h$ to $h-1$ which follows immediately after
an up-step (thus, together, forming a peak of the
path) is $\nu_h$, and where the weight of
a down-step from height $h$ to $h-1$ which follows after another
down-step is $\la_h$. Thus, the weight of
the Dyck path in Figure~\ref{CKfig2} is 
$\nu_2\nu_4\nu_4\la_3\nu_3\la_2\la_1\nu_1=
\nu_1\nu_2\nu_3\nu_4^2\la_1\la_2\la_3.$
With these definitions, the theorem of
\index{Roblet, Emmanuel}Roblet and 
\index{Viennot, Xavier G\'erard}Viennot \cite[Prop.~1]{RoViAA} reads as follows.

\begin{theorem}%
With the weight $\hat w$ defined as above, 
the generating function $\sum _{P} ^{}\hat w(P)$, where the sum is over 
all Dyck paths starting at the origin and returning to the
$x$-axis, is given by
\begin{multline} \label{CKeq:RV} 
\GF{\LL{(0,0)\to(*,0);\{(1,1),\,(1,-1)\}\mid y\ge0};\hat w}\\
=\cfrac 1{1-(\nu_1-\la_1)-
\cfrac {\la_1}{1-(\nu_2-\la_2)-
\cfrac {\la_2}{1-(\nu_3-\la_3)-\cdots}}} .
\end{multline}
\end{theorem}

\section{Lattice paths and orthogonal polynomials}
\label{CKs3.4}\index{orthogonal polynomial}\index{polynomial,
orthogonal}

Orthogonal polynomials play an important role 
in many different subject areas, may they be pure or applied. The reader is referred to
\cite{SzegOP} for an in-depth introduction. It is well-known
that the theory of orthogonal polynomials is intimately
connected with Hankel determinants and continued fractions,
which we just discussed in Section~\ref{CKs3.3} from a
combinatorial point of view.
It is \index{Viennot, Xavier G\'erard}Viennot
\cite{VienAE} who made the connection, and who showed 
that a large part of the theory of orthogonal
polynomials is in fact combinatorics. The key objects in this {\em
combinatorial theory of orthogonal polynomials} are \index{path,
Motzkin}\index{Motzkin path}\index{Motzkin, Theodore S.}{\em Motzkin
  paths}. 
If one is forced to,
one may compress the interplay between the theory of orthogonal
polynomials and path enumeration to two key facts: first,
(generalized) moments of orthogonal polynomials are generating functions
   for Motzkin paths, see Theorem~\ref{CKP3.1}; 
second, generating functions for bounded Motzkin paths 
   can be expressed in terms of orthogonal polynomials; see
Theorem~\ref{CKT3.11}. But, of course, this combinatorial theory
of orthogonal polynomials has much more to offer, of which
we present an extract in this section,
with a focus on path enumeration.

\medskip
We call a sequence $(p_n(x))_{n\ge0}$ of polynomials
over $\C$,
where $p_n(x)$ is of degree~$n$.
\index{orthogonal}\index{polynomial, orthogonal}{\em
orthogonal} if there exists a \index{linear functional}\index{functional,
linear}linear functional $L$ on polynomials
over $\C$ (i.e., a linear map, which maps a polynomial to a complex
number) such that
\begin{equation} \label{CKe3.10}
L(p_n(x)p_m(x))=\begin{cases} 0,&\text {if $n\ne m$,}\\
\text {nonzero},&\text {if $n=m$.}\end{cases}
\end{equation}
We alert the reader that our definition deviates from the classical
{\em analytic} definition in that we {\em do not} require $L(p_n(x)^2)$
to be {\it positive}. The above somewhat weaker notion of
orthogonality is sometimes referred to as {\em formal orthogonality}.
The term `formal' expresses the fact that the corresponding theory does not
require any analytic tools, just formal, algebraic arguments. In
fact, the formal theory could be equally well developed over any
field $K$ of characteristic 0 (instead of over $\C$). 

\medskip
It is easy to see that it is not true that
for every linear functional $L$ there is a corresponding sequence
of orthogonal polynomials. Let us consider the example of the linear
functional defined by $L(x^n):=1$, $n=0,1,2,\dots$. Equivalently,
this means that $L(p(x))=p(1)$. In order to construct a corresponding
sequence of orthogonal polynomials, we start with $p_0(x)$. This must
be a polynomial of degree $0$, but otherwise we are completely free.
Without loss of generality we may choose $p_0(x)\equiv 1$. To
determine $p_1(x)$, we use \eqref{CKe3.10} with $m=0$ and $n=1$. Thus we
obtain $p_1(x)=x-1$. But then we have $L(p_1(x)^2)=L((x-1)^2)=0$,
which violates the requirement \eqref{CKe3.10}, with $m=n=1$, that
$L(p_1(x)^2)$ should be nonzero.

On the other hand, if we have a linear functional $L$ such that there
exists a sequence of orthogonal polynomials, then it is easy
to see that all other sequences
are just linear multiples of the former sequence.

\begin{lemma} \label{CKL3.1}%
Let $L$ be a linear functional on polynomials and $(p_n(x))_{n\ge0}$
be a sequence of polynomials orthogonal with respect to $L$. If
$(q_n(x))_{n \ge0}$ is another sequence of polynomials orthogonal
with respect to $L$, then there are nonzero numbers $a_n\in\C$ such
that $q_n(x)=a_n\,p_n(x)$.
\end{lemma}

Lemma~\ref{CKL3.1} justifies that from now on we will restrict our
attention to sequences of \index{monic polynomial}\index{polynomial,
monic}{\em monic} polynomials.

One of the key results in the theory of
orthogonal polynomials is \index{Favard,
Jean}\index{Favard's Theorem}\index{theorem, Favard's}{\em 
Favard's Theorem}, which we state next. 

\begin{theorem} \label{CKT3.6}%
A sequence $(p_n(x))_{n\ge0}$ of monic polynomials, $p_n(x)$ being of
degree $n$, is orthogonal if and only if 
there exist sequences $(b_n)_{n\ge0}$ and $(\la_n)_{n\ge1}$, with
$\la_n\ne0$ for all $n\ge1$, such that the three-term recurrence
\begin{equation}
\label{CKeq:three-term}
xp_{n}(x)=p_{n+1}(x)+b_np_n(x)+\la_{n}p_{n-1}(x),\quad \quad 
\text{ for } n\geq 1,
\end{equation}
holds, with initial conditions $p_0(x)=1$ and $p_1(x)=x-b_0$.

\end{theorem}

In our
context, more important than the statement of the theorem itself is its
proof, which introduces Motzkin paths in a surprising way in
\eqref{CKe3.16}, and in
particular Theorem~\ref{CKP3.1} below, which is the key ingredient
in the proof,
given after the proof of Theorem~\ref{CKP3.1}.

\begin{theorem} \label{CKP3.1}%
Let the polynomials $p_n(x)$ be
given by the three-term recurrence \eqref{CKeq:three-term},
and let $L$ be the linear functional defined by $L(1)=1$
and $L\big(p_n(x)\big)=0$ for $n\ge1$. Then
\begin{equation} \label{CKe3.17}
L\big(x^n\,p_k(x)\,p_l(x)\big)=\la_1\cdots\la_l\cdot
\GF{\LL{(0,k)\to(n,l);M\mid 0\le y};w},
\end{equation}
where $w$ is the weight on Motzkin paths defined in Section~\ref{CKs3.3}.
\end{theorem}
\begin{proof} 
We prove the assertion by induction on $n$.

If $n=0$, then we have to show
\begin{equation} \label{CKe3.17a}
L\big(p_k(x)\,p_l(x)\big)=\la_1\cdots\la_l\cdot
\de_{k,l},
\end{equation}
where $\de_{k,l}$ denotes the Kronecker delta.
We establish this claim by induction on $k+l$. It is obviously
true for $k=l=0$. Without loss of generality, we assume $k\ge l$.
Then, using the three-term recurrence \eqref{CKeq:three-term} twice,
together with the induction hypothesis, we have
\begin{align*}
L\big(p_k(x)\,p_l(x)&\big)\\
&=
L\big(p_k(x)\,xp_{l-1}(x)\big)
-b_{l-1}L\big(p_k(x)\,p_{l-1}(x)\big)
-\la_{l-1}L\big(p_k(x)\,p_{l-2}(x)\big)\\
&=
L\big(xp_k(x)\,p_{l-1}(x)\big)\\
&=
L\big(p_{k+1}(x)\,p_{l-1}(x)\big)
+b_kL\big(p_{k}(x)\,p_{l-1}(x)\big)
+\la_kL\big(p_{k-1}(x)\,p_{l-1}(x)\big)\\
&=
\la_kL\big(p_{k-1}(x)\,p_{l-1}(x)\big).
\end{align*}
Clearly, this achieves the induction step, and thus establishes
\eqref{CKe3.17a}.

We may now continue with the induction on~$n$.
For the induction step, we apply
\eqref{CKeq:three-term} with $n=k$ on the left-hand side of
\eqref{CKe3.17}. This leads to
\begin{multline*}
L\big(x^n\,p_k(x)\,p_l(x)\big)\\=
L\big(x^{n-1}\,p_{k+1}(x)\,p_l(x)\big)
+b_kL\big(x^{n-1}\,p_k(x)\,p_l(x)\big)
+\la_kL\big(x^{n-1}\,p_{k-1}(x)\,p_l(x)\big).
\end{multline*}
By the induction hypothesis, we may interpret the 
right-hand side of this equality as generating function for
Motzkin paths, as described by \eqref{CKe3.17} with $n$ replaced
by $n-1$. It is then
straightforward to see that this implies \eqref{CKe3.17} itself.
\end{proof}

Now we have all the prerequisites available
in order to prove Theorem~\ref{CKT3.6}.

\bigskip
\noindent
{\bf Proof of Theorem~\ref{CKT3.6}}\indent
For showing the forward implication, let
$(p_n(x))_{n\ge0}$ be a sequence of monic polynomials, $p_n(x)$ of
degree $n$, which is orthogonal with respect to the linear functional
$L$. Then we can express $xp_n(x)$ in terms of a linear combination
of the polynomials $p_{n+1}(x)$, $p_n(x)$, \dots, $p_0(x)$,
\begin{equation} \label{CKe3.15}
xp_n(x)=p_{n+1}(x)+b_np_n(x)+\la_np_{n-1}(x)+\om_{n,n-2}p_{n-2}(x)
+\dots+\om_{n,0}p_0(x).
\end{equation} 
We have to show that in fact the first three terms on the right-hand
side suffice, i.e., that all other terms are zero. 

In order to do that, we multiply both sides of \eqref{CKe3.15} by $p_i(x)$,
for some $i<n-1$, and apply $L$ on both sides. Because of
\eqref{CKe3.10}, on the right-hand side it is only the term
$\om_{n,i}L\big(p_i(x)^2\big)$ which survives. On the left-hand side we
obtain $L\big(xp_i(x)p_n(x)\big)$. The polynomial $xp_i(x)$ of degree
$i+1$ can be expressed as a linear combination of the polynomials
$p_{i+1}(x)$, $p_i(x)$, \dots, $p_0(x)$. Because of
\eqref{CKe3.10} and $i<n-1$, we therefore conclude that
$L\big(xp_i(x)p_n(x)\big)=0$. Hence, $\om_{n,i}$ is indeed $0$ for
$i<n-1$. Similarly, we have
$$\la_n\,L\big(p_{n-1}(x)^2\big)=L\big(xp_{n-1}(x)p_n(x)\big)=
L\big(p_n(x)^2\big),$$
which is nonzero because of \eqref{CKe3.10}. Hence, we have $\la_n\ne0$, as
desired.

\medskip
For the proof of the backward implication, we must construct a linear
functional $L$ such that \eqref{CKe3.10} holds, given a sequence
$(p_n(x))$ of polynomials, $p_n(x)$ of degree $n$, satisfying the
three-term recurrence \eqref{CKeq:three-term}. We construct $L$ by
defining $L(1)=1$ and $L\big(p_n(x)\big)=0$ for $n\ge1$.
Theorem~\ref{CKP3.1} with $n=0$
immediately implies that $L\big(p_k(x)p_l(x)\big)=0$ if $k\ne l$, as there
is no Motzkin path from $(0,k)$ to $(0,l)$, and
that $L\big(p_k(x)^2\big)=\la_1\cdots\la_k\ne0$. This completes the proof of
the theorem.
\bigskip

In the above proof, we have found a linear functional $L$ by
defining (cf.\ Theorem~\ref{CKP3.1})
its moments $\mu_n:=L(x^n)$
to be generating functions for Motzkin paths, namely
\begin{equation} \label{CKe3.16}
\mu_n=\GF{\LL{(0,0)\to(n,0);M\mid 0\le y};w}=
\underset{\text {from $(0,0)$ to $(n,0)$}}
{\sum _{P\text { a Motzkin path}} ^{}}w(P),
\end{equation}
the
weights of the paths carrying the coefficients in the three-term recurrence
\eqref{CKeq:three-term}. This definition generates a linear functional
with $\mu_0=L(1)=1$. It is easy to see 
that all other such linear functionals are constant nonzero multiples of the
linear functional defined by \eqref{CKe3.16}. 
This justifies to restrict ourselves to linear functionals with
first moment equal to $1$.

In view of Theorem~\ref{CKT3.4}, the backward implication of Theorem~\ref{CKT3.6}
can also be phrased in the following way.

\begin{corollary} \label{CKC3.0}%
Let $(p_n(x))_{n\ge0}$ be a sequence of polynomials satisfying the
three-term recurrence \eqref{CKeq:three-term} with initial
conditions $p_0(x)=1$ and $p_1(x)=x-b_0$. Then $(p_n(x))_{n\ge0}$ is
orthogonal with respect to the linear functional $L$, where the generating
function of its moments $\mu_n=L(x^n)$ is given by
\begin{equation} \label{CKe3.16a}
\sum _{n\ge0} ^{}\mu_n\, z^n=
\cfrac 1{1-b_0z-
\cfrac {\la_1z^2}{1-b_1z-
\cfrac {\la_2z^2}{1-b_2z-\cdots}}}\ .
\end{equation}
All other linear functionals with respect to which the sequence
$(p_n(x))_{n\ge0}$ is orthogonal are constant nonzero multiples of $L$.
\end{corollary}

\begin{remark}\em
A continued fraction of the type \eqref{CKe3.16a} is called a
\index{J-fraction}\index{fraction, J-fraction}\index{Jacobi, Carl Gustav Jacob}%
{\em Jacobi continued fraction} or {\em J-fraction}.
\end{remark}

\bigskip
\noindent
{\bf Proof of Corollary~\ref{CKC3.0}}\indent
Combine \eqref{CKe3.16} and \eqref{CKe3.6} with $b_i$ replaced by $b_iz$
and $\la_i$ replaced by $\la_iz$.
\bigskip

Below, we illustrate what we have found so far by an
example. The polynomials which appear in this example, the 
\index{Chebyshev, Pafnuty Lvovich}%
\index{polynomial, Chebyshev}\index{Chebyshev polynomial}{\em Chebyshev
polynomials}, are of particular importance for path counting.

\begin{example}\label{CKex:Ch}%
\em
We choose $b_i=0$ and $\la_i=1$ for all $i$. Then the
three-term recurrence \eqref{CKeq:three-term} becomes
\begin{equation} \label{CKet3.1}
xu_{n}(x)=u_{n+1}(x)+u_{n-1}(x),\quad \quad 
\text{ for } n\geq 1,
\end{equation}
with initial values $u_0(x)=1$ and $u_1(x)=x$. These polynomials are,
up to re\-pa\-ra\-metri\-za\-tion, 
\index{Chebyshev, Pafnuty Lvovich}%
\index{polynomial, Chebyshev}\index{Chebyshev polynomial}{\em Chebyshev
polynomials of the second kind}. To see that, recall that the latter
are defined by
$$U_n(\cos\th)=\frac {\sin((n+1)\th)} {\sin\th},$$
or, equivalently,
$$U_n(x)=\frac {\sin((n+1)\arccos x)} {\sqrt {1-x^2}}.$$
Because of the easily verified fact that
$$\sin((n+1)\th)+\sin((n-1)\th)=2\cos\th\sin n\th,$$
the Chebyshev polynomials of the second kind satisfy the three-term
recurrence
\begin{equation} \label{CKet3.2}
2xU_{n}(x)=U_{n+1}(x)+U_{n-1}(x),\quad \quad 
\text{ for } n\geq 1,
\end{equation}
with initial values $U_0(x)=1$ and $U_1(x)=2x$. Therefore we have
\begin{equation} \label{CKet3.3}
U_n(x)=u_n(2x)
\end{equation}
for all $n$.

It is straightforward to verify
\begin{equation} \label{CKet3.3a}
U_n(x)=\sum _{k\ge0} ^{}(-1)^k\binom {n-k}k (2x)^{n-2k},
\end{equation}
whence, by \eqref{CKet3.3}, we have
$$u_n(x)=\sum _{k\ge0} ^{}(-1)^k\binom {n-k}k x^{n-2k}.$$

Another well-known fact is
$$\frac {2} {\pi}\int _{0} ^{\pi}\sin((n+1)\th)\sin((m+1)\th)\,d\th =
\begin{cases} 1,&n=m,\\ 0,&n\ne m.\end{cases}$$
Substitution of $x=\cos\th$ then yields
\begin{equation} \label{CKet3.4}
\frac {2} {\pi}\int _{-1} ^{1}U_n(x)U_m(x)\sqrt{1-x^2}\,dx =
\de_{nm}.
\end{equation}
Thus the linear functional $L$ for Chebyshev polynomials of the
second kind is given by
$$L(p(x))=\frac {2} {\pi}\int _{0} ^{\pi}p(x)\sqrt{1-x^2}\,dx.$$

Using \eqref{CKe3.16} we can now easily compute the corresponding
moments. On the right-hand side of \eqref{CKe3.16} all the terms
corresponding to paths which contain a level-step vanish, because
$b_i=0$ for all $i$. Therefore, what the right-hand side counts are
paths which contain only up-steps and down-steps (and never pass
below the $x$-axis). Clearly, there cannot be such a path if $n$ is
odd. If $n$ is even, then by \eqref{CKe1.5} the number of these paths
is the \index{Catalan numbers}\index{Catalan, Eug\`ene Charles}Catalan number 
$\frac {1} {n/2+1}\binom {n}{n/2}$. Hence, by also taking into
account \eqref{CKet3.3}, we have shown that
$$\frac {2} {\pi}\int _{-1} ^{1}x^m\sqrt{1-x^2}=\begin{cases} \frac
{1} {4^n}\frac {1} {n+1}\binom {2n}{n},&m=2n,\\0,&m=2n+1.\end{cases}$$

\medskip
\index{Chebyshev, Pafnuty Lvovich}%
\index{polynomial, Chebyshev}\index{Chebyshev polynomial}Chebyshev 
polynomials are not only tied to Catalan paths (Dyck paths), i.e., paths
that consist of just up- and down-steps, but also to Motzkin paths.
To see this, let us now choose $b_i=\la_i=1$ for all $i$. Then the
three-term recurrence \eqref{CKeq:three-term} becomes
\begin{equation} \label{CKet3.9}
xm_{n}(x)=m_{n+1}(x)+m_n(x)+m_{n-1}(x),\quad \quad 
\text{ for } n\geq 1,
\end{equation}
with initial values $m_0(x)=1$ and $m_1(x)=x-1$. Comparison with
\eqref{CKet3.2} reveals that these polynomials are expressible by means
of Chebyshev polynomials of the second kind as
\begin{equation} \label{CKet3.10}
m_n(x)=U_n\(\frac {x-1} {2}\).
\end{equation}
We will take advantage of this relation in Section~\ref{CKs3.6} to
obtain further enumerative results on Motzkin paths.
\end{example}

We now come back to the earlier observed fact that not all linear
functionals allow for a corresponding sequence of orthogonal
polynomials. Which linear functionals do is told by the following
theorem. 
The criterion is given in terms of
\index{determinant, Hankel}\index{Hankel determinant}\index{Hankel, Hermann}{\em Hankel
determinants} of the moments of $L$. A Hankel determinant 
(or \index{determinant, persymmetric}\index{persymmetric
determinant}{\em persymmetric} or \index{Tur\'anian
determinant}\index{determinant, Tur\'anian}\index{Tur\'an, P\'al}{\em Tur\'anian determinant}) 
is a determinant of a matrix which has
constant entries along antidiagonals, i.e., it is a determinant of
the form $\det_{1\le i,j,\le n}(a_{i+j}).$ 
We omit the proof here, but
\index{Viennot, Xavier G\'erard}Viennot \cite[Ch.~IV, Cor.~6 and 7]{VienAE}
has shown that it can be given by an elegant application of the
main theorem on \index{non-intersecting lattice paths}%
\index{lattice paths, non-intersecting}non-intersecting 
lattice paths, Theorem~\ref{CKT2.3}, by using the interpretation
of moments in terms of generating functions for Motzkin paths
as given in Theorem~\ref{CKP3.1}.

\begin{theorem} \label{CKT3.7}%
Let $L$ be a linear functional on polynomials with $n$-th moment
$\mu_n=L(x^n)$.
For any non-negative integer $n$ let
$$\De_n=\det\begin{pmatrix} \mu_0&\mu_1&\mu_2&\dots&\mu_n\\
\mu_1&\mu_2&\hdotsfor2 &\mu_{n+1}\\
\mu_2&\hdotsfor3&\mu_{n+2}\\
\vdots&&&&\vdots\\
\mu_n&\hdotsfor3&\mu_{2n}
\end{pmatrix}$$
and
$$\chi_n=\det\begin{pmatrix} \mu_0&\mu_1&\dots&\mu_{n-1}&\mu_n\\
\mu_1&\mu_2&\dots&\mu_n&\mu_{n+1}\\
\vdots&\vdots&&\vdots&\vdots\\
\mu_{n-1}&\mu_{n}&\dots&\mu_{2n-2}&\mu_{2n-1}\\
\mu_{n+1}&\mu_{n+2}&\dots&\mu_{2n}&\mu_{2n+1}
\end{pmatrix}$$
Let $(p_n(x))_{n\ge0}$ be the sequence
of monic polynomials which is orthogonal with respect
to $L$. Then the
polynomials satisfy the three-term recurrence \eqref{CKeq:three-term}
with 
\begin{equation} \label{CKe3.23}
\la_n=\frac {\De_n\De_{n-2}} {\De_{n-1}^2}
\end{equation}
and
\begin{equation} \label{CKe3.24}
b_n=\frac {\chi_{n}} {\De_{n}}-\frac {\chi_{n-1}} {\De_{n-1}}.
\end{equation}

In particular,
given a linear functional $L$ on the set of polynomials, then there
exists a sequence of orthogonal polynomials which are orthogonal with
respect to $L$ if and only if all Hankel determinants
$\De_n=\det_{0\le i,j\le n}(\mu_{i+j})$ of moments are nonzero.
\end{theorem}

Implicit in \eqref{CKe3.23} is the Hankel determinant evaluation
\begin{equation} \label{CKe3.26}
\De_n=\la_1^n\la_2^{n-1}\cdots\la_n^1,
\end{equation}
which expresses the close interplay between Hankel determinants,
moments of orthogonal polynomials, and Motzkin path enumeration
(via Theorem~\ref{CKP3.1}).

\medskip
We conclude this section with an explicit, determinantal formula
for orthogonal polynomials, given the moments of the orthogonality
functional. Again, 
\index{Viennot, Xavier G\'erard}Viennot \cite[Ch.~IV, \S4]{VienAE}
has given a beautiful combinatorial proof for this formula.
using
\index{non-intersecting lattice paths}%
\index{lattice paths, non-intersecting}non-intersecting lattice paths.

\begin{theorem} \label{CKT3.8}%
Let $L$ be a linear functional defined on polynomials with moments
$\mu_n=L(x^n)$. Then the corresponding sequence
$(p_n(x))_{n\ge0}$ of monic orthogonal polynomials is given by
\begin{equation}
\label{CKeq:ortho}
p_{n}(x)=\frac {1} {\De_{n-1}}\det\begin{pmatrix}
		\mu_0	&	\mu_1	&	\mu_2	&	\dots &	\mu_n \\
		\mu_1	&	\mu_2	&	\dots	&	\mu_n	&	\mu_{n+1} \\
		\mu_2	&	\dots	&	\mu_n	&	\mu_{n+1}&\mu_{n+2} \\
		\hdotsfor{5} \\
		\mu_{n-1} &	\mu_n	&	\mu_{n+1}&\dots	&\mu_{2n-1} \\
		1	&	x&\dots&	x^{n-1}&x^n 
	\end{pmatrix},
\end{equation}
where, again, $\De_{n-1}=\det_{0\le i,j\le n-1}(\mu_{i+j})$.
\end{theorem}

\begin{proof}
It suffices to check that $L(x^mp_n(x))=0$ for $0\le m<n$. Indeed, by
\eqref{CKeq:ortho} we have
\begin{equation*}
L(x^mp_{n}(x))=\frac {1} {\De_{n-1}}\det\begin{pmatrix}
		\mu_0	&	\mu_1	&	\mu_2	&	\dots &	\mu_n \\
		\mu_1	&	\mu_2	&	\dots	&	\mu_n	&	\mu_{n+1} \\
		\mu_2	&	\dots	&	\mu_n	&	\mu_{n+1}&\mu_{n+2} \\
		\hdotsfor{5} \\
		\mu_{n-1} &	\mu_n	&	\mu_{n+1}&\dots	&\mu_{2n-1} \\
		\mu_m	&	\mu_{m+1}&\dots&
\mu_{m+n-1}&\mu_{m+n}
	\end{pmatrix}.
\end{equation*}
Thus the result is zero, because for $0\le m<n$ the $m$-th and the
last row in the above determinant are identical.
\end{proof}

In Section~\ref{CKs3.6} we derive several further enumeration results on Motzkin
paths which feature orthogonal polynomials.

We close this section by pointing out that Motzkin paths can be seen
as so-called
\index{heap of pieces}{\em heaps of pieces}.
The corresponding theory has been developed by
\index{Viennot, Xavier G\'erard}Viennot \cite{VienAF}.
As a matter of fact, it is the combinatorial realization
of the \index{Cartier, Pierre}{\it Cartier--\index{Foata, Dominique}Foata
monoid\/} \cite{CaFoAA}. 

For further intriguing work on the connections between lattice path
counting, Hankel determinants, and continued fractions, the reader is referred to
\index{Gessel, Ira Martin}Gessel and 
\index{Xin, Guoce}Xin \cite{GeXiAA}, and 
also \index{Sulanke, Robert A.}Sulanke and 
\index{Xin, Guoce}Xin \cite{SuXiAA}.

\section{Motzkin paths in a strip}
\label{CKs3.6}
In Sections~\ref{CKs3.1} and \ref{CKs3.3} we have derived enumeration
results for Motzkin paths which start and terminate on the $x$-axis.
In particular, Theorem~\ref{CKT3.4} provided a continued fraction for
the generating function with respect to a very general weight.
This continued fraction can be compactly brought in 
numerator/denominator form, using
orthogonal polynomials. In fact, more generally, a compact
expression for the generating function of Motzkin paths which
start and terminate at {\em arbitrary} points can be given, 
again using orthogonal polynomials.

In order to be able to state the corresponding result, we need
two definitions. Recall that,
given sequences $(b_n)_{n\ge0}$ and $(\la_n)_{n\ge1}$, with
$\la_n\ne0$ for all $n\ge1$, the three-term recurrence
\eqref{CKeq:three-term},
\begin{equation} \label{CKeq:three-term2}
xp_{n}(x)=p_{n+1}(x)+b_np_n(x)+\la_{n}p_{n-1}(x),\quad \quad 
\text{ for } n\geq 1,
\end{equation}
with initial conditions $p_0(x)=1$ and $p_1(x)=x-b_0$, produces
a sequence $(p_n(x))_{n\ge0}$ of orthogonal polynomials. 
We also need associated ``shifted" polynomials (often simply called
\index{associated orthogonal polynomials}\index{polynomial,
associated orthogonal}{\em
associated orthogonal polynomials}), denoted by $(Sp_n(x))_{n\ge0}$,
which arise from the sequence $(p_n(x))$ by replacing $\la_i$ by
$\la_{i+1}$ and $b_i$ by $b_{i+1}$, $i=0,1,2,\dots$, everywhere in
the three-term recurrence \eqref{CKeq:three-term2} and in the initial
conditions. Furthermore, given a polynomial $p(x)$ of degree $n$, we
denote the corresponding \index{polynomial,
reciprocal}\index{reciprocal polynomial}{\em reciprocal polynomial}
$x^np(1/x)$ by $p^*(x)$.

\begin{theorem} \label{CKT3.11}%
With the weight $w$ defined as before Theorem~\ref{CKT3.4}, 
the generating function for
\index{path, Motzkin}\index{Motzkin path}\index{Motzkin, Theodore S.}Motzkin paths running from 
height $r$ to height $s$
which stay weakly below the line $y=k$ is given by
\begin{multline} \label{CKe3.43}
\sum _{n\ge0} ^{}\GF{\LL{(0,r)\to(n,s);M\mid 0\le y\le k};w}x^n\\
=\begin{cases} \dfrac {x^{s-r}p^*_r(x)S^{s+1}p_{k-s}^*(x)} {p_{k+1}^*(x)},
&\text{if }r\le s,\\
\la_r\cdots\la_{s+1}\dfrac {x^{r-s}p^*_s(x)S^{r+1}p_{k-r}^*(x)}
{p_{k+1}^*(x)},
&\text{if }r\ge s.
\end{cases}
\end{multline}
In particular, the generating function for
\index{path, Motzkin}\index{Motzkin path}\index{Motzkin, Theodore S.}Motzkin paths running from 
the origin back to the $x$-axis
which stay weakly below the line $y=k$, is given by
\begin{equation} \label{CKe3.39}
\sum _{n\ge0} ^{}\GF{\LL{(0,0)\to(n,0);M\mid 0\le y\le k};w}x^n
=\frac {Sp_k^*(x)} {p_{k+1}^*(x)}.
\end{equation}
\end{theorem}
\begin{proof}
Consider the directed graph, $P_{k+1}$ say, with vertices 
$v_0,v_1,\dots v_k$, where for $h=0,1,\dots, k-1$ there is an arc
from $v_h$ to $v_{h+1}$ as well as an arc from $v_{h+1}$ to $v_h$, 
and where there is a loop for each vertex $v_h$. 
Motzkin paths which never exceed height $k$ 
correspond in a one-to-one fashion to 
{\em walks} on $P_{k+1}$.
In this correspondence, an up-step from height $h$ to
$h+1$ in the Motzkin path corresponds to a step from vertex $v_h$ to
vertex $v_{h+1}$ in the walk, and similarly for level- and
down-steps. To make the correspondence also weight-preserving, we
attach a weight of $1$ to an arc from $v_h$ to $v_{h+1}$,
$h=0,1,\dots,k-1$, a weight of $\la_h$ to an arc from $v_{h}$ to
$v_{h-1}$, and a weight of $b_h$ to a loop at $v_h$.

By the 
\index{transfer matrix method}\index{method, transfer matrix}{\em transfer
matrix method} (see e.g\@. \cite[Theorem~4.7.2]{StanAP}),
the generating function for walks from $v_r$ to $v_s$ is given by 
$$\frac {(-1)^{r+s}\det(I-xA;s,r)} {\det(I-xA)},$$
where $A$ is the (weighted) {\it adjacency matrix} of $P_{k+1}$, where $I$
is the $(k+1)\times (k+1)$ identity matrix, and where $\det(I-xA;s,r)$
is the minor of $(I-xA)$ with the $s$-th row and $r$-th column
deleted. 

Now, the (weighted) adjacency matrix of $P_{k+1}$ with the property
that the weight of a particular walk would correspond to the weight $w$ of the
corresponding Motzkin path is the tridiagonal matrix
$$A=\begin{pmatrix} b_0&1&0&\dots\\
\la_1&b_1&1&0&\dots\\
0&\la_2&b_2&1&0&\dots\\
\vdots&\ddots&\ddots&\ddots&\ddots&\ddots&\vdots\\
&\dots&0&\la_{k-2}&b_{k-2}&1&0\\
&&\dots&0&\la_{k-1}&b_{k-1}&1\\
&&&\dots&0&\la_{k}&b_k\end{pmatrix}.$$
It is easily verified that, with this choice of $A$, we have
$\det(I-xA)=p_{k+1}^*(x)$ (by expanding the determinant 
with respect to the last row
and comparing with the three-term recurrence \eqref{CKeq:three-term}), 
and, similarly, that the numerator in \eqref{CKe3.43} agrees with  
$(-1)^{r+s}\det(I-xA;r,s)$.
\end{proof}

\begin{example}\label{CKex:Cheb}%
\em
We illustrate Theorem~\ref{CKT3.11} for the special cases
which were considered in Example~\ref{CKex:Ch}.

Let first
$b_i=0$ and $\la_i=1$ for all $i$. Combinatorially, we are talking
about paths consisting of up- and down-steps, that is, 
\index{Catalan path}\index{path, Catalan}\index{Catalan, 
Eug\`ene Charles}Catalan paths (Dyck paths).
Since for this choice
of $b_i$'s and $\la_i$'s there is no difference between the
orthogonal polynomials and the corresponding associated orthogonal
polynomials arising from \eqref{CKeq:three-term},  
Example~\ref{CKex:Ch} tells us that
\begin{equation*} 
p_n(x)=Sp_n(x)=U_n(x/2). 
\end{equation*}
From \eqref{CKe3.43}, it then follows that
\begin{multline} 
\sum _{n\ge0} ^{}\big\vert\LL{(0,r)\to(n,s);
\{(1,1),(1,-1)\}\mid 0\le y\le k}\big\vert
\cdot x^n\\
=\begin{cases} \dfrac {U_r(1/2x)\,U_{k-s}(1/2x)} {x\,U_{k+1}(1/2x)},
&\text{if }r\le s,\\
\dfrac {U_s(1/2x)\,U_{k-r}(1/2x)}
{x\,U_{k+1}(1/2x)},
&\text{if }r\ge s.
\end{cases}
\end{multline}

Next let
$b_i=\la_i=1$ for all $i$. Combinatorially, we are talking
about paths consisting of up-, down-, and level-steps, that is, 
\index{Motzkin path}\index{path, Motzkin}\index{Motzkin, 
Theodore S.}Motzkin paths.
Again, since for this choice
of $b_i$'s and $\la_i$'s there is no difference between the
orthogonal polynomials and the corresponding associated orthogonal
polynomials arising from \eqref{CKeq:three-term},  
Example~\ref{CKex:Ch} tells us that
\begin{equation*} 
p_n(x)=Sp_n(x)=U_n\left(\frac {x-1} {2}\right). 
\end{equation*}
From \eqref{CKe3.43}, it then follows that
\begin{multline} \label{CKeq:Motz} 
\sum _{n\ge0} ^{}\big\vert\LL{(0,r)\to(n,s);
M\mid 0\le y\le k}\big\vert
\cdot x^n\\
=\begin{cases} \dfrac {
U_r\left(\frac {1-x} {2x}\right)\,U_{k-s}\left(\frac {1-x} {2x}\right)} 
{x\,U_{k+1}\left(\frac {1-x} {2x}\right)},
&\text{if }r\le s,\\
\dfrac {
U_s\left(\frac {1-x} {2x}\right)\,U_{k-r}\left(\frac {1-x} {2x}\right)}
{x\,U_{k+1}\left(\frac {1-x} {2x}\right)},
&\text{if }r\ge s.
\end{cases}
\end{multline}
\end{example}

\begin{example}\label{CKex:gambl}%
\em
The standard application of \eqref{CKeq:Motz} concerns the 
\index{gambler's ruin}{\it
gambler's ruin problem} (see also \cite[Ch.~XIV]{FellAA}): 
two players $A$ and $B$ have initially $a$
and $R-a$ dollars, respectively. They play several rounds, in each of
which the probability that player $A$ wins is $p_A$, the probability
that player $B$ wins is $p_B$, and the probability
that there is a tie is $p_T=1-p_A-p_B$.
If one player wins,
(s)he takes a dollar from the other.
If there is a tie, nothing happens. The play stops when one of the
players is bankrupt. What is the probability that
player $A$, say, goes bankrupt after $N$ rounds? 

\begin{figure}[h]
$$
\Gitter(12,5)(-1,0)
\Koordinatenachsen(12,5)(-1,0)
\Pfad(0,1),13141133444\endPfad
\DickPunkt(0,1)
\DickPunkt(11,0)
\PfadDicke{.5pt}
\Pfad(-2,4),11111111111111\endPfad
\Label\l{y=R-2}(-3,4)
\hbox{\hskip6cm}
$$
\caption{}
\label{CKfig1}
\end{figure}

By disregarding the last round (which is necessarily a round in which
$B$ wins),
this problem can be represented by a lattice path starting at 
$(0,a-1)$, ending at $(N-1,0)$, 
with steps $(1,1)$
(corresponding to player $A$ to win a round), $(1,-1)$
(corresponding to player $B$ to win a round), and
$(1,0)$ (corresponding to a tie), which does not pass below the
$x$-axis, and which does not pass above the
horizontal line $y=R-2$. For example, the lattice path in Figure~\ref{CKfig1}
corresponds to the play, where player~A starts with 2~dollar, player~B
starts with 4~dollar, the
outcome of the rounds is in turn TATBTTAABBBB (the letter $A$ symbolizing
a round where $A$ won, with an analogous meaning of the letter $B$,
and the letter $T$ symbolizing a tie), so that $A$ goes bankrupt
after $N=12$ rounds (while $B$ did not). 

If we assign the {\it
weight\/} $p_A$ to an up-step $(1,1)$, $p_B$ to a down-step $(1,-1)$,
and $p_T$ to a level-step $(1,0)$, then the probability of this play
is the product of the weights of all the steps of the path $P$
times $p_B$
(corresponding to the last round where $B$ wins and $A$ goes bankrupt;
in our example, it is $p_Tp_Ap_Tp_Bp_Tp_Tp_Ap_Ap_Bp_Bp_Bp_B$). If we
write $p(P)$ for the product of the weights of the steps of $P$, 
then, in order to solve 
the problem, we need to compute the sum $\sum _{P} ^{}p_Bp(P)$, where
the sum is over all the above described paths from $(0,a-1)$ to $(N-1,0)$.

Clearly, \eqref{CKeq:Motz} with $r=a-1$ and $s=0$ provides the
solution for the above problem, in terms of a generating function.
Since the zeroes of the Chebyshev polynomials are explicitly known,
one can apply partial fraction decomposition to obtain an explicit
formula for the coefficients in the generating function. If this
is carried out, then we get
\begin{multline} \label{CKeq:Motz2} 
\big\vert\LL{(0,r)\to(n,s);
M\mid 0\le y\le k}\big\vert\\
=\frac {2} {k+2}\sum _{j=1} ^{k+1}
\left(2\cos \frac {\pi j} {k+2}+1\right)^n\cdot
\sin\frac {\pi j(r+1)} {k+2}\cdot \sin\frac
{\pi j(s+1)} {k+2}. 
\end{multline}
\end{example}

\section{Further results for lattice paths in the plane}\label{CKsec:var}

In this section we collect various further results on the
enumeration of two-dimensional lattice paths, 
respectively pointers to further such results.

\medskip
The first set of results that we describe concerns lattice
paths in the plane integer lattice $\Z^2$ which consist of
steps from a finite set $\SS$ that contains steps of the form $(1,b)$.
Here, $b$ is some integer. Say,
\begin{equation} \label{CKeq:SS} 
\SS=\{(1,b_1),\,(1,b_2),\dots,(1,b_m)\}.
\end{equation}
We also assume that to each step $(1,b_j)$ there is associated a
weight $w_j\in\C$.

\index{Banderier, Cyril}Banderier and
\index{Flajolet, Philippe}Flajolet \cite{BaFlAA} completely solved
the exact and asymptotic enumeration of lattice paths consisting
of steps from $\SS$ obeying certain restrictions. We concentrate
here on the exact enumeration results.

The key object in their theory is the 
\index{characteristic polynomial}\index{polynomial, characteristic}{\it
characteristic polynomial\/} of the step set $\SS$,
\begin{equation} \label{CKeq:char}
P_\SS(u)=\sum_{j=1}^m w_ju^{b_j}.
\end{equation}
If we write $c=-\min_jb_j$ and $d=\max_jb_j$, then $P_\SS(u)$ can be
rewritten in the form
$$
P_\SS(u)=\sum_{j=-c}^d p_ju^{j},
$$
for appropriate coefficients $p_j$. Associated with the
characteristic polynomial is the 
\index{characteristic equation}\index{equation, characteristic}{\em
characteristic equation}
\begin{equation} \label{CKeq:chareq}
1-zP_\SS(u)=0, 
\end{equation}
or, equivalently,
\begin{equation} \label{CKeq:chareq2}
u^c-zu^cP_\SS(u)=u^c-z\sum_{j=0}^{c+d}p_{j-c}u^j=0.
\end{equation}
The form \eqref{CKeq:chareq2} has only non-negative powers in $u$,
and it shows that, counting multiplicity, there are $c+d$
solutions to the characteristic equation when $u$ is expressed
as a function in $z$. These $c+d$ solutions fall into two
categories; there are $c$ ``small branches'' $u_1(z),u_2(z),
\dots,u_c(z)$ satisfying
$$
u_j(z)\sim e^{2\pi i(j-1)/c}p_{-c}^{1/c}z^{1/c}
\quad \text{as }z\to0,
$$
and $d$ ``large branches" $u_{c+1}(z),u_{c+2}(z),
\dots,u_{c+d}(z)$ satisfying
$$
u_j(z)\sim e^{2\pi i(c+1-j)/d}p_{d}^{-1/d}z^{-1/d}
\quad \text{as }z\to0.
$$
One can show that there are functions $A(z)$ and $B(z)$ which are
analytic and non-zero at $0$ such that, in a neighbourhood of $0$,
\begin{align} \label{CKeq:small}
u_j(z)&=\om^{j-1}z^{1/c}A(\om^{j-1}z^{1/c}), \text{ with }\om=e^{2\pi
  i/c},
\quad j=1,2,\dots,c,\\
u_j(z)&=\varpi^{c+1-j}z^{-1/d}B(\varpi^{j-c-1}z^{1/d}), \text{ with }\varpi=e^{2\pi
  i/d},
\quad j=c+1,c+2,\dots,c+d.
\label{CKeq:large}
\end{align}

We are now in the position to state the 
enumeration results for lattice paths with steps from $\SS$
without further restriction. In the formulation, we use
$\ell(P)$ to denote the length of a path $P$, and $h(P)$
to denote the abscissa (height) of the end point of $P$.

\begin{theorem} \label{CKthm:BF1}
The generating function $\sum_P z^{\ell(P)}u^{h(P)}$
for lattice paths $P$ which start at the origin and consist of
steps from $\SS$ as given in \eqref{CKeq:SS} equals
\begin{equation} \label{CKeq:BF1a}
\GF{\LL{(0,0)\to (*,*);\SS};z^{\ell(\,.\,)}u^{h(\,.\,)}}
=\frac {1} {1-zP_\SS(u)}, 
\end{equation}
with $P_\SS(u)$ the characteristic polynomial of\/ $\SS$ given in
\eqref{CKeq:char}. Moreover, the generating function 
$\sum_P z^{\ell(P)}$ for those paths $P$ which end at height~$0$
equals
\begin{equation} \label{CKeq:BF1b}
\GF{\LL{(0,0)\to (*,0);\SS};z^{\ell(\,.\,)}}
=z\sum_{j=1}^c \frac {u_j'(z)} {u_j(z)}
=z\frac {d} {dz}\big(u_1(z)u_2(z)\cdots u_c(z)\big) ,
\end{equation}
where $u_1(z),u_2(z),\dots,u_c(z)$ are the small branches
given in \eqref{CKeq:small}. Finally, 
for $k<c$ the generating function 
$\sum_P z^{\ell(P)}$ for those paths $P$ which end at height~$k$
equals
\begin{equation} \label{CKeq:BF1c}
\GF{\LL{(0,0)\to (*,k);\SS};z^{\ell(\,.\,)}}
=z\sum_{j=1}^c \frac {u_j'(z)} {u_j^{k+1}(z)}
=-\frac {z} {k}\frac {d} {dz}
\Bigg(\sum_{j=1}^c u_j^{-k}(z)\Bigg),
\end{equation}
where again $u_1(z),u_2(z),\dots,u_c(z)$ are the small branches
given in \eqref{CKeq:small}, while for $k>-d$ it equals
\begin{equation} \label{CKeq:BF1d}
\GF{\LL{(0,0)\to (*,k);\SS};z^{\ell(\,.\,)}}
=-z\sum_{j=c+1}^{c+d} \frac {u_j'(z)} {u_j^{k+1}(z)}
=\frac {z} {k}\frac {d} {dz}
\Bigg(\sum_{j=c+1}^{c+d} u_j^{-k}(z)\Bigg),
\end{equation}
\end{theorem}

\begin{proof}
By elementary combinatorial principles, the generating function\break
$\sum_P z^{\ell(P)}u^{h(P)}$
for lattice paths $P$ which start at the origin and consist of
steps from $\SS$ is given by $\sum_{n\ge0}z^nP_\SS^n(u)$, which
equals \eqref{CKeq:BF1a}.

In order to determine the generating function 
$\sum_P z^{\ell(P)}$ for those lattice paths $P$ which end at
height $0$, we have to extract the coefficient of $u^0$ in
\eqref{CKeq:BF1a}. This can be achieved by
computing the contour integral
\begin{equation} \label{CKeq:cI}
\frac {1} {2\pi i}\int_C \frac {1} {1-zP_\SS(u)}\frac {du} {u}, 
\end{equation}
where $C$ is a contour encircling the origin in the positive
direction. One has to choose
$C$ so that, for sufficiently small $z$, the small branches lie within the
contour, while the large branches lie outside. Then, by the
residue theorem, only the small branches contribute to the integral
\eqref{CKeq:cI}. The residue at $u=u_j(z)$ equals
(assuming that, in addition, we have chosen $z$ so that all small
branches are different)
$$
\underset{u=u_j(z)}{\text{Res}}\left(\frac {1} {u\big(1-zP_\SS(u)\big)}\right)
=-\frac {1} {zu_j(z)P_\SS'(u_j(z))}.
$$
The integral in \eqref{CKeq:cI} equals the sum of these residues.
This sum simplifies to \eqref{CKeq:BF1b} since differentiation of
both sides of the characteristic equation \eqref{CKeq:chareq}\break shows
that $P'_\SS(u_j(z))^{-1}=-z^2u_j'(z)$ for all small branches $u_j(z)$.

The arguments for establishing \eqref{CKeq:BF1c} and \eqref{CKeq:BF1d}
are similar.
\end{proof}

The second set of results concerns lattice paths starting at the
origin with steps from $\SS$ which do not run below the $x$-axis.

\begin{theorem} \label{CKthm:BF2}
The generating function $\sum_P z^{\ell(P)}u^{h(P)}$
for lattice paths $P$ which start at the origin, consist of
steps from $\SS$ as given in \eqref{CKeq:SS}, and do not
run below the $x$-axis, equals
\begin{align} \notag
\GF{\LL{(0,0)\to (*,*);\SS\mid y\ge0};z^{\ell(\,.\,)}u^{h(\,.\,)}}
&=\frac {
\prod _{j=1} ^{c}(u-u_j(z))} {u^c(1-zP_\SS(u))}\\
&=-\frac {1} {p_dz}
\prod _{j=c+1} ^{c+d}\frac 1{(u-u_j(z))},
\label{CKeq:BF2a}
\end{align}
with $P_\SS(u)$ the characteristic polynomial of\/ $\SS$ given in
\eqref{CKeq:char}, and
$u_1(z),u_2(z),\dots,\break u_c(z)$ and 
$u_{c+1}(z),u_{c+2}(z),\dots,u_{c+d}(z)$ 
the small and large branches given in \eqref{CKeq:small} and \eqref{CKeq:large}.
In particular, the generating function 
$\sum_P z^{\ell(P)}$ for those paths $P$ which end at height~$0$
equals
\begin{align} \notag
\GF{\LL{(0,0)\to (*,0);\SS\mid y\ge0};z^{\ell(\,.\,)}}
&=\frac {(-1)^{c-1}}{p_{-c}z}
\prod _{j=1} ^{c}u_j(z))\\
&=\frac {(-1)^{d-1}} {p_dz}
\prod _{j=c+1} ^{c+d}\frac 1{u_j(z))}.
\label{CKeq:BF2b}
\end{align}
\end{theorem}

\begin{proof}
Here, we use the so-called 
\index{kernel method}\index{method, kernel}{\em kernel method\/}
(cf.\ e.g.\ \cite{BoJeAA}). Let $F(z,u)$ denote the generating function
on the left-hand side of \eqref{CKeq:BF2a}. 
Then we have
\begin{equation} \label{CKeq:F(z,u)} 
F(z, u) = 1+zP_\SS(u)F(z, u) - z[u^{<0}]\big(P_\SS(u)F(z, u)\big),
\end{equation}
where $[u^{<0}]G(z,u)$ means that in the series $G(z,u)$
all monomials $z^nu^m$ with $m\ge0$ are dropped.
For, any lattice path that is counted by $F(z,u)$ is either empty,
or it consists of a step
($zP_\SS(u)$ describes
the possibilities) added to a path, except that the steps that would
take the walk below
level $0$ are to be taken out (the operator $[u^{<0}]$ extracts
the terms to be taken out). Since $P_\SS(u)$ involves
only a finite number of negative powers, we may rewrite
\eqref{CKeq:F(z,u)} in the form
\begin{equation} \label{CKeq:F(z,u)2}
F(z,u)(1 - zP_\SS(u)) = 1 - z\sum_{k=0}^{c-1}
r_k (u)F_k (z),
\end{equation}
for some Laurent polynomials $r_k (u)$ that can be computed
from $P_\SS(u)$ via \eqref{CKeq:F(z,u)},
$$
r_k (u) = [u^{<0}](P_\SS(u)u^k) =
\sum_{j=-c}^{-k-1}p_ju^{j+k}.
$$
Here, $F_k(z)$ is the generating function 
$\sum_P z^{\ell(P)}$ for those paths $P$ which end at height~$k$.

In the current context, the factor $1-zP_\SS(u)$ on the left-hand side
of \eqref{CKeq:F(z,u)2} (which is identical with the left-hand side of
the characteristic equation \eqref{CKeq:chareq}) is called the 
\index{kernel}{\em kernel}. The idea of the kernel method is
to substitute $u=u_j(z)$, $j=1,2,\dots,c$ (that is, the small
branches) on both sides of \eqref{CKeq:F(z,u)2} so that the kernel
--- and thus the left-hand side --- vanishes. In this way,
we arrive at the system of equations
\begin{equation*} 
u_j^c(z) - z\sum_{k=0}^{c-1}u_j^c(z)
r_k (u_j(z))F_k (z)=0,\quad j=1,2,\dots,c.
\end{equation*}
This system of linear equations in the unknowns $F_0(z),F_1(z),
\dots,F_{c-1}(z)$ could now be solved. Alternatively,
we could observe that the expression
$$
u^c - z\sum_{k=0}^{c-1}u^c
r_k (u)F_k (z)
$$
is a polynomial in $u$ of degree $c$ with leading monomial $u^c$.
Its roots are exactly the small branches $u_j(z)$, $j=1,2,\dots,c$.
Hence, it factorizes as
\begin{equation} \label{CKeq:ufak} 
u^c - z\sum_{k=0}^{c-1}u^c
r_k (u)F_k (z)=
\prod _{j=1} ^{c}(u-u_j(z)).
\end{equation}
Extraction of the coefficient of $u^0$ on both sides gives immediately
$F_0(z)$, the generating function for the paths which end at height
$0$. This leads directly to \eqref{CKeq:BF2b}. The formula
\eqref{CKeq:BF2a} follows from \eqref{CKeq:F(z,u)2} and \eqref{CKeq:ufak}.
\end{proof}

Sometimes, the kernel method is also applicable if the the set
of steps $\SS$ is infinite. This is. for instance, the case for
\index{\L ukasiewicz path}\index{path, \L ukasiewicz}%
\index{\L ukasiewicz, Jan}{\it\L ukasiewicz paths}, which are paths
consisting of steps from
$\SS_L=\big\{(1,b):b\in\{-1,0,1,2,\dots\}\big\}$, which start
at the origin, return to the $x$-axis, never running below it.
In that case, the equation \eqref{CKeq:F(z,u)2} for the generating
function $\sum_P z^{\ell(P)}u^{h(P)}$ becomes
\begin{equation} \label{CKeq:GFL} 
F(z,u)\left(1 - \frac {z} {u(1-u)}\right) = 1 - zu^{-1}F_0 (z),
\end{equation}
where, as before, $F_0(z)$ is the generating function for those
paths which end at height $0$ (that is, return to the $x$-axis).
Here, the kernel is
$$
1 - \frac {z} {u(1-u)},
$$
and it vanishes for $u(z)=\frac {1-\sqrt{1-4z}} {2}$.
If this is substituted in \eqref{CKeq:GFL}, then we obtain
$$
\GF{\LL{(0,0)\to (*,0);\SS_L\mid y\ge0};z^{\ell(\,.\,)}}
=F_0(z)=\frac {1-\sqrt{1-4z}} {2z},
$$
the \index{number, Catalan}\index{Catalan
number}\index{Catalan, Eug\`ene Charles}Catalan number generating
function \eqref{CKe1.13}. Hence, also \L ukasiewicz paths of
length $n$ are enumerated by the Catalan number $C_n=\frac {1}
{n+1}\binom {2n}n$.

To conclude this topic, it must be mentioned that
\index{Banderier, Cyril}Banderier and
\index{Gittenberger, Bernhard}Gittenberger \cite{BaGiAA}
have extended the analyses of \cite{BaFlAA} to also
include the area statistics.

\medskip
A very cute problem, which arose in a probabilistic context
around 2000, is the problem of counting paths (walks)
in the \index{slit plane}\index{plane, slit}{\it slit plane}.
The slit plane is the integer lattice $\Z^2$ where one
has taken out the half-axis $\{(k,0):k\le 0\}$.
Investigation of this problem started with the conjecture
that the number of paths in the slit plane which start
at $(1,0)$ and do $2n+1$ horizontal or vertical unit steps
(in the positive or in the negative direction) is given by the Catalan
number $C_{2n+1}$. This conjecture was proved by
\index{Bousquet-M\'elou, Mireille}Bousquet-M\'elou and 
\index{Schaeffer, Gilles}Schaeffer in \cite{BoScAA},
but they provide much stronger and more general results 
on the enumeration of lattice paths in the slit plane
in that paper. When it is not possible to find exact
formulas, then the focus is on the nature of the
generating function, whether it be algebraic or not,
D-finite or not, etc. Methods used are the cycle
lemma and the kernel method.

\medskip
An innocent looking three-candidate ballot problem
stands at the beginning of another long line of investigation:
Let 
$E_1,E_2,E_3$ be candidates in an election, $E_1$ receiving $e_1$
votes, $E_2$ receiving $e_2$ votes, and $E_3$ receiving $e_3$ votes,
$e_1\ge \max\{e_2,e_3\}$. How many ways of counting the votes are
there such that at any stage during the counting candidate $E_1$ has
at least as many votes as $E_2$ and at least as many votes as 
$E_3$? In lattice path formulation
this means to count all simple lattice paths in $\Z^3$ from the
origin to $(e_1,e_2,e_3)$ staying in the region
$\{(x_1,x_2,x_3):x_1\ge x_2\text{ and }x_1\ge x_3\}$.%
\footnote{It seems that this is a non-planar lattice path problem, contradicting
the title of the section
However, the problem can be translated into a two-dimensional problem,
see \cite{BousAH}.}
We state the
result below. Solutions were given by \index{Kreweras, Germain}Kreweras 
\cite{KrewAB} and
\index{Niederhausen, Heinrich}Niederhausen \cite{NiedAJ}, see also
\index{Gessel, Ira Martin}Gessel \cite{GessAF}. 
This line of research was picked up later by
\index{Bousquet-M\'elou, Mireille}Bousquet-M\'elou \cite{BousAH}
who showed, again with the help of the kernel method, that the generating function 
of these ``Kreweras walks" is algebraic.
It must be pointed out that this counting problem is a
``non-example" for the reflection principle (see Section~\ref{CKsec:refl}), that is, the
reflection principle does not apply.
The reason is that, if one tries to set it up for application of the 
reflection principle, then one realizes that the nice property that for
permutations other than the identity permutation some hyperplane has
to be touched would fail.

\begin{theorem} \label{CKT7.2}
Let $e_1\ge\max\{e_2,e_3\}$. The number of all lattice paths in $\Z^3$
from $(0,0,0)$ to $(e_1,e_2,e_3)$ subject to $x_1\ge x_2$  and $x_1\ge
x_3$ is given by
\begin{multline} \label{CKe7.7}
\vv{\LL{(0,0,0)\to (e_1,e_2,e_3)\mid x_1\ge \max\{x_2,x_3\}}}\\
=\binom {e_1+e_2+e_3}{e_1,e_2,e_3} - \frac {e_2+e_3} {1+e_1}\binom
{e_1+e_2+e_3} {e_1,e_2,e_3}\\
+\sum _{i,j\ge 1} ^{}(-1)^{i+j}\frac
{(e_1+e_2+e_3)!\,(2i+2j-2)!\,(i+j-2)!}
{i!\,(e_3-i)!\,j!\,(e_2-j)!\,(2i-1)! \,(2j-1)!\,(i+j+e_1)!}.
\end{multline}
In particular, if $e_1=e_2$ this number simplifies to
\begin{multline} \label{CKe7.8}
\vv{\LL{(0,0,0)\to (e_1,e_1,e_3)\mid x_1\ge \max\{x_2,x_3\}}}\\
=2^{2e_3+1}\frac {(2e_1+e_3)!\,(2e_1-2e_3+1)!} {(2e_1+2)!\,e_3!\,(e_1-e_3)!^2}.
\end{multline}
\end{theorem}

We come to a relatively recent research field:
the enumeration of walks in the quarter plane.
The question that was posed is: given a particular
step set, can one find an explicit formula for 
the corresponding generating function, and, if not,
is the generating function rational, algebraic, D-finite,
or neither? For ``small" step sets, the analysis is
now complete, due to work by
   \index{Bousquet-M\'elou, Mireille}Bousquet-M\'elou and 
\index{Mishna, Marni}Mishna \cite{BoMiAA},
by \index{Bostan, Alin}Bostan and
\index{Kauers, Manuel}Kauers \cite{BoKaAA},
and by \index{Bostan, Alin}Bostan, 
\index{Kurkova, I.}Kurkova,
\index{Raschel, Kilian}Raschel and 
\index{Salvy, Bruno}Salvy \cite{BoKRAA,BoRSAA}.
However, there is not yet a good understanding
how, or whether at all, one can decide from the
step set that the generating function has one of
the above mentioned properties.

\medskip
The last topic that I mention here is the connection
between Dyck and Schr\"oder path enumeration on the one hand, and
Hilbert series for diagonal harmonics and Macdonald polynomials
on the other hand.
This topic would by itself require a whole chapter.
We refer the reader to the survey \cite{HaglAA}
and the references therein. One of the most intriguing
combinatorial problems originating from the
investigations in this area is new statistics for
Dyck paths, most prominently ``bounce" and ``dinv."
It has been shown (algebraically) that
the pair (bounce. area) is equally distributed as
(area, bounce), and the same for area and dinv.
However, although much effort has been put into it,
so far nobody could come up with a direct combinatorial
reason (in the best case: a bijection) why this
symmetry holds.

\section{Non-intersecting lattice paths}
\label{CKsec:nonint}

The technique of non-intersecting lattice paths is a powerful
counting method. We have already seen its effectiveness in
Section~\ref{CKs1.6}. Originally, non-intersecting paths
arose in matroid theory, in the work of
\index{Lindstr\"om, Bernt}Lindstr\"om \cite{LindAA}.
Lindstr\"om's result was rediscovered
(not always in its most general form) in the 1980s at about the same
time in three different
communities, not knowing of each other at that time: in statistical
physics by 
\index{Fisher, Michael Ellis}Fisher \cite[Sec.~5.3]{FishAA} in order to apply it to
the analysis of vicious walkers as a model of wetting and melting, 
in combinatorial chemistry by 
\index{John, Peter}John and \index{Sachs, Horst}Sachs \cite{JoSaAB} and
\index{Gronau, H.-D. O. F.}Gronau, 
\index{Just, W.}Just, 
\index{Schade, W.}Schade, 
\index{Scheffler, P.}Scheffler and 
\index{Wojciechowski, J.}Wojciechowski \cite{GrJSAA}
in order to compute Pauling's bond order
in benzenoid hydrocarbon molecules, and in enumerative combinatorics
by 
\index{Gessel, Ira Martin}Gessel and 
\index{Viennot, Xavier G\'erard}Viennot \cite{GeViAA,GeViAB} in order to count
tableaux and plane partitions. 
It must however be
mentioned that in fact the same idea appeared even earlier in work by
\index{Karlin, Samuel}Karlin and 
\index{McGregor, James}McGregor \cite{KaMGAB,KaMGAC} in a probabilistic
framework, as well as that the so-called 
\index{Slater, John Clarke}``Slater determinant"
in quantum mechanics (cf.\ \cite{SlatZY} and \cite[Ch.~11]{SlatZZ}) 
may qualify as an ``ancestor" of the 
\index{determinant}determinantal formula of
Lindstr\"om.
Since then, many more applications
have been found, particularly in plane partition and rhombus
tiling enumeration, see e.g.\ \cite{BretAG,CiEKAA,FiscAA,StemAE}
and Chapter~[Tiling Enumeration by Jim Propp] for more
information on this topic.

We devote this section to developing the
theory of non-intersecting lattice paths and give some
sample applications. This 
will be continued in Section~\ref{CKsec:turn}, where we
give results on the enumeration of non-intersecting lattice paths
in the plane with respect to turns.

\medskip
The most general version of the non-intersecting path
theorem (\cite[Lemma~1]{LindAA}, \cite[Theorem~1]{GeViAB})
is formulated for paths in a 
\index{graph, directed}\index{directed graph}{\em directed graph}.
Let $G$ be a directed graph with vertices $V$ and (directed) edges
$E$.
A \index{path}{\em path} (actually, the usual notion in graph theory is 
\index{walk}{\em walk}) in $G$
is a sequence $v_0,v_1,\dots,v_m$ of vertices, for some $m$, such
that there is an edge from $v_i$ to $v_{i+1}$, $i=0,1,\dots,m-1$. 
We denote the set of all paths in $G$ from $A$ to $E$ by $L_G(A\to
E)$. The
directed graph $G$ is called \index{acyclic graph}\index{graph,
acyclic}{\em acyclic} if there is no non-trivial closed path in $G$, i.e., if
there is no path that starts and ends in the same vertex other than a
zero-length path.

The central definition is that a family 
$\P=(P_1,P_2,\dots,P_n)$ of paths $P_i$ in $G$ is called
{\em non-intersecting\/} if no two
paths of $\P$ have a vertex in common. 
Otherwise $\P$ is called {\em
\index{intersecting lattice paths}\index{lattice paths,
intersecting}intersecting\/}. 
In the context of lattice path enumeration, the graph $G$ comes
from a lattice. In many examples, the vertices of $G$ are
the lattice points $\Z^2$ in the plane, and the edges of $G$
connect a point $(i,j)$ to $(i+1,j)$, respectively a point
$(i,j)$ to $(i,j+1)$.
Figure~\ref{CKF2.1Aa} displays a family
of non-intersecting lattice paths in this sense, 
Figure~\ref{CKF2.1Ab} a family of
intersecting lattice paths.
(It is very important
to note that, in the geometric
realization of paths as piecewise linear trails, the corresponding trails 
may very well have common points, but never in starting and end
points of steps, see Figure~\ref{CKF2.1B} for such an example. In
particular, non-intersecting lattice paths may even cross each other
in the geometric visualization.)

\begin{figure}
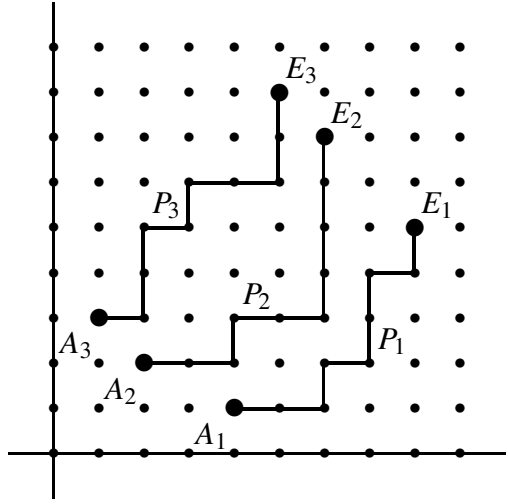
 
$$
{
\Gitter(10,10)(0,0)
\Koordinatenachsen(10,10)(0,0)
\Pfad(1,3),122121122\endPfad
\Pfad(2,2),112112222\endPfad
\Pfad(4,1),11212212\endPfad
\DickPunkt(1,3)
\DickPunkt(2,2)
\DickPunkt(4,1)
\DickPunkt(5,8)
\DickPunkt(6,7)
\DickPunkt(8,5)
\Label\lu{A_3}(1,3)
\Label\lu{A_2}(2,2)
\Label\lu{A_1}(4,1)
\Label\ro{E_3}(5,8)
\Label\ro{E_2}(6,7)
\Label\ro{E_1}(8,5)
\Label\ro{P_3}(2,5)
\Label\ro{P_2}(4,3)
\Label\ro{P_1}(7,2)
\hskip5cm
}
$$
\caption{Family of non-intersecting paths}
\label{CKF2.1Aa}
\end{figure}

\begin{figure} 
$$
{
\Gitter(10,10)(0,0)
\Koordinatenachsen(10,10)(0,0)
\Pfad(1,3),122121122\endPfad
\Pfad(2,2),122211122\endPfad
\Pfad(4,1),11221222\endPfad
\DickPunkt(1,3)
\DickPunkt(2,2)
\DickPunkt(4,1)
\DickPunkt(5,8)
\DickPunkt(6,7)
\DickPunkt(7,6)
\Label\lu{A_3}(1,3)
\Label\lu{A_2}(2,2)
\Label\lu{A_1}(4,1)
\Label\ro{E_3}(5,8)
\Label\ro{E_2}(6,7)
\Label\ro{E_1}(7,6)
\Label\ro{P_3}(1,4)
\Label\ro{P_2}(3,3)
\Label\ro{P_1}(7,3)
\hskip5cm
}
$$
\caption{Family of intersecting paths}
\label{CKF2.1Ab}
\end{figure}

\begin{figure}
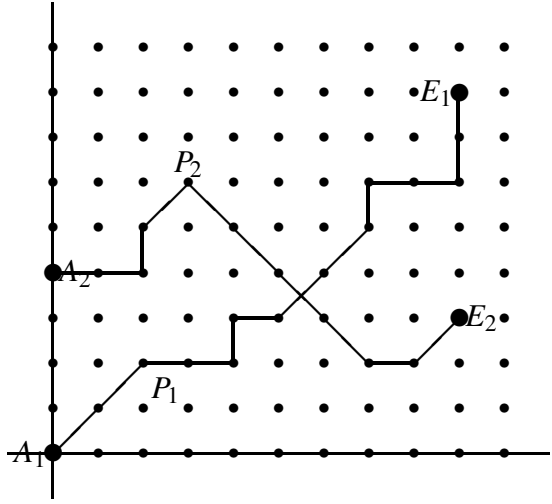
 
$$
\Gitter(11,10)(0,0)
\Koordinatenachsen(11,10)(0,0)
\Pfad(0,0),3311213321122\endPfad
\Pfad(0,4),1123444413\endPfad
\DickPunkt(0,0)
\DickPunkt(9,8)
\DickPunkt(0,4)
\DickPunkt(9,3)
\Label\l{A_1}(0,0)
\Label\l{E_1}(9,8)
\Label\r{A_2}(0,4)
\Label\r{E_2}(9,3)
\Label\ru{P_1}(2,2)
\Label\o{P_2}(3,6)
\hskip5.5cm
$$
\caption{Non-intersecting lattice paths may even cross}
\label{CKF2.1B}
\end{figure}

Returning to the general setup,
we furthermore assume that to any edge $e$ in the graph $G$ 
there is assigned a weight $w(e)$
(an element in some commutative ring $\mathcal R$).
The \index{weight}{\em weight\/} of a path $P$ is the product 
$w(P)=\prod _{e} ^{}w(e)$, where the product is over all edges $e$ of the
path $P$. 
The weight $w(\P)$
of a family $\P=(P_1,P_2,\dots,P_n)$ of paths is defined as the product of all
the weights of paths in the family,
$w(\P)=w((P_1,P_2,\dots,P_n))=\prod _{i=1} ^{n} w(P_i)$.

Given two sequences $\A=(A_1,A_2,\dots,A_n)$ and $\E=(E_1,E_2,\dots,E_n)$ 
of vertices of $G$, we write $L_G(\A\to \E)$ for the set of all
families $(P_1,P_2,\dots,P_n)$ of paths, where $P_i$ runs from $A_i$
to $E_i$, $i=1,2,\dots,n$, whereas $L_G(\A\to \E\mid \text
{non-intersecting})$ denotes the subset of families of non-intersecting
paths.

We need one more piece of notation.
Given a permutation $\si\in\mathfrak S_n$ and
a vector $\mathbf v=(v_1,v_2,\dots,v_n)$, by $\mathbf v_\si$ we mean
$(v_{\si(1)},v_{\si(2)},\dots ,v_{\si(n)})$.
We are now in the position to state and prove the main theorem
on non-intersecting paths, due to 
\index{Lindstr\"om, Bernt}Lindstr\"om \cite[Lemma~1]{LindAA}.

\begin{theorem} \label{CKT2.3}%
Let $G$ by a directed, acyclic graph, and
let $\A=(A_1,A_2,\break\dots,A_n)$ and\/ $\E=(E_1,E_2,\dots,E_n)$ be sequences
of vertices in $G$.
Then
\begin{multline} \label{CKe2.5a}
\sum _{\si\in\mathfrak S_n} ^{}(\sgn\si)\cdot
\GF{L_G(\A_\si\to\E\mid \text{\em non-intersecting});w}\\=
\det_{1\le i,j\le n}\big(\GF{L_G(A_j\to E_i);w}\big).
\end{multline}
\end{theorem}

\begin{proof}
By expanding the determinant on the right-hand side of \eqref{CKe2.5a}, 
we obtain
\begin{align} \notag
\det_{1\le i,j\le n}\big(\GF{L_G(A_j\to E_i);w}\big)&=\sum _{\si\in\mathfrak S_n}
^{} \sgn\si \prod _{i=1} ^{n}\GF{L_G(A_{\si(i)}\to E_i);w}\\
\label{CKe2.2}
&=\underset{\si\in\mathfrak S_n,\,\P\in L_G(\A_\si\to\E)}{\sum _{(\si,\P)} ^{}}
\sgn\si\,w(\P).
\end{align}
The sum in \eqref{CKe2.2} expresses the determinant in \eqref{CKe2.5a} as a
generating function for pairs $(\si,\P)$ in the set
\begin{equation} \label{CKe2.3}
\bigcup _{\si\in\mathfrak S_n} ^{}L_G(\A_\si\to\E). 
\end{equation}

We now define a sign-reversing, weight-preserving involution
$\ph$ on the set of
all pairs $(\si,\P)$ in \eqref{CKe2.3} with the property that $\P$ is
intersecting. 
Sign-reversing means that if
${\ph}((\si,\P))=(\si_{\ph},\P_{\ph})$ then 
$\sgn\si=-\sgn\si_{\ph}$, while weight-preserving means that
$w(\P)=w(\P_{\ph})$. Suppose that we had already constructed
such a $\ph$. Then, in the sum \eqref{CKe2.2}, all
contributions of pairs $(\si,\P)$ in \eqref{CKe2.3} where $\P$ is
intersecting would cancel. Only contributions of pairs $(\si,\P)$ in
\eqref{CKe2.3} where $\P$ is non-intersecting would survive,
establishing \eqref{CKe2.5a}.

Next we construct the sign-reversing, weight-preserving involution $\ph$. Let
$(\si,\P)$ be in $L_G(\A_\si\to\E)$ where $\P$ is intersecting. In the
left-hand picture of Figure~\ref{CKF2.1} an example is shown with
$G$ the directed graph corresponding to the integer lattice
$\Z^2$, $n=3$, and $\si=213$.
Among all pairs of paths with a common point,
choose the lexicographically largest, say $(P_i,P_j)$, $i<j$,
and among all common points of
that pair choose the last along the paths. 
(It does not matter on which of the two
paths of the pair we choose the last common point since the
graph $G$ is acyclic.) Denote
this common point by $M$. In our example, the common points
between paths are
$(3,2)$, $(3,3)$, $(4,5)$. The lexicographically largest pair
of paths with common points is $(P_2,P_3)$. The last common point
of this pair is $M=(4,5)$. 

Returning to the general construction of the involution $\ph$,
we now interchange the
initial portions of $P_i$ and $P_j$ up to $M$. To be more precise, we
form the new paths
$$P'_i=[\text{subpath of $P_j$ from $A_{\si(j)}$ to $M$ joined with
subpath of $P_i$ from $M$ to $E_i$}]$$
and
$$P'_j=[\text{subpath of $P_i$ from $A_{\si(i)}$ to $M$ joined with
subpath of $P_j$ from $M$ to $E_j$}].$$
Then we define
\begin{align}\notag
\ph\big((\si,\P)\big)&
=\ph\big((\si,(P_1,\dots,P_i,\dots,P_j,\dots,P_n))\big)\\
\notag
&=(\si\circ(ij),(P_1,\dots,P'_i,\dots,P'_j,\dots,P_n)),
\end{align}
where $(i,j)$ denotes the transposition interchanging $i$ and $j$. The
right-hand picture in Figure~\ref{CKF2.1} shows what is obtained by this
operation in our example.
\begin{figure}
$$
\Einheit.55cm
\Gitter(8,8)(0,0)
\Koordinatenachsen(8,8)(0,0)
\SPfad(3,1),222211\endSPfad
\Pfad(1,2),111121\endPfad
\Pfad(1,3),111222\endPfad
\DickPunkt(3,1)
\DickPunkt(1,2)
\DickPunkt(1,3)
\DickPunkt(6,3)
\DickPunkt(5,5)
\DickPunkt(4,6)
\Kreis(4,5)
\Label\u{A_1}(3,1)
\Label\lu{A_2}(1,2)
\Label\l{A_3}(1,3)
\Label\ro{E_1}(6,3)
\Label\ro{E_2}(5,5)
\Label\ro{E_3}(4,6)
\Label\lo{M}(4,5)
\Label\u{P_1}(5,2)
\Label\lu{P_2}(3,5)
\Label\r{P_3}(4,4)
\Label\u{\si=213}(3.5,-1)
\hskip5cm
\raise2.5cm\hbox{$\longleftrightarrow$\hskip1cm}
\Gitter(8,8)(0,0)
\Koordinatenachsen(8,8)(0,0)
\SPfad(1,3),111221\endSPfad
\Pfad(1,2),111121\endPfad
\Pfad(3,1),222212\endPfad
\DickPunkt(3,1)
\DickPunkt(1,2)
\DickPunkt(1,3)
\DickPunkt(6,3)
\DickPunkt(5,5)
\DickPunkt(4,6)
\Kreis(4,5)
\Label\u{A_1}(3,1)
\Label\lu{A_2}(1,2)
\Label\l{A_3}(1,3)
\Label\ro{E_1}(6,3)
\Label\ro{E_2}(5,5)
\Label\ro{E_3}(4,6)
\Label\lo{M}(4,5)
\Label\u{P_1}(5,2)
\Label\lu{P'_3}(3,5)
\Label\r{P'_2}(4,4)
\Label\u{\si'=231}(3.5,-1)
\hskip3.8cm
$$
\caption{}
\label{CKF2.1}
\end{figure}
The image $\ph((\si,\P))$ is again an element of the set in
\eqref{CKe2.3} since the new permutation of the starting points of $\P$
is exactly $\si\circ(ij)$. Moreover,
$(P_1,\dots,P'_i,\dots,P'_j,\dots,P_n)$ is intersecting since $P'_i$
and $P'_j$ are. From all this it is obvious that when $\ph$ is
applied to $\ph((\si,\P))$ we arrive back at $(\si,\P)$. Hence,
$\ph$ is an involution. Since $\si$ and $\si\circ(ij)$ differ in
sign, $\ph$ is sign-reversing. Finally, since the total (multi)set
of edges in the path families does not change under application of
$\ph$, the map $\ph$ is also weight-preserving. This finishes
the proof.
\end{proof}

\medskip
The most frequent situation in which the general result in
Theorem~\ref{CKT2.3} is applied is the one where
non-intersecting paths
can only occur if the starting and end points are connected via
the identity permutation, that is, if $(P_1,P_2,\dots,P_n)$
can only be non-intersecting if $P_i$ connects $A_i$ with $E_i$,
$i=1,2,\dots,n$. In that situation, Theorem~\ref{CKT2.3}
simplifies to the following result.

\begin{corollary} \label{CKT2.2a}%
Let $G$ be a directed, acyclic graph, and
let $\A=(A_1,A_2,\break\dots,A_n)$ and\/ $\E=(E_1,E_2,\dots,E_n)$ be sequences
of vertices in $G$ such that the only permutation $\si$ that
allows for a family $(P_1,P_2,\dots,P_n)$ of non-intersecting
paths such that $P_i$ connects $A_{\si(i)}$ with $E_i$, $i=1,2,\dots,n$, 
is the identity permutation.
Then the generating function $\sum _{\P} ^{}w(\P)$ for
families $\P=(P_1,P_2,\dots,P_n)$ of non-intersecting
paths, where $P_i$ is a path running from
$A_i$ to $E_i$, $i=1,2,\dots,n$, is given by 
\begin{equation} \label{CKe2.4a}
\GF{L_G(\A\to\E\mid \text{\em non-intersecting});w}=
\det_{1\le i,j\le n}\big(\GF{L_G(A_j\to E_i);w}\big).
\end{equation}
\end{corollary}

The standard application of Corollary~\ref{CKT2.2a} concerns
\index{semistandard tableau}\index{tableau, semistandard}{\em semistandard
tabl\-eaux}. These are important objects particularly in the
representation theory of the general and the special linear
groups, cf.\ \cite{SagaAQ}. 

Let $\La=(\la_1,\la_2,\dots,\la_n)$ and $\Mu=(\mu_1,\mu_2,\dots,\mu_n)$
be $n$-tuples of integers such that
{\refstepcounter{equation}\label{CKe2.11}}
\alphaeqn
\begin{gather}
\lambda
_1\ge\lambda
_2\ge\lambda
_3\ge\dots\ge\lambda
_n ,\label{CKe2.11a}\\
\mu_1\ge\mu_2\ge\mu_3\ge\dots\ge\mu_n,\label{CKe2.11b}\\
\text {and }\la_i\ge\mu_i\text{ for all }i.\label{CKe2.11c}
\end{gather}
\reseteqn
A semistandard tableau $T$ of shape $\La/\Mu$ is
an array of integers 
\begin{equation} \label{CKe2.6}
\hskip1cm
\begin{matrix} 
&&&\pi_{1,\mu_1+1}&\multicolumn{3}{c}{\dotfill} &\pi_{1,\la_1}\\
&&\pi_{2,\mu_2+1}\quad \dots&\pi_{2,\mu_1+1}&\multicolumn{2}{c}{\dotfill} 
&\pi_{2,\la_2}\\
&\iddots&&\vdots&&\iddots\\
\pi_{n,\mu_n+1}&\multicolumn{3}{c}{\dotfill} &\pi_{n,\la_n}
\end{matrix}
\end{equation}
such that entries along rows are weakly increasing and entries
along columns are strictly increasing.
A semistandard tableau of shape $(7,6,6,4)/(3,3,1,0)$
is shown in Figure~\ref{CKF2.4}.a.
(The lower and upper bounds on the entries displayed to the left
and right of the tableau should be ignored at this point.)

Let $\mathbf a=(a_1,a_2\dots,a_n)$ and $\mathbf b=(b_1,b_2,\dots,b_n)$ be
sequences of integers such that
{\refstepcounter{equation}\label{CKe2.12}}
\alphaeqn
\begin{gather}
a_1\le a_2\le \dots\le a_n\hphantom{.}\label{CKe2.12a}\\
b_1\le b_2\le \dots\le b_n,\label{CKe2.12b}\\
\text {and }a_i\ge b_i\text{ for all }i.\label{CKe2.12c}
\end{gather}
\reseteqn

We claim that {\em semistandard tableaux of shape $\La/\Mu$ where the
entries in row $i$ are at most $a_i$ and at least $b_i$ bijectively
correspond to families $(P_1,P_2,\dots,P_n)$ of non-intersecting
lattice paths $P_i$, where $P_i$ runs from $(\mu_i-i,b_i)$ to
$(\la_i-i,a_i)$.}

This is seen as follows. Let $\pi$ be a semistandard tableau
of shape $\La/\Mu$ where the entries in row $i$ are at most $a_i$ and 
at least $b_i$. 
\begin{figure}
$$
\underset{\text {\small a. Semistandard tableau}}
{
\begin{matrix} 
-2\le&&&&1&1&4&5&\le 7\hphantom{0}\\
-1\le&&&&5&5&7&&\le8\hphantom{0}\\
\hphantom{-}0\le&&2&5&6&9&9&&\le 9\hphantom{0}\\
\hphantom{-}1\le&1&5&6&8&&&&\le10
\\
\mathstrut\\
\end{matrix}
}
$$
$$\updownarrow$$
$$
\underset{\text {\small b. Family of non-intersecting lattice paths}}
{
\Einheit0.5cm
\Gitter(7,11)(-4,-2)
\Koordinatenachsen(7,11)(-4,-2)
\Pfad(2,-2),2221122212122\endPfad
\Pfad(1,-1),222222112212\endPfad
\Pfad(-2,0),22122212122211\endPfad
\Pfad(-4,1),1222212122122\endPfad
\DickPunkt(2,-2)
\DickPunkt(1,-1)
\DickPunkt(-2,0)
\DickPunkt(-4,1)
\DickPunkt(6,7)
\DickPunkt(4,8)
\DickPunkt(3,9)
\DickPunkt(0,10)
\Label\ro{P_4}(-3,3)
\Label\ro{P_3}(-1,3)
\Label\ro{P_2}(2,4)
\Label\ru{P_1}(4,3)
\Label\ro{1}(2,1)
\Label\ro{1}(3,1)
\Label\ro{4}(4,4)
\Label\ro{5}(5,5)
\Label\ro{5}(1,5)
\Label\ro{5}(2,5)
\Label\ro{7}(3,7)
\Label\ro{2}(-2,2)
\Label\ro{5}(-1,5)
\Label\ro{6}(0,6)
\Label\ro{9}(1,9)
\Label\ro{9}(2,9)
\Label\ro{1}(-4,1)
\Label\ro{5}(-3,5)
\Label\ro{6}(-2,6)
\Label\ro{8}(-1,8)
\Label\ro{}(1,-4)
\hskip1.1cm
\vphantom{\sum_{i=0}^{dsf}}
\mathstrut}
$$
\caption{}
\label{CKF2.4}
\end{figure}
The semistandard tableau $\pi$ is
mapped to a family of lattice paths
by associating the $i$-th row of $\pi$ with a path
$P_i$ from $(\mu_i-i,b_i)$ to $(\la_i-i,a_i)$ where
the entries in the $i$-th row are interpreted
as heights of the horizontal steps in the path
$P_i$. Thus, from $\pi$ we obtain the family $\P=(P_1,\dots,P_n)$
of lattice paths. The lower picture of Figure~\ref{CKF2.4} 
displays the family of lattice paths that in this way
results from the array displayed in the upper
picture of Figure~\ref{CKF2.4}.

Clearly, the property that the columns of $\pi$ are strictly
increasing translates into the condition that $(P_1,P_2,\dots,P_n)$
is non-intersecting. 

By applying Corollary~\ref{CKT2.2a} to this situation, we obtain
the following enumeration result.

\begin{theorem} \label{CKT2.4}
Let $\La=(\la_1,\la_2,\dots,\la_n)$ and
$\Mu=(\mu_1,\mu_2,\dots,\mu_n)$ be sequences of integers satisfying
\eqref{CKe2.11}. Let $\mathbf a=(a_1,a_2,\dots,a_n)$ and $\mathbf
b=(b_1,b_2,\break\dots,b_n)$ be sequences of integers satisfying
\eqref{CKe2.12}.
Then the number of all semistandard tableaux
$\pi$ of shape $\La/\Mu$
where the entries in row $i$ are at most $a_i$ and at least $b_i$
equals
\begin{equation} \label{CKe2.13} 
\det_{1\le i,j\le n}\(\binom{a_i-b_j+\la_i-\mu_j-i+j}
{\la_i-\mu_j-i+j}\).
\end{equation}
More generally,
the generating function $\sum _{\pi} ^{}q^{n(\pi)}$ for the same set
of semistandard tableaux $\pi$, where $n(\pi)$ denotes the sum of all
entries of $\pi$, equals
\begin{equation} \label{CKe2.14} 
\det_{1\le i,j\le n}\(q^{b_j(\la_i-\mu_j-i+j)}
\qbinom{a_i-b_j+\la_i-\mu_j-i+j}
{\la_i-\mu_j-i+j}q\).
\end{equation}
\end{theorem}

If the shape $\La/\Mu$ is a straight shape and the
bounds $\mathbf a$ and $\mathbf b$ are constant, that is,
if, say, $\Mu=(0,0,\dots,0)$, $\mathbf b=(1,1,\dots,1)$, 
and $\mathbf a=(a,a,\dots,a)$, 
then the above determinants
can be evaluated in closed form. Rewritten appropriately, the result is
the \index{formula, hook-content}\index{hook-content formula}%
{\it hook-content formula}.
(We refer the reader to \cite[Sec.~3.10]{SagaAQ} or \cite[Cor.~7.21.6]{StanBI}
for unexplained terminology).

\begin{theorem} \label{CKT2.4H}
Let $\La=(\la_1,\la_2,\dots,\la_n)$ 
be a sequence of non-negative integers satisfying
\eqref{CKe2.11}. 
Then the number of all semistandard tableaux
$\pi$ of shape $\La$ with
entries in row $i$ being positive integers 
at most $a$
equals
\begin{equation} \label{CKe2.13H} 
\prod _{\rho} ^{}\frac {a+c(\rho)} {h(\rho)},
\end{equation}
where the product is over all cells $\rho$ in the Ferrers diagram
of the partition $\La$, $c(\rho)$ is
the content of the cell $\rho$, and $h(\rho)$ is
the hook-length of the cell $\rho$.
More generally,
the generating function $\sum _{\pi} ^{}q^{n(\pi)}$ for the same set
of semistandard tableaux $\pi$ equals
\begin{equation} \label{CKe2.14H} 
q^{\sum_{i=1}^n i\la_i}
\prod _{\rho} ^{}\frac {1-q^{a+c(\rho)}} {1-q^{h(\rho)}}.
\end{equation}
\end{theorem}

By introducing more general weights, in the same way one can
provide combinatorial proofs for the 
\index{Jacobi--Trudi identity}\index{identity, Jacobi--Trudi}%
\index{Jacobi, Carl Gustav Jacob}Jacobi--Trudi-type
identities for 
\index{Schur function}\index{function, Schur}Schur functions 
(cf.\ \cite[Sec.~4.5]{SagaAQ}),
respectively formulas for so-called 
\index{flagged Schur function}\index{Schur function, flagged}%
\index{function, flagged Schur}flagged Schur functions,
originally introduced by
\index{Lascoux, Alain}Lascoux and
\index{Sch\"utzenberger, Marcel-Paul}Sch\"utzenberger \cite{LaScAA}, see also 
\index{Wachs, Michelle L.}\cite{WachAA}.

If in an array \eqref{CKe2.6} one also introduces a relationship between
the first row and the last row, then one is led to define so-called
\index{cylindric partition}\index{partition, cylindric}{\em cylindric
partitions}, as was done by
\index{Gessel, Ira Martin}Gessel and
\index{Krattenthaler, Christian}Krattenthaler \cite{GeKrAA}.
Also in that theory, non-intersecting paths play an essential role.

\medskip
It may seem that the general result in
Theorem~\ref{CKT2.3} is a very artificial statement, 
perhaps being of no use.
However, even this result does have several applications.
For example, the most elegant proof of 
the determinant formula
for higher-dimensional path counting under a general two-sided
restriction (\cite{SulaAC}; see Section~\ref{CKs7.5},
Theorem~\ref{CKT7.10}) fundamentally
makes use of the full generality of Theorem~\ref{CKT2.3}.
Further applications of the general formula \eqref{CKe2.5a} can be found 
in rhombus tiling enumeration (see \cite{CiEKAA,FiscAA}), in
combinatorial commutative algebra (see \cite{KratBO}), and
in the combinatorial theory of orthogonal
polynomials developed in Section~\ref{CKs3.4} (see
the proof of Theorem~\ref{CKT3.8} as given in \cite{VienAE}).

\medskip
In several applications, one has to deal with the problem of
enumerating non-intersecting lattice paths where the starting and end
points are not fixed. Either it is only the starting points which are
fixed and the end points are any points from a given set (or the
other way round), or it is even that the starting points may come
from one set and the end points from another. The solution to these
counting problems comes from \index{Pfaffian}\index{Pfaff, Johann Friedrich}{\em Pfaffians}.

A Pfaffian is very similar
to a determinant. Whereas in the definition of a determinant there
appear \index{permutation}{\em permutations}, 
in the definition of a Pfaffian there
appear \index{matching, perfect}\index{perfect matching}{\em 
perfect matchings}. A {\em
perfect matching} of a set of
objects, $\mathcal A$ say, is a pairing of the objects. 
For example, if $\mathcal A=\{1,2,3,4,5,6\}$, then
$\{\{1,3\},\{2,5\},\{4,6\}\}$ is a matching of $\mathcal A$. 
A matching of $\{1,2,\dots,N\}$ can be realized geometrically by
drawing points labelled $1,2,\dots,N$ along a line, and then
connecting any two points whose labels are paired in the matching by a
curve, so that there are no touching points between curves and no
triple intersections. Figure~\ref{CKFPf2.1} shows the geometric
realization of $\{\{1,3\},\{2,5\},\{4,6\}\}$.
Any two pairs $\{i,k\}$ and $\{j,l\}$ in a matching for which
$i<j<k<l$ are called a \index{crossing}{\em crossing} of the
matching. The \index{sign, of a matching}{\em sign} $\sgn\pi$ of a matching
$\pi$ is $(-1)^c$, where $c$ is the number of crossings of $\pi$.
Thus, the sign of $\{\{1,3\},\{2,5\},\{4,6\}\}$ is $(-1)^2=+1$.
In the geometric realization of a matching,
its sign can be read off as $(-1)^{c'}$, where
$c'$ is the number of crossing points between two curves. (It is
easily checked that it does not matter how we draw the curves.)

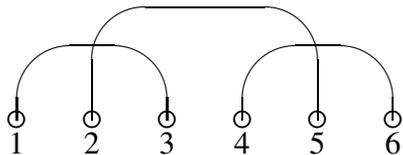
\begin{figure} 
\vskip1cm
$$
\Einheit1cm
\Kreis(0,0)
\Kreis(1,0)
\Kreis(2,0)
\Kreis(3,0)
\Kreis(4,0)
\Kreis(5,0)
\Label\u{1}(0,0)
\Label\u{2}(1,0)
\Label\u{3}(2,0)
\Label\u{4}(3,0)
\Label\u{5}(4,0)
\Label\u{6}(5,0)
\begin{picture}(4,4)
\unitlength1cm
\put(1,0){\oval(2,2)[t]}
\put(2.5,0){\oval(3,3)[t]}
\put(4,0){\oval(2,2)[t]}
\end{picture}
\hskip5cm
$$
\vskip.3cm
\caption{A perfect matching}
\label{CKFPf2.1}
\end{figure}

With these definitions, the 
\index{Pfaffian}\index{Pfaff, Johann Friedrich}{\em Pfaffian}
\glossary{\Pf(A)}$\Pf(A)$ of an upper triangular array
$A=\break(a_{ij})_{1\le i<j\le 2n}$ is defined by
\begin{equation} \label{CKpf2.1}
\Pf(A):=\sum _{\pi\text { a perfect matching of }\{1,2,\dots,2n\}} ^{}
\sgn\pi\prod _{\{i,j\}\in\pi} ^{}a_{ij}.
\end{equation}
For example, for $n=2$ we have
$$\Pf\big((a_{ij})_{1\le i<j\le
4}\big)=a_{12}a_{34}-a_{13}a_{24}+a_{14}a_{23}.$$

Alternatively, the Pfaffian could be defined as the (appropriate) 
square root of a skew symmetric matrix. To be precise, if $A$ is a
skew symmetric matrix, then
\begin{equation} \label{CKpf2.2}
\Pf(A)^2=\det(A),
\end{equation}
where $\Pf(A)$ has to be interpreted as the Pfaffian of the upper
triangular part of $A$. For a (combinatorial) proof of this fact see
e.g\@. \cite[Prop.~2.2]{StemAE}.

\medskip
Now let again $G$ be a directed, acyclic graph. Let
$A_1,A_2,\dots,A_n$ be vertices in $G$, and
$\E=(\dots, E_1,E_2,\dots)$ be an ordered set of vertices.
What we want to count is the
number of {\em all families $\P=(P_1,P_2,\dots,P_n)$ of non-intersecting
paths, where $P_i$ runs from
$A_i$ to some vertex in $\E$, $i=1,2,\dots,n$\/}. The difference to
the situation in Theorem~\ref{CKT2.3} is that the end vertices of the
paths are not fixed. 
Nevertheless, we adopt the earlier notation for this
more general situation. To be precise, by 
\glossary{$L_G(\A\to\E\mid \text{non-intersecting})$}%
$L_G(\A\to\E\mid \text{non-intersecting})$
we mean the set of all families 
$\P=(P_1,P_2,\dots,P_n)$ of non-intersecting
paths in $G$, where $P_i$ runs from
$A_i$ to some vertex $E_{k_i}$ in $\E$, $i=1,2,\dots,n$ 
and $k_1<k_2<\dots<k_n$. 
The corresponding enumeration result is due to
\index{Okada, Soichi}Okada \cite[Theorem~3]{OkadAA} and
\index{Stembridge, John R.}Stembridge \cite[Theorem~3.1]{StemAE}.

\begin{theorem} \label{CKTPf2.1a}%
Let $G$ be a directed, acyclic graph with a weight function
$w$ on its edges.
Let $\A=(A_1,A_2,\dots,A_{2n})$ and\/ $\E=(\dots,E_1,E_2,\dots)$ be sequences
of vertices in $G$.
Then 
\begin{equation} \label{CKpf2.4a}
\sum _{\si\in \mathfrak S_{2n}} ^{}(\sgn\si)\cdot
\GF{L_G(\A_\si\to\E\mid \text{\em non-intersecting});w}=
\Pf_{1\le i<j\le 2n}(Q_G(i,j;w)),
\end{equation}
where $L_G(\A_\si\to\E\mid \text{\em non-intersecting})$ is the set
of all families $(P_1,P_2,\dots,P_{2n})$ of non-intersecting paths,
$P_i$ connecting $A_{\si(i)}$ with $E_{k_i}$, $i=1,2,\dots,2n$
and $k_1<k_2<\dots k_{2n}$,
and $Q_G(i,j;w)$ is the generating function $\sum _{(P',P'')}
^{}w(P')w(P'')$ for all pairs $(P',P'')$ of
non-intersecting lattice paths, where $P'$ connects $A_i$ with 
some $E_k$ and 
$P''$ connects $A_j$ with some $E_l$, with $k<l$.
\end{theorem} 

The proof uses the same involution idea as the proof of Theorem~\ref{CKT2.3}
does. See \cite[Proof of Theorem~3.1]{StemAE}.

Similarly to Theorem~\ref{CKT2.3},
the most frequent situation in which the general result in
Theorem~\ref{CKTPf2.1a} is applied is the one where
non-intersecting paths
can only occur if the starting and end points are connected via
the identity permutation, that is, if $(P_1,P_2,\dots,P_{2n})$
can only be non-intersecting if $P_i$ connects $A_i$ with $E_{k_i}$,
$i=1,2,\dots,2n$ and $k_1<k_2<\dots<k_{2n}$. 
In that situation, Theorem~\ref{CKTPf2.1a}
simplifies to the following result.

\begin{corollary} \label{CKTPf2.1aa}%
Let $G$ be a directed, acyclic graph with a weight function
$w$ on its edges.
Let $\A=(A_1,A_2,\dots,A_{2n})$ and\/ $\E=(\dots,E_1,E_2,\dots)$ be sequences
of vertices in $G$ such that the only permutation $\si$ that
allows for a family $(P_1,P_2,\dots,P_{2n})$ of non-intersecting
paths such that $P_i$ connects $A_{\si(i)}$ with $E_{k_i}$, 
$i=1,2,\dots,2n$ and $k_1<k_2<\dots<k_{2n}$, 
is the identity permutation.
Then the generating function $\sum _{\P} ^{}w(\P)$ for
families $\P=(P_1,P_2,\dots,P_{2n})$ of non-intersecting
paths, where $P_i$ is a path running from
$A_i$ to $E_{k_i}$, $i=1,2,\dots,2n$ and $k_1<k_2<\dots<k_{2n}$, 
is given by 
\begin{equation} \label{CKpf2.4aa}
\GF{L_G(\A\to\E\mid \text{\em non-intersecting});w}=
\Pf_{1\le i<j\le 2n}(Q_G(i,j;w)),
\end{equation}
where $Q_G(i,j;w)$ has the same meaning as in Theorem~\ref{CKTPf2.1a}.
\end{corollary} 

Theorem~\ref{CKTPf2.1a} and Corollary~\ref{CKTPf2.1aa} 
are only formulated for an even number of paths.
However, a simple trick allows us to also use it for an odd number of
paths: one introduces a ``phantom" vertex $X$ that cannot be reached by
any other vertex (one can think of 
a vertex at infinity), and adjoins this point as a new starting
point {\it and\/} as a new end point. A family of non-intersecting 
paths would necessarily contain the zero-length path from $X$ to $X$
as one of the paths, which therefore can be ignored.
Theorem~\ref{CKTPf2.1a} applies, and yields the following corollary.

\begin{corollary} \label{CKTPf2.1u}%
Let $G$ be a directed, acyclic graph with a weight function
$w$ on its edges.
Let $\A=(A_1,A_2,\dots,A_{2n-1})$ and\/ $\E=(\dots,E_1,E_2,\dots)$ be sequences
of vertices in $G$.
Then 
\begin{equation} \label{CKpf2.4u}
\sum _{\si\in \mathfrak S_{2n}} ^{}(\sgn\si)\cdot
\GF{L_G(\A_\si\to\E\mid \text{\em non-intersecting});w}=
\Pf(Q_G(i,j;w))_{1\le i<j\le 2n},
\end{equation}
where $L_G(\A_\si\to\E\mid \text{\em non-intersecting})$ has the
same meaning as in Theorem~\ref{CKTPf2.1a},
where for $j\le 2n-1$ the quantity $Q_G(i,j;w)$ has the
same meaning as in Theorem~\ref{CKTPf2.1a},
and where $Q_G(i,2n;w)$ is the generating function $\sum _{P}
^{}w(P)$ for all paths running from $A_i$ to some point of $\E$.

In particular, if the vertices $\A$ and\/ $\E$ are such that
the only permutation $\si$ that
allows for a family $(P_1,P_2,\dots,P_{2n-1})$ of non-intersecting
paths such that $P_i$ connects $A_{\si(i)}$ with $E_{k_i}$, 
$i=1,2,\dots,2n-1$ and $k_1<k_2<\dots<k_{2n-1}$, 
is the identity permutation,
then the generating function $\sum _{\P} ^{}w(\P)$ for
families $\P=(P_1,P_2,\dots,P_{2n-1})$ of non-intersecting
paths, where $P_i$ is a path running from
$A_i$ to $E_{k_i}$, $i=1,2,\dots,2n-1$ and $k_1<k_2<\dots<k_{2n-1}$, 
is given by 
\begin{equation} \label{CKpf2.4ub}
\GF{L_G(\A\to\E\mid \text{\em non-intersecting});w}=
\Pf(Q_G(i,j;w))_{1\le i<j\le 2n}.
\end{equation}
\end{corollary}

For various applications of this theorem, see e.g.\
\cite{KratBD,OkadAA,StemAE}.

\medskip
Next we consider a mixed case, in which the starting points of the
paths are fixed, some end points are fixed, but some end points can
be chosen from a given set. To be precise, let $m$ and $n$ 
be a positive integer with $m\le n$, let 
$\A=(A_1,A_2,\dots,A_n)$ and $\E=(E_1,E_2,\dots,E_m)$ be
vertices in a given directed, acyclic graph $G$, let 
$\hat\E=(\dots,\hat E_1,\hat E_2,\dots)$ be an ordered set of 
(finitely many or infinitely many) vertices.
What we want to determine is the
generating function for 
{\em all families $\P=(P_1,P_2,\break\dots,P_n)$ of non-intersecting
paths, where for
$i=1,2,\dots,m$ the path $P_i$ runs from $A_i$ to $E_i$, and where
for $i=m+1,m+2,\dots,n$ the path $P_i$ runs from
$A_i$ to some point $\hat E_{k_i}$ in $\hat\E$, with
$k_{m+1}<k_{m+2}<\dots<k_n$}. We write
\glossary{$L_G(\A\to\E\cup\hat\E\mid \text 
{non-intersecting})$}$L_G(\A\to\E\cup\hat\E\mid \text
{non-intersecting})$ for the set of these families of paths.
The corresponding enumeration result is due to
\index{Stembridge, John R.}Stembridge \cite[Theorem~3.2]{StemAE}.

\begin{theorem} \label{CKTPf2.2}%
Let $G$ be a directed, acyclic graph with a weight function
$w$ on its edges, and let $m$ and $n$ 
be a positive integer such that $m\le n$ and $m+n$ is even.
Let $\A=(A_1,A_2,\dots,A_n)$, $\E=(E_1,E_2,\dots,E_m)$ and\/
$\hat\E=(\dots,\hat E_1,\hat E_2,\dots)$ be sequences
of vertices in $G$.
Then 
\begin{equation} \label{CKpf2.7}
\sum _{\si\in \mathfrak S_{n}} ^{}(\sgn\si)\cdot
\GF{L_G(\A_\si\to\E\cup\hat\E\mid \text{\em non-intersecting});w}=
\Pf\begin{pmatrix} \hphantom{-{}}Q&H\\-H^t&0\end{pmatrix},
\end{equation}
where $L_G(\A_\si\to\E\cup \hat\E\mid \text{\em non-intersecting})$ is the set
of all families $(P_1,P_2,\dots,P_{n})$ of non-intersecting paths,
$P_i$ connecting $A_{\si(i)}$ with $E_i$, $i=1,2,\dots,m$,
$P_i$ connecting $A_{\si(i)}$ with $\hat E_{k_i}$, $i=m+1,m+2,\dots,n$
and $k_{m+1}<k_{m+2}<\dots< k_{n}$,
where $Q=(Q_G(i,j;w))_{1\le i,j\le n}$ is a skew-symmetric matrix with
$Q_G(i,j;w)$ denoting the generating function $\sum _{(P',P'')}
^{}w(P')w(P'')$ for all pairs $(P',P'')$ of
non-intersecting lattice paths, 
where $P'$ connects $A_i$ with 
some $\hat E_k$ and 
$P''$ connects $A_j$ with some $\hat E_l$, with $k<l$,
and where
$H=(H_G(i,j;w))_{1\le i\le n,\, 1\le j\le m}$ is the rectangular matrix
with $H_G(i,j;w)$ denoting the generating function $\sum _{P}
^{}w(P)$ for all paths $P$ from $A_i$ to $E_{j}$.
The Pfaffian of a skew-symmetric matrix has to be interpreted
according to the remark containing \eqref{CKpf2.2}.

In particular, if the vertices $\A$ and\/ $\E\cup\hat\E$ are such that
the only permutation $\si$ that
allows for a family $(P_1,P_2,\dots,P_{n})$ of non-intersecting
paths such that $P_i$ connects $A_{\si(i)}$ with $E_{i}$, 
$i=1,2,\dots,m$, and $P_i$ connects $A_{\si(i)}$ with $\hat E_{k_i}$, 
$i=m+1,m+2,\dots,n$ and $k_{m+1}<k_{m+2}<\dots<k_{n}$, 
is the identity permutation,
then the generating function $\sum _{\P} ^{}w(\P)$ for
all families $(P_1,P_2,\dots,P_{n})$ of non-intersecting
lattice paths, where for
$i=1,2,\dots,m$ the path $P_i$ runs from $A_i$ to $E_i$, and where
for $i=m+1,m+2,\dots,n$ the path $P_i$ runs from
$A_i$ to some point $\hat E_{k_i}$ in $\hat\E$, $k_{m+1}<k_{m+2}<\dots<k_n$,
is given by 
\begin{equation} \label{CKpf2.7b}
\GF{L_G(\A\to\E\cup\hat\E\mid \text{\em non-intersecting});w}=
(-1)^{\binom m2}\Pf\begin{pmatrix} \hphantom{-{}}Q&H\\-H^t&0\end{pmatrix}.
\end{equation}
\end{theorem}

Again,
the proof uses the same involution idea as the proof of Theorem~\ref{CKT2.3}
does. See \cite[Proof of Theorem~3.2]{StemAE}.
Applications can for example be found in
\cite{CiKrAE,CiKrAH,StemAE}.

\medskip
As a matter of fact, Theorem~\ref{CKTPf2.2} is a special case of the so-called
minor summation formula due to 
\index{Ishikawa, Masao}Ishikawa and 
\index{Wakayama, Masato}Wakayama \cite[Theorem~2]{IsWaAA}.

\begin{theorem}\label{CKthm:IsWa}%
Let $m$, $n$, $p$ be integers such that $n+m$ is even and $0 \le n-m \le p$.
Let $M$ be any $n \times p$ matrix, $H$ be any $n \times m$ matrix,
 and $A=(a_{ij})_{1 \le i,j \le p}$ be any skew-symmetric matrix.
Then we have
$$
\sum_{K}
 \Pf \left( A^K_K \right)
 \det \big( M_K \ \vdots \  H \big)
=
(-1)^{\binom m2}
 \Pf \begin{pmatrix}
 M\,A\,M^t & H \\ - H^t & 0 \end{pmatrix}.
$$
where $K$ runs over all $(n-m)$-element subsets of $[1, p]$,
 $A^K_K$ is the skew-symmetric matrix obtained by picking the rows and
 columns indexed by $K$, and $M_K$ is the sub-matrix of $M$ consisting of
 the columns corresponding to $K$.
\end{theorem}

Theorem~\ref{CKTPf2.2} results from the special case of Theorem~\ref{CKthm:IsWa} 
where $A$ is the $p\times p$ skew-symmetric matrix with all $1$s above
the diagonal, and $M$ and $H$ are 
matrices the entries of which are appropriately chosen
path generating functions. This is based on the well-known fact
(see e.g.\ \cite[Prop.~2.3(c)]{StemAE}) that $\Pf(1)_{1\le i<j\le 2N}=1$
for all $N$.

\medskip
The last theorem in this section addresses the case where starting
{\it and\/} end points are chosen from given sets.
To be precise, let 
$\A=(A_1,A_2,\dots,A_n)$ and 
$\E=(\dots,E_1,\break E_2,\dots)$ be ordered sets of 
vertices (finitely many or infinitely many in the case of $\E$).
What we want to determine is the generating function for 
{\em all families $\P=(P_1,P_2,\dots,P_s)$ of non-intersecting
paths, where for
$i=1,2,\dots,s$ the path $P_i$ runs from some $A_{k_i}$ to some 
$E_{l_i}$}. 
The corresponding enumeration result is due to
\index{Okada, Soichi}Okada \cite[Theorem~4]{OkadAA} and
\index{Stembridge, John R.}Stembridge \cite[Theorem~4.1]{StemAE}.
In the formulation below, by abuse of notation, $\A'\subseteq\A$
means that $\A'$ is a subsequence of $\A$, with an analogous meaning
for $\E'\subseteq\E$.

\begin{theorem} \label{CKTPf2.4}%
Let $G$ be a directed, acyclic graph with a weight function
$w$ on its edges, and 
let $\A=(A_1,A_2,\dots,A_n)$ and\/
$\E=(\dots,E_1,E_2,\dots)$ be sequences
of vertices in $G$.

\smallskip
{\em(a)} If $n$ is even, then
\begin{multline} \label{CKpf2.8a}
\sum_{s=0}^{n/2}
t^s
\underset{\vert \A'\vert=\vert \E'\vert=2s}
{\sum _{\A'\subseteq \A,\,\E'\subseteq\E} ^{}}
\GF{L_G(\A'\to\E'\mid \text{\em non-intersecting});w}\\
=
\Pf_{1\le i<j\le n}\big(
(-1)^{i+j-1}+tQ_G(i,j;w)
\big),
\end{multline}
where $Q_G(i,j;w)$ has the same meaning as in Theorem~\ref{CKTPf2.1a}.

\smallskip
{\em(b)} If $n$ is odd, then
\begin{multline} \label{CKpf2.8b}
\sum_{s=0}^{n}
t^s
\underset{\vert \A'\vert=\vert \E'\vert=s}
{\sum _{\A'\subseteq \A,\,\E'\subseteq\E} ^{}}
\GF{L_G(\A'\to\E'\mid \text{\em non-intersecting});w}\\
=
\underset{1\le i<j\le n+1}{\Pf}\left(
(-1)^{i+j-1}+t^2Q_G(i,j;w)
\right),
\end{multline}
where for $j\le n$ the quantity $Q_G(i,j;w)$ has the same meaning 
as in Theorem~\ref{CKTPf2.1a}, while $Q_G(i,n+1;w)$ equals the generating
function $t^{-1}\sum _{P}
^{}w(P)$ for all paths $P$ from $A_i$ to some vertex in $\E$.

\smallskip
{\em(c)} If $n$ is even, then
\begin{multline} \label{CKpf2.8c}
\sum_{s=0}^{n}
t^s
\underset{\vert \A'\vert=\vert \E'\vert=s}
{\sum _{\A'\subseteq \A,\,\E'\subseteq\E} ^{}}
\GF{L_G(\A'\to\E'\mid \text{\em non-intersecting});w}\\
=
\underset{1\le i<j\le n+2}{\Pf}\left(
(-1)^{i+j-1}+t^2Q_G(i,j;w)
\right),
\end{multline}
where for $j\le n+1$ the quantity $Q_G(i,j;w)$ has the same meaning 
as in {\em(b)}, and $Q_G(i,n+2;w)=0$.
\end{theorem}

For applications of this theorem in plane partition enumeration
see \cite{OkadAA} and \cite{StemAE}.

\medskip
If one weakens the condition of non-intersection
of lattice paths in the plane to the
requirement that paths are allowed to touch each other
in isolated points but not to change sides, then one arrives at the
model of 
\index{osculating path}\index{path, osculating}%
{\it osculating paths}. The motivation to consider
this model comes from an observation of
\index{Bousquet-M\'elou, Mireille}Bousquet-M\'elou and 
\index{Habsieger, Laurent}Habsieger \cite{BoHaAA} that
\index{alternating sign matrix}\index{matrix, alternating sign}%
{\it alternating sign matrices} are in bijection with
families of osculating paths with appropriate starting
and end points. Alternating sign matrices are fascinating,
but notoriously difficult to count, therefore it may be
useful to investigate objects which are equivalent to
them. So far, this point of view has not led to much,
but recently
\index{Brak, Richard}Brak and 
\index{Galleas, Wellington}Galleas \cite{BrGaAA}
proved a constant term formula for families of
osculating paths.

\section{Lattice paths and their turns}\label{CKsec:turn}

In this section we consider {\em turns\/}\index{turn of a lattice 
path}\index{path, turn} of lattice paths. Literally, a {\em turn\/} of a 
lattice path is any 
vertex of a path where the direction of the path changes. The
enumeration of lattice paths with a given number of turns is motivated
by problems of correlated random walks, distribution of runs
(cf.\ \cite{MohaAE}),
coefficients of Hilbert polynomials of determinantal and Pfaffian
rings (cf.\ \cite{KrPrAA,KulkAC}), and summations for Schur functions
(cf.\ \cite{KratAP}).

The approach for the enumeration of simple plane lattice paths
with respect to their number of turns which we present here
is by encoding lattice paths in terms of two-rowed
arrays, a point of view put forward in \cite{KratBL}. 

\medskip
For simple lattice paths in the plane there are two types of turns. We
call a vertex $T$ of a path a {\em North-East turn\/}\index{North-East turn}
\index{turn, North-East} ({\em NE-turn\/}\index{NE-Turn} for short) if
$T$ is reached by a step towards north and left by a step 
towards east. We
call a vertex $T$ of a path an {\em east-north turn\/}\index{east-north turn}
\index{turn, east-north} ({\em EN-turn\/}\index{EN-Turn} for short) if
$T$ is reached by a step towards east and left by a step 
towards north. Thus, the NE-turns of the path $P_0$ in
Figure~\ref{CKF4.1} are $(1,1)$, $(2,3)$, and $(5,4)$, the EN-turns of
$P_0$ are $(2,1)$, $(5,3)$, and $(6,4)$.
\begin{figure} 
$$\Koordinatenachsen(8,8)(0,0)
\Gitter(8,8)(0,0)
\Pfad(1,-1),221221112122\endPfad
\DickPunkt(1,-1)
\DickPunkt(6,6)
\Label\ro{P_0}(3,3)
\hskip4cm
$$
\caption{}
\label{CKF4.1}
\end{figure}
We denote by \glossary{$NE(P)$}$NE(P)$ the number of NE-turns of $P$ and
by $EN(P)$\glossary{$EN(P)$} the number of EN-turns of $P$.

Now we describe the encoding of paths in terms of two-rowed arrays.
Actually, we use two encodings, one corresponding to NE-turns, one
corresponding to EN-turns. Let $(p_1,q_1)$, $(p_2,q_2)$, \dots,
$(p_\ell,q_\ell)$ be the NE-turns of a path $P$. Then the {\em NE-turn
representation\/}\index{NE-turn representation} 
of $P$ is defined by the two-rowed array
\begin{equation}\label{CKe4.1}\begin{array} {cccc}
p_1&p_2&\dots&p_\ell\\
q_1&q_2&\dots&q_\ell,
\end{array}
\end{equation}
which consists of two strictly increasing sequences. Clearly, if 
$P$ runs from $(a,b)$ to $(c,d)$ then $a\le p_1<p_2<\dots<
p_\ell\le c-1$ and $b+1\le q_1<q_2<\dots<q_\ell\le d$ are satisfied. 
If we wish to make this fact transparent, we write
\begin{equation} \label{CKe4.2} 
\begin{array} {rccccl}
a\le&\quad p_1&p_2&\dots&p_\ell&\quad \le c-1\\
b+1\le&\quad q_1&q_2&\dots&q_\ell&\quad \le d.
\end{array}
\end{equation}
For a given starting point and a given final point,
by definition the empty array is the representation for the only path
that has no NE-turn.
For the path in our running example
we obtain the NE-turn representation
$$\begin{matrix} 1& 2& 5\\1& 3& 4\end{matrix}\ ,$$
or, with bounds included,
$$\begin{array}{rcccl}\\
1\le&\quad 1&\ 2&\ 5&\quad \le5\\
0\le&\quad 1&\ 3&\ 4&\quad \le6\end{array}\ .$$

Similarly, if $(p_1,q_1)$, $(p_2,q_2)$, \dots,
$(p_\ell,q_\ell)$ denote the EN-turns 
of a path $P$, then \eqref{CKe4.1} is called the {\em EN-turn
representation\/}\index{EN-turn representation} 
of $P$.
If
$P$ runs from $(a,b)$ to $(c,d)$ then $a+1\le p_1<p_2<\dots<
p_\ell\le c$ and $b\le q_1<q_2<\dots<q_\ell\le d-1$ are satisfied. Again,
as earlier, we write
\begin{equation} \label{CKe4.3} 
\begin{array} {rccccl}
a+1\le&\quad p_1&p_2&\dots&p_\ell&\quad \le c\\
b\le&\quad q_1&q_2&\dots&q_\ell&\quad \le d-1.
\end{array}
\end{equation}
For a given starting point and a given final point,
by definition the empty array is the representation for the only path
that has no EN-turn.
For the path in our running example
we obtain the EN-turn representation
$$\begin{matrix} 2& 5& 6\\1& 3& 4\end{matrix}\ ,$$
or, with bounds included,
$$\begin{array}{rcccl}\\
2\le&\quad 2&\ 5&\ 6&\quad \le6\\
-1\le&\quad 1&\ 3&\ 4&\quad \le5\end{array}\ .$$


Also two-rowed arrays with its
rows being of unequal length will be considered. These arrays will
also have the property that the rows are strictly increasing. So, by
convention, whenever we speak of two-rowed arrays, we
mean two-rowed arrays with strictly increasing rows. For these
arrays we will use a notation of the kind \eqref{CKe4.2} or
\eqref{CKe4.3} as well. 
We shall frequently use the short notation $(\mathbf a\mid \mathbf b)$ 
for two-rowed arrays, where $\mathbf a$
denotes the sequence $(a_i)$ of elements of the first row, 
and $\mathbf b$ denotes the
sequence $(b_i)$ of elements of the second row.

From \eqref{CKe4.2} we see at once that the number of all paths from
$(a,b)$ to $(c,d)$ with exactly $\ell$ NE-turns equals the number of
$\ell$-element subsets of $\{a,a+1,\dots,c-1\}$ times the number of
$\ell$-element subsets of $\{b+1,b+2,\dots,d\}$. A similar argument
holds for EN-turns. Thus we have proved
\begin{multline}\label{CKe4.4}
\vv{\LL{(a,b)\to (c,d)\mid NE(.)=\ell}}=
\vv{\LL{(a,b)\to (c,d)\mid EN(.)=\ell}}\\
=\binom {c-a}{\ell}
\binom{d-b}\ell.
\end{multline}

\medskip
A lattice path statistic that is frequently used is the number of
{\em runs\/} of a lattice path. A {\em run\/}\index{run of a
path}\index{path, run} in a path $P$ is a maximal subpath of $P$
consisting of steps of equal type. We write $\run(P)$ for the number
of runs of $P$. The runs of the path $P_0$ in Figure~\ref{CKF4.1} are
the subpaths from $(1,-1)$ to $(1,1)$, from $(1,1)$ to $(2,1)$, from
$(2,1)$ to $(2,3)$, from $(2,3)$ to $(5,3)$, from $(5,3)$ to $(5,4)$,
from $(5,4)$ to $(6,4)$, and from $(6,4)$ to $(6,6)$. Thus we have
$\run(P_0)=7$. Obviously, the number of runs of a path is exactly one
more than the total number of turns (both, NE-turns and EN-turns).
Besides, there is also a close relation between NE-turns and
runs, which allows us to translate any enumeration result for NE-turns
into one for runs. 

To avoid case by case formulations, depending on whether the number of
runs is even or odd, we prefer to consider generating functions.
Suppose we know the number of all paths from $A$ to $E$ satisfying
some property $R$ and containing a given number of NE-turns.
Then we 
know the generating function $GF(L(A\to E\mid R);x^{NE(.)})$. For
brevity, let
us denote it by $F(A\to E\mid R;x)$. We define four refinements of
$F(A\to E\mid R;x)$. Let $F_{hv}(A\to E\mid R;x)$ be the generating
function $\sum _{P} ^{}x^{NE(P)}$ where the sum is over all paths in
$L(A\to E\mid R)$ that start with a horizontal step and end with a
vertical step. The notations $F_{hh}(A\to E\mid R;x)$, $F_{vh}(A\to
E\mid R;x)$, and $F_{vv}(A\to E\mid R;x)$ are defined analogously. 
The relation between
enumeration by runs and enumeration by NE-turns is given by
\begin{multline} \label{CKe4.5}
GF(L(A\to E\mid R);x^{\run(.)})=xF_{hh}(A\to E\mid R;x^2) +x^2
F_{hv}(A\to E\mid R;x^2)\\
+ F_{vh}(A\to E\mid R;x^2)+xF_{vv}(A\to E\mid R;x^2).
\end{multline}

All the four refinements of the NE-turn generating function can be
expressed in terms of NE-turn generating functions. This is seen by
setting up a few linear equation and solving them. Evidently, 
the following is true:
\begin{multline} \notag
F(A\to E\mid R;x)=F_{hh}(A\to E\mid R;x) +
F_{hv}(A\to E\mid R;x)\\
+ F_{vh}(A\to E\mid R;x)+F_{vv}(A\to E\mid R;x).
\end{multline}
Besides, if $E_1=(1,0)$ and $E_2=(0,1)$ denote the standard unit
vectors, we have
\begin{align} \notag
F_{hh}(A\to E\mid R;x)+F_{hv}(A\to E\mid R;x)&=F(A+E_1\to E\mid R;x)\\
\notag
F_{hv}(A\to E\mid R;x)+F_{vv}(A\to E\mid R;x)&=F(A\to E-E_2\mid R;x)\\
\notag
F_{hv}(A\to E\mid R;x)&=F(A+E_1\to E-E_2\mid R;x).
\end{align}
Solving for $F_{hh}$, $F_{hv}$, $F_{vh}$, $F_{vv}$ we get
{\refstepcounter{equation}\label{CKe4.6}}
\alphaeqn
\begin{align} \label{CKe4.6a}
F_{hh}(A\to E\mid R;x)&=F(A+E_1\to E\mid R;x)-F(A+E_1\to E-E_2\mid R;x)\\
\label{CKe4.6b}
F_{hv}(A\to E\mid R;x)&=F(A+E_1\to E-E_2\mid R;x)\\
\notag
F_{vh}(A\to E\mid R;x)&=F(A\to E\mid R;x)+(A+E_1\to E-E_2\mid R;x)\\
\label{CKe4.6c}
&\hskip1cm -F(A+E_1\to E\mid R;x)-F(A\to E-E_2\mid R;x)\\
\label{CKe4.6d}
F_{vv}(A\to E\mid R;x)&=F(A\to E-E_2\mid R;x)-F(A+E_1\to E-E_2\mid R;x).
\end{align}
\reseteqn

\medskip
As we know from Section~\ref{CKsec:lin1}, counting paths restricted by
$x=y$, or even by two lines $x=y+t$ and $x=y+s$, is effectively solved
by the reflection principle. Of course, reflection by itself is useless
for counting paths by turns, since the reflection of portions of paths
does not take care of turns. It might introduce new turns or make
turns disappear. However, there are ``analogues" of reflection for
two-rowed arrays, which are due to \index{Krattenthaler, Christian}Krattenthaler and 
\index{Mohanty, Sri Gopal}Mohanty 
\cite{KrMoAC}. 

\begin{theorem} \label{CKT4.1}%
Let $a\ge b$ and $c\ge d$. The number of all paths from $(a,b)$ to
$(c,d)$ staying weakly below $x=y$ with exactly $\ell$ NE-turns is given by
\begin{multline} 
\vv{\LL{(a,b)\to (c,d)\mid x\ge y,\,NE(.)=\ell}}\\
=\binom {c-a}{\ell}\binom{d-b}\ell - \binom {c-b-1}{\ell-1}
\binom{d-a+1}{\ell+1} ,\label{CKe4.7} 
\end{multline}
and with exactly $\ell$ EN-turns is given by
\begin{multline} 
\vv{\LL{(a,b)\to (c,d)\mid x\ge y,\,EN(.)=\ell}}\\
=\binom {c-a}{\ell}\binom{d-b}\ell - \binom {c-b+1}{\ell}
\binom{d-a-1}{\ell} .\label{CKe4.8} 
\end{multline}
\end{theorem}

\begin{proof}We start with proving \eqref{CKe4.7}. 
By the NE-turn representation \eqref{CKe4.2}, the paths from $(a,b)$
to $(c,d)$ staying weakly below $x=y$ with exactly $\ell$ NE-turns
 can be
represented by 
{\refstepcounter{equation}\label{CKe4.9}}
\alphaeqn
\begin{equation} \label{CKe4.9a} 
\begin{array} {rccccl}
a\le&\quad p_1&p_2&\dots&p_\ell&\quad \le c-1\\
b+1\le&\quad q_1&q_2&\dots&q_\ell&\quad \le d,
\end{array}
\end{equation}
where
\begin{equation} \label{CKe4.9b} 
p_i\ge q_i, \quad \quad i=1,2,\dots,\ell.
\end{equation}
\reseteqn
Following the argument in the proof of Theorem~\ref{CKT1.1}, the number
of these two-rowed arrays is the number of {\em all\/} two-rowed arrays
of the type \eqref{CKe4.9a} minus those two-rowed arrays of the type
\eqref{CKe4.9a} which {\em violate\/} \eqref{CKe4.9b}, i.e., where $p_i<q_i$ for
some $i$ between $1$ and $\ell$. We know the first number from
\eqref{CKe4.4}. 

Concerning the second number, we claim that two-rowed arrays of the
type \eqref{CKe4.9a} which violate \eqref{CKe4.9b} are in one-to-one
correspondence with two-rowed arrays of the type
\begin{equation} \label{CKe4.10} 
\begin{array} {rcccccl}
b+1\le&\quad &&\bar p_2&\dots&\bar p_\ell&\quad \le c-1\\
a\le&\quad \bar q_0&\bar q_1&\bar q_2&\dots&\bar q_\ell&\quad \le d.
\end{array}
\end{equation}
The number of all these two-rowed arrays is $\binom {c-b-1}{\ell-1}
\binom {d-a+1}{\ell+1}$, as desired. So it only remains to construct
the one-to-one correspondence.

Take a two-rowed array $(\mathbf p\mid \mathbf q)$ of the type \eqref{CKe4.9a} such that
$p_i<q_i$ for some $i$. Let $I$ be the largest integer such that
$p_I<q_I$. Then map $(\mathbf p\mid \mathbf q)$ to
\begin{equation} \label{CKe4.11} 
\begin{array} {ccccccccc}
&&q_1\kern15pt&\dots\dots &q_{I-1}&p_{I+1}&\dots&p_\ell\\
p_1&p_2&\dots\dots&\kern15ptp_I&q_I&q_{I+1}&\dots&q_\ell
\end{array}.
\end{equation}
Note that both rows are strictly increasing because of
$q_{I-1}<q_I<q_{I+1}\le p_{I+1}$ and $p_I<q_I$. By some case by case
analysis it can be seen that \eqref{CKe4.11} is of type \eqref{CKe4.10}.
For example, if $I=\ell$ then we must check $q_{I-1}\le c-1$, among
others. Clearly, this follows from the inequalities $q_{I-1}<q_I\le
d\le c$.

The inverse of this map is defined in the same way. Let $(\bar {\mathbf
p}\mid
\bar{\mathbf q})$
be a two-rowed array of the type \eqref{CKe4.10}. Let $\bar I$ be the
largest integer such that $\bar p_{\bar I}<\bar q_{\bar I}$. If there
are none, take $\bar I=2$. Then map $(\bar {\mathbf p}\mid \bar {\mathbf q})$ to
\begin{equation} \label{CKe4.12} 
\begin{array} {ccccccc}
\bar q_0&\multicolumn{2}{c}{\dotfill}&\bar q_{\bar I-1}&\bar p_{\bar
I+1} &\dots&\bar p_\ell\\
\bar p_2&\dots&\bar p_{\bar I}&\bar q_{\bar I}&\multicolumn{2}{c}{\dotfill}&\bar q_\ell
\end{array}.
\end{equation}
It is not difficult to check that the mappings \eqref{CKe4.11} and
\eqref{CKe4.12} are inverses of each other. This completes the proof of
\eqref{CKe4.7}.

The second identity, \eqref{CKe4.8}, can be established similarly.
\end{proof}

\begin{remark}\em
The above proof leads in fact to $q$-analogues; see \cite{KrMoAC}.
\end{remark}

A refinement of
Theorem~\ref{CKT1.3} taking into account turns may as well be derived
in this way. 

\begin{theorem} \label{CKT4.2}%
Let $a+t\ge b\ge a+s$ and\/ $c+t\ge d\ge c+s$. The number of all paths
from $(a,b)$ to $(c,d)$ staying weakly below the line $y=x+t$ and above the
line $y=x+s$ with exactly $\ell$ NE-turns is given by
\begin{multline} \label{CKe4.17}
\vv{\LL{(a,b)\to (c,d)\mid x+t\ge y\ge x+s,\,NE(.)=\ell}}\\
=\sum _{k\in\Z} ^{}\bigg(\binom
{c-a-k(t-s)} {\ell+k}\binom {d-b+k(t-s)}{\ell-k}\hskip3cm\\
-\binom {c-b-k(t-s)
+s-1}{\ell+k} \binom {d-a+k(t-s)-s+1}{\ell-k}\bigg). 
\end{multline}
\end{theorem}

Some of the results in Section~\ref{CKsec:rat} allow also for 
refinements taking into account turns.

\begin{theorem} \label{CKT4.4}%
Let $\mu$ be a positive integer and
$c\ge \mu d$. The number of all lattice paths from the origin to
$(c,d)$ which stay weakly below $x=\mu y$ with exactly $\ell$ NE-turns
is given by
\begin{equation} \label{CKe4.31} 
\vv{\LL{(0,0)\to (c,d)\mid x\ge \mu y,\,NE(.)=\ell}}
=\binom c\ell \binom d\ell-\mu\binom {c-1}{\ell-1} \binom
{d+1}{\ell+1},\hskip.5cm
\end{equation}
and with exactly $\ell$ EN-turns is given by
\begin{multline} \label{CKe4.32} 
\vv{\LL{(0,0)\to (c,d)\mid x\ge \mu y,\,EN(.)=\ell}}\\
=\frac {c-\mu d+1} {c+1}\binom {c+1}{\ell} \binom {d-1}{\ell-1}
=\binom {c+1}\ell \binom {d-1}{\ell-1}-\mu\binom {c}{\ell-1} \binom
{d}{\ell}.
\end{multline}
\end{theorem}
Two-rowed arrays may also be used to prove this result,
see \cite{KratBL}. A very elegant alternative proof using a
rotation operation on paths is given by
\index{Goulden, Ian P.}Goulden and 
\index{Serrano, Luis Guillermo}Serrano \cite{GoSeAA}.

\medskip
We conclude this section by stating results on the enumeration of
families of 
\index{non-intersecting lattice paths}%
\index{paths, non-intersecting}{\it non-intersecting} lattice paths 
with respect to turns. This type of problem
originally arose from the
study of the Hilbert polynomial of certain determinantal and Pfaffian
rings (cf.\ \cite{KrPrAA,KulkAC}). The results are due to 
\index{Krattenthaler, Christian}Krattenthaler \cite{KratBE}.
We do not provide proofs. Suffice it to mention that they work by
using two-rowed arrays.

Let $\A=(A_1,A_2,\dots,A_n)$ and
$\E=(E_1,E_2,\dots,E_n)$ be points in $\Z^2$. How many families
$\P=(P_1,P_2,\dots,P_n)$ of non-intersecting lattice paths, where $P_i$
runs from $A_i$ to $E_i$, $i=1,2,\dots,n$, are there such that the
total number of NE-turns in $\P$ is some fixed number, $\ell$ say? 

We give the following three theorems about
the counting of non-intersecting lattice paths with a given number of
turns. The first theorem concerns counting families of non-intersecting
lattice paths with given starting and end points with a given number
of NE-turns.
\begin{theorem} \label{CKT4.10}%
Let 
$A_i=(a_1^{(i)},a_2^{(i)})$ and\/ $E_i=(e_1^{(i)},e_2^{(i)})$ be lattice points satisfying
$$a_1^{(1)}\le a_1^{(2)}\le\dotsb\le a_1^{(n)},\quad
a_2^{(1)}>a_2^{(2)}>\dotsb>a_2^{(n)},$$
and
$$e_1^{(1)}< e_1^{(2)}<\dotsb< e_1^{(n)},\quad
e_2^{(1)}\ge e_2^{(2)}\ge \dotsb\ge e_2^{(n)}.$$
The number of all families $\P=(P_1,P_2,
\dots,P_n)$ of non-intersecting lattice paths $P_i:A_i\to E_i$,
such that the paths of\/ $\P$ altogether contain exactly $\ell$ NE-turns, is
\begin{equation}\label{CKe4.51}
\sum _{\ell_1+\dotsb+\ell_n=\ell} ^{}\det_{1\le i,j\le n}
\bigg(\binom{e_1^{(j)}-a_1^{(i)}+i-j}{\ell_i+i-j}
\binom{e_2^{(j)}-a_2^{(i)}-i+j}{\ell_i}\bigg).
\end{equation}
\end{theorem}
The second theorem concerns counting families of non-intersecting
lattice paths staying weakly below $x=y$, with given starting and end points,
by their number of NE-turns.
\begin{theorem}\label{CKT4.11}%
Let $A_i=(a_1^{(i)},a_2^{(i)})$ and
$E_i=(e_1^{(i)},e_2^{(i)})$ be lattice points satisfying
$$a_1^{(1)}\le a_1^{(2)}\le\dotsb\le a_1^{(n)},\quad
a_2^{(1)}>a_2^{(2)}>\dotsb>a_2^{(n)},$$
$$
e_1^{(1)}< e_1^{(2)}<\dotsb< e_1^{(n)},\quad
e_2^{(1)}\ge e_2^{(2)}\ge \dotsb\ge e_2^{(n)},$$
and
$a_1^{(i)}\ge a_2^{(i)},\quad e_1^{(i)}\ge e_2^{(i)},\quad
i=1,2,\dots,n.$
The number of all families $\P=(P_1,P_2,\break\dots,P_n)$ of non-intersecting
lattice paths $P_i:A_i\to E_i$, which stay weakly below the line $x=y$, and
where the paths of\/ $\P$ altogether
contain exactly $\ell$ NE-turns, is
\begin{multline}\label{CKe4.52}
\sum _{\ell_1+\dotsb+\ell_n=\ell} ^{}\det_{1\le i,j\le n}
\bigg(\binom{e_1^{(j)}-a_1^{(i)}+i-j}{\ell_i+i-j}
\binom{e_2^{(j)}-a_2^{(i)}-i+j}{\ell_i}\\
-\binom{e_1^{(j)}-a_2^{(i)}-i-j+1}{\ell_i-j}
\binom{e_2^{(j)}-a_1^{(i)}+i+j-1}{\ell_i+i}\bigg).
\end{multline}
\end{theorem}
In the third theorem (basically) the same families of non-intersecting
lattice paths as before are counted, but with respect to their number
of EN-turns. By a rotation by $180^\circ$ this could be translated
into a result about counting families of non-intersecting lattice paths
staying {\em above\/} $x=y$, with given starting and end points, with
respect to {\em NE-turns\/}.
\begin{theorem}\label{CKT4.12}%
Let $A_i=(a_1^{(i)},a_2^{(i)})$ and 
$E_i=(e_1^{(i)},e_2^{(i)})$ be lattice points satisfying 
$$a_1^{(1)}< a_1^{(2)}<\dotsb< a_1^{(n)},\quad
a_2^{(1)}\ge a_2^{(2)}\ge \dotsb\ge a_2^{(n)},$$
$$
e_1^{(1)}\le  e_1^{(2)}\le \dotsb\le  e_1^{(n)},\quad
e_2^{(1)}> e_2^{(2)}> \dotsb> e_2^{(n)},$$
and
$a_1^{(i)}\ge  a_2^{(i)},\quad e_1^{(i)}\ge  e_2^{(i)},\quad
i=1,2,\dots,n.$
The number of all families $\P=(P_1,P_2,\break\dots,P_n)$ of non-intersecting
lattice paths $P_i:A_i\to E_i$, which stay weakly below
the line $x=y$, and where the paths of\/
$\P$ altogether
contain exactly $\ell$ EN-turns, is
\begin{multline}\label{CKe4.53}
\sum _{\ell_1+\dotsb+\ell_n=\ell} ^{}\det_{1\le  i,j\le  n}
\bigg(\binom{e_1^{(j)}-a_1^{(i)}+i-j}{\ell_i+i-j}
\binom{e_2^{(j)}-a_2^{(i)}-i+j}{\ell_i}\\
-\binom{e_1^{(j)}-a_2^{(i)}-i-j+3}{\ell_i-j+1}
\binom{e_2^{(j)}-a_1^{(i)}+i+j-3}{\ell_i+i-1}\bigg).
\end{multline}
\end{theorem}

\section{Multidimensional lattice paths}\label{CKsec:multi}

This section and the following three contain 
enumeration results for lattice paths in spaces
of higher dimension. Most of the time, we shall be concerned with
the $d$-dimensional lattice $\Z^d$. 
The coordinates in $d$-dimensional space will be
denoted by $x_1,x_2,\dots,x_d$.

Obviously, as a basis to start with, we need the number of all simple
paths in $\Z^d$
(that is, paths consisting of positive unit steps in the direction
of some coordinate axis) 
from $(a_1,a_2,\dots,a_d)$ to $(e_1,e_2,\dots,e_d)$.
Since these lattice paths can be seen as \index{permutation}permutations 
of $e_1-a_1$
steps in $x_1$-direction, $e_2-a_2$ steps in $x_2$-direction, \dots,
$e_d-a_d$ steps in $x_d$-direction, the answer is a multinomial
coefficient,
\begin{equation} \label{CKe7.1} 
\vv{\LL{(a_1,\dots,a_d)\to (e_1,\dots,e_d)}}=\binom {\sum _{i=1}
^{d}(e_i-a_i)} {e_1-a_1,e_2-a_2,\dots,e_d-a_d}.
\end{equation}

\section{Multidimensional lattice paths bounded by a hyperplane}
\label{CKs7.3}

Here, we consider simple lattice paths in $\Z^{d+1}$
restricted by a hyperplane of the form $x_0=\sum _{i=1} ^{d}\mu_ix_i$
where the $\mu_i$'s, $i=0,1,\dots,d$, are non-negative integers. It should be
noted that the reflection principle does not apply because, in general,
the set of steps is not invariant under reflection
with respect to such a hyperplane (except of course when $\mu_i=1$
for all $i$, in which case the reflection principle does apply).

\begin{theorem} \label{CKT7.4}
Let $\mu_0,\mu_1,\dots,\mu_d$ be non-negative integers and
$c_0,c_1,\dots,c_d$ integers such that $c_0\ge \sum _{i=1}
^{d}\mu_ic_i$. The number of all lattice paths from the origin to
$(c_0,c_1,\dots,c_d)$ not crossing the hyperplane $x_0= \sum _{i=1}
^{d}\mu_ix_i$ is given by
\begin{equation} \label{CKe7.20} 
\vv{\LL{\boldsymbol0\to (c_0,c_1,\dots,c_d)\mid x_0\ge\sum _{i=1}
^{d}\mu_ix_i}} =\frac {c_0-\sum _{i=1} ^{d}\mu_ic_i+1} 
{1+\sum _{i=0} ^{d}c_i}\binom {1+\sum _{i=0}
^{d}c_i} {c_0+1,c_1,c_2,\dots,c_d}.
\end{equation}
\end{theorem}

We omit the proof. Both proofs of Theorem~\ref{CKT1.4},
the generating function proof and
the proof by use of the cycle lemma, can be extended to proofs
of the above theorem.

\medskip
To conclude this section, we point out that Sato \cite{SatoAB} has extended
his generating function results for the number of paths in the plane
integer lattice between
two parallel lines that we presented in Section~\ref{CKs1.4} to the
multidimensional case. Similarly, the result of Niederhausen 
on the enumeration of paths in the plane integer lattice subject to
a piece-wise linear boundary, which was presented in Section~\ref{CKs1.3},
has a multidimensional extension, see \cite[Sec.~2.2]{NiedAC}.

\section{Multidimensional paths with a general boundary} \label{CKs7.5}

In this section we generalize the enumeration problem of
Section~\ref{CKs1.6} to arbitrary dimensions. Let $n_1,n_2,\dots,n_d$
be non-negative integers and $\mathbf a$ and
$\mathbf b$ be increasing 
integer functions defined on the box 
$$[\boldsymbol0,\mathbf n]:=\prod _{i=1}
^{d}\{0,1,\dots,n_i\}$$ 
such that $\mathbf a\ge \mathbf b$. $\mathbf a$
is increasing
means that $\mathbf a(\mathbf i)\le \mathbf a(\mathbf j)$ whenever $\mathbf i\le
\mathbf j$ in the usual product order. We ask for the number of all paths
in $\Z^{d+1}$ from $(\boldsymbol 0,\mathbf b(\boldsymbol 0))$ to
$(\mathbf n,\mathbf a(\mathbf n))$ that always stay in the region ``that is
bounded by $\mathbf a$ and $\mathbf b$", by which we mean the region
\begin{equation} \label{CKe7.40} 
\{(\mathbf i,y):\mathbf b(\mathbf i)\le y\le\mathbf a(\mathbf i)\}.
\end{equation}
The generalization of Theorem~\ref{CKT1.15}, due to
\index{Handa, B. R.}Handa and
\index{Mohanty, Sri Gopal}Mohanty \cite{HaMoAB}, reads as follows.

\begin{theorem} \label{CKT7.10}
Let $n_1,n_2,\dots,n_d$ be non-negative integers and $p=\prod _{i=1} ^{d}n_i$.
Assume that the points in the box $[\boldsymbol0,\mathbf n]$ are
$\boldsymbol0=\mathbf u_0,\mathbf u_1,\mathbf u_2,\dots,\mathbf u_p=\mathbf n$,
ordered lexicographically. Then the number of all lattice paths 
in $\Z^{d+1}$ from $(\boldsymbol 0,\mathbf b(\boldsymbol 0))$ to
$(\mathbf n,\mathbf a(\mathbf n))$ that always stay in the region
\eqref{CKe7.40} equals
\begin{multline} \label{CKe7.41} 
\vv{\LL{(\boldsymbol 0,\mathbf b(\boldsymbol 0))\to
(\mathbf n,\mathbf a(\mathbf n))\mid \mathbf a\ge \mathbf y\ge \mathbf b}}\\
=(-1)^{\sum _{i=1} ^{d}n_i+\prod _{i=1} ^{d}n_i}
\det_{0\le i,j\le \sum _{i=1} ^{d}n_i-1}\(\binom {\mathbf a(\mathbf u_{i})-
\mathbf b(\mathbf u_{j+1})+1} {\mathbf u_{j+1}-\mathbf u_{i}}\).
\end{multline}
\end{theorem}

The most elegant and illuminating proof is by the use of
non-intersecting lattice paths, see \cite{SulaAC}. 
\index{Sulanke, Robert A.}Sulanke proves
in fact a $q$-analogue in \cite{SulaAC}.

\section{The reflection principle in full generality}\label{CKsec:refl}

We have explained the reflection principle in the proof of
Theorem~\ref{CKT1.1} in Section~\ref{CKsec:lin1}, where it solved the
problem of counting simple lattice paths in the plane bounded by 
the diagonal. Nothing prevents us from applying the same idea
in a higher-dimensional setting.
It is then natural to ask: how far
can we go with the reflection principle? What is the most general
situation where it applies? This question was raised and answered by
\index{Gessel, Ira Martin}Gessel and \index{Zeilberger, Doron}Zeilberger 
\cite{GeZeAA}, and, independently, by \index{Biane, Philippe}Biane
\cite{BianAB} in a more restricted setting, 
see also \index{Grabiner, David J.}Grabiner and
\index{Magyar, Peter}Magyar \cite{GrMrAA}.

The standard example, which will serve as our running example,
is the problem of counting all paths from
$(a_1,a_2,\dots,a_d)$ to $(e_1,e_2,\dots,e_d)$ which always stay in
the region $x_1\ge x_2\ge \dots\ge x_d$. This problem is equivalent to
several other enumeration problems, the most prominent being 
the {\em $d$-candidate ballot problem}\index{ballot problem,
$d$-candidate} (for the 2-candidate ballot problem see
Section~\ref{CKsec:lin1}) and the problem
of counting {\em standard Young tableaux} of a given shape.

In the {\it $d$-candidate ballot problem} 
there are $d$ candidates in an election,
say $E_1,E_2,\dots,E_d$, $E_1$ receiving $e_1$ votes, $E_2$
receiving $e_2$ votes, \dots, $E_d$ receiving $e_d$ votes. How many
ways of counting the votes are there, such that  at any stage during
the counting $E_1$ has at least as many votes as $E_2$, $E_2$ has at
least as many votes as $E_3$, etc.? It is evident that by encoding
each vote for candidate $E_i$ by a step in $x_i$-direction this ballot
problem is transferred into counting paths from the origin to
$(e_1,e_2,\dots,e_d)$ which are staying in the region 
$x_1\ge x_2\ge \dots \ge x_d$. 

A {\it standard Young tableaux} of skew shape $\la/\mu$,
where $\la=(\la_1,\la_2,\dots,\la_d)$ and
$\mu=(\mu_1,\mu_2,\dots,\mu_d)$ are $d$-tuples of non-negative
integers which are in non-increasing order and satisfy $\la_i\ge\mu_i$
for all $i$, is an arrangement of the
numbers $1,2,\dots,\sum _{i=1} ^{d}(\la_i-\mu_i)$
of the form
$$\begin{matrix} 
&&&\pi_{1,\mu_1+1}&\multicolumn{3}{c}{\dotfill} &\pi_{1,\la_1}\\
&&\pi_{2,\mu_2+1}\quad \dots&\pi_{2,\mu_1+1}&\multicolumn{2}{c}{\dotfill} 
&\pi_{2,\la_2}\\
&\iddots&&\vdots&&\iddots\\
\pi_{d,\mu_d+1}&\multicolumn{3}{c}{\dotfill} &\pi_{d,\la_d}
\end{matrix}
$$
such that numbers along rows and columns are increasing.
See Chapter~[Standard Young Tableaux by Adin and Roichman] 
for more information on these important combinatorial
objects. By encoding an entry $i$ located in row $j$ of the tableau 
by a step in $x_j$-direction, $i=1,2,\dots,
\sum_{i=1} ^{d}(\la_i-\mu_i)$,
it is easy to see that standard tableaux of shape
$\la/\mu$ are in bijection with lattice paths from $\mu=(\mu_1,\mu_2,
\dots,\mu_d)$ to $\la=(\la_1,\la_2,\dots,\la_d)$
consisting of positive unit steps in the direction of some coordinate
axis, and which stay in the
region $x_1\ge x_2\ge \dots \ge x_d$. 

It is a classical result due to
\index{MacMahon, Percy Alexander}MacMahon \cite[p.~175]{MacMAZ}
(see also \cite[\S103]{MacMAA}) that
the solution to the counting problem is given by a determinant, 
see e.g.\ \cite[Prop.~7.10.3 combined with Cor.~7.16.3]{StanBI}.

\begin{theorem} \label{CKT7.1}%
Let $A=(a_1,a_2,\dots,a_d)$ and\/ $E=(e_1,e_2,\dots,e_d)$ be points in
$\Z^d$ with $a_1\ge a_2\ge\dots\ge a_d$ and $e_1\ge e_2\ge\dots\ge
e_d$. The number of all lattice paths from $A$ to $E$
consisting of positive unit steps in the direction of some coordinate
axis and staying in the region
$x_1\ge x_2\ge\dots\ge x_d$ equals
\begin{equation} \label{CKe7.2} 
\vv{L(A\to E\mid x_1\ge x_2\ge\dots\ge x_d)}=\Big(\sum _{i=1}
^{d}(e_i-a_i)\Big)! \det_{1\le i,j\le d}\(\frac {1} {(e_i-a_j-i+j)!}\).
\end{equation}
\end{theorem}

If the starting point $A$ equals the origin, then the above
determinant can be reduced to a Vandermonde determinant by
elementary column operations, and thus it can be evaluated in closed form.
If one rewrites the result appropriately, then
one arrives at the celebrated 
\index{hook formula}\index{formula, hook}{\it hook formula}
due to 
\index{Frame, James Sutherland}Frame,
\index{Robinson, Gilbert de Beauregard}Robinson
and \index{Thrall, Robert McDowell}Thrall \cite{FrRTAA}.
(We refer the reader to \cite[Sec.~3.10]{SagaAQ} or \cite[Cor.~7.21.6]{StanBI}
for unexplained terminology).

\begin{theorem} \label{CKT7.1H}%
Let $E=(e_1,e_2,\dots,e_d)$ be a point in
$\Z^d$ with $e_1\ge e_2\ge\dots\ge
e_d\ge0$. The number of all lattice paths from the origin to $E$
consisting of positive unit steps in the direction of some coordinate
axis and staying in the region
$x_1\ge x_2\ge\dots\ge x_d$ equals
\begin{equation} \label{CKe7.2H} 
\vv{L(A\to E\mid x_1\ge x_2\ge\dots\ge x_d)}=
\frac {\Big(\sum _{i=1}
^{d}e_i\Big)!}
{
\prod _{\rho} ^{}h(\rho)},
\end{equation}
where the product is over all cells $\rho$ in the Ferrers diagram
of the partition $(e_1,e_2,\break\dots,e_d)$, and\/ $h(\rho)$ is
the hook-length of the cell $\rho$.
\end{theorem}

It was pointed out by 
\index{Zeilberger, Doron}Zeilberger \cite{ZeilAD} that the formula
in \eqref{CKe7.2} can be proved by means of the reflection principle.
The natural environment for a ``general reflection principle" is within
the setting of \index{reflection group}{\em reflection groups\/}. A
{\em reflection group\/} is a group which is generated by all
reflections with respect to the hyperplanes $H$ in a given set $\mathcal
H$ of hyperplanes (in some $\R^d$). We review the facts about
reflection groups that are relevant for us below. 
For an excellent exposition of the subject see \index{Humphreys, James
Edward}Humphreys \cite{HumpAC}.
As we already said,
the situation of Theorem~\ref{CKT7.1} will be our running example.

As above, let $\mathcal H$ be a (finite) set of hyperplanes in some
$\R^d$. Let $W$ denote the group that is generated by the
corresponding reflections. By definition, $W$ is a subgroup of
$O(d)$. 
Some of the elements of $W$ happen to be reflections
with respect to a hyperplane (not necessarily belonging to $\mathcal H$),
and let $\mathcal R$ denote the collection
of all these hyperplanes. Of course, $\mathcal R$ contains $\mathcal H$. In
the example when $\mathcal H$ is the set of hyperplanes $H_i$ given by
\begin{equation} \label{CKe7.10} 
H_i:\quad x_i-x_{i+1}=0,\quad i=1,2,\dots,d-1,
\end{equation}
(these are the hyperplanes restricting the paths in Theorem~\ref{CKT7.1}),
the group $W$ is the permutation group $\mathfrak S_d$,
acting on $\R^d$ by permuting coordinates. All the reflections in
this group are the interchanges of two coordinates $x_i$ and $x_j$,
$1\le i<j\le d$, corresponding to the transpositions $(i,j)$ in
$\mathfrak S_d$. Hence, the corresponding set $\mathcal R$ of hyperplanes in this
case is 
\begin{equation}\label{CKe7.11}
\mathcal R=\{x_i-x_j=0:\ 1\le i<j\le d\}.
\end{equation} 

The hyperplanes in $\mathcal R$ cut the space into different regions. The
connected components of the complement of $\bigcup
_{H\in\mathcal R} ^{}H$
in $\R^d$  are called \index{chamber}{\em chambers\/}. Each
chamber is enclosed by a set $\mathcal R_0$ of bordering hyperplanes.
Clearly, $\mathcal R_0$ is a subset of $\mathcal R$.
In our
running example a typical chamber is the region
\begin{equation} \label{CKe7.12}
\{(x_1,x_2,\dots,x_d):\ x_1> x_2> \dots> x_d\},
\end{equation}
which is bounded by the hyperplanes in \eqref{CKe7.10}.  As a matter of
fact, in this special case any chamber has the form
\begin{equation} \label{CKe7.12'}
\{(x_1,x_2,\dots,x_d):\ x_{\si(1)}> x_{\si(2)}> \dots> x_{\si(d)}\},
\end{equation}
where $\si$ is some permutation in $\mathfrak S_d$.

It can be
shown that the reflections with respect to the hyperplanes in $\mathcal
R_0$ generate the complete reflection group $W$. Another important
fact is that, given one chamber $C$, all chambers are $w(C)$, where
$w$ runs through the elements of the reflection group $W$, all $w(C)$'s
being distinct.

Now we are in the position to formulate and prove \index{Gessel, Ira
Martin}Gessel and
\index{Zeilberger, Doron}Zeilberger's result \cite[Theorem~1]{GeZeAA}. The motivation for the
technical conditions in the statement of the theorem 
involving $k_H$ and $r_H$ is that they
make sure that it is not possible to ``jump" over a hyperplane
without touching it in a lattice point.

\begin{theorem}\label{CKT7.3}%
Let $C$ be a chamber of some reflection group $W$, determined by the
hyperplanes in the set $\mathcal R_0$. Let $\SS$ be a set of steps which is
invariant under $W$, i.e., $w(\SS)=\SS$, and with the property
that for all hyperplanes $H\in \mathcal R_0$ and all steps $s\in \SS$
the Euclidean inner product $(s,r_H)$ is either $0$ or $\pm k_H$,
where $k_H$ is a fixed constant, $r_H$ is a fixed non-zero vector 
perpendicular to $H$, both depending only on the hyperplane
$H$.
Furthermore, let $A$ and\/ $E$ be lattice points inside
the chamber $C$ such that also $w(A)$ and $w(E)$ are lattice points for all
$w\in W$, and such that for all hyperplanes $H\in \mathcal R_0$ 
the Euclidean inner product $(A,r_H)$ is an integral multiple of
$k_H$.

Then the number of all lattice paths from $A$ to $E$, with exactly $m$ 
steps from
$\SS$, and staying strictly inside the chamber $C$, equals
\begin{equation}\label{CKe7.13}
\vv{L_m(A\to E;\SS\mid \text {\em inside }C)}
=\sum _{w\in W} ^{}(\sgn w) \vv{L_m(w(A)\to E;\SS)},
\end{equation}
where $\sgn w=\det w$, considering $w$ as an orthogonal transformation
of\/ $\R^d$.
\end{theorem}

\begin{remark}\em
A weighted version of the above theorem in which steps 
carry weights such that images of steps under $W$ carry the same
weight holds as well.
\end{remark}

\begin{proof}
We may rewrite \eqref{CKe7.13} in the form
\begin{equation} \label{CKe7.14} 
\vv{L_m(A\to E;\SS\mid \text {inside }C)}
=\sum _{(w,P)} ^{}\sgn w ,
\end{equation}
where the sum is over all pairs $(w,P)$ with $w\in W$ and $P\in
L_m(w(A)\to E;\SS)$.
The proof of \eqref{CKe7.14} is by a sign-reversing involution on the set
of all such pairs $(w,P)$,
where $P$ touches at least one of the
hyperplanes in $\mathcal H_0$. Sign-reversing has to be
understood with respect to $\sgn w$. Provided the existence of such an
involution, the only contributions to the sum in \eqref{CKe7.14} would be
by pairs $(w,P)$ where $P$ does not touch any of the hyperplanes in
$\mathcal H_0$. We claim that this 
can only be the case for $w=\text{id}$. In fact, as we already mentioned, it
is one of the properties of a reflection group $W$ that, given one
chamber $C$, all chambers are $w(C)$, $w\in W$, and all $w(C)$'s
are distinct. Therefore, if $A$ is in $C$ and $w\ne \text {id}$, then $w(A)$
must be in a different chamber and so cannot be in $C$. Thus, evidently,
any path from $w(A)$ to $E$, the point $E$ being inside $C$, 
must touch at least one of the bordering hyperplanes.
This would prove \eqref{CKe7.14} and hence the theorem.

Now we construct the promised involution. 
Fix some order of the hyperplanes in $\mathcal H_0$.
Let $(w,P)$ be a pair with $w\in W$, $P\in L_m(w(A)\to E;\SS)$, and
$P$ touching at least one of the hyperplanes in $\mathcal H_0$.
Consider all meeting points of $P$ with hyperplanes in $\mathcal H_0$. Choose the
last meeting point along the path $P$ and denote it by $M$. 
$M$ must be a lattice point
because of the assumptions that involve the constants $k_H$. 
Let $H$ be the first
hyperplane (in the chosen order) that meets $P$.
Then we form the new path
$P'$ by reflecting the portion of $P$ from the starting point
$w(A)$ up to $M$ with respect to $H$ and leaving the portion
from $M$ to $P$ invariant. By assumption, reflection of a step from
$\SS$ is again a step in $\SS$. So, also $P'$ consists of
steps from $\SS$ only. Evidently, the starting point
of $P'$ is $w_Hw(A)$, where $w_H$ denotes the reflection
with respect to $H$. Hence,
$(w_Hw,P')$ is a pair under consideration for the sum in
\eqref{CKe7.14}, and $P'$ touches one of the
hyperplanes in $\mathcal H_0$ (namely $H$). This mapping is an involution since
nothing was changed after $M$. Moreover, we have
$\sgn w_Hw=-\sgn w$. Therefore it is also sign-reversing.
This completes the proof of the theorem.
\end{proof}

In the case of our running example, $W$ is the group of
\index{permutation group}\index{permutation}permutations 
of coordinates, $C$ is given by \eqref{CKe7.12}, the set
of hyperplanes is \eqref{CKe7.10}, the set of steps is
$\{\ep_1,\ep_2,\dots,\ep_d\}$, with $\ep_i$ denoting the
positive unit vector in $x_i$-direction. 
If $H_i$ is the hyperplane $x_i-x_{i+1}=0$,
we may choose $r_{H_i}=\ep_i-\ep_{i+1}$, so that
all constants $k_{H_i}$ are $1$, $i=1,2,\dots,d-1$.
Since the number of lattice paths between two given lattice points is 
given by a multinomial coefficient (see \eqref{CKe7.1}), it is then
not difficult to see that \eqref{CKe7.13} yields \eqref{CKe7.2} in
this special case.

\medskip
Which other examples are covered by the general setup in
Theorem~\ref{CKT7.3}? The answer is that all reflection groups that are
``relevant" in our context are completely classified. The meaning of ``relevant" is as
follows. In order our formula \eqref{CKe7.13} to make sense, the sum on the
right-hand side of \eqref{CKe7.13} should be finite. So, only
reflection groups that are ``discrete" and act ``locally finite" will
be of interest to us. It is exactly these reflection groups that are
precisely known (see \index{Humphreys, James Edward}Humphreys \cite[Sec.~4.10]{HumpAC},
\index{Bourbaki, Nicolas}Bourbaki 
\cite[Ch.~V, VI]{BourAA}). 

The classification of all {\it finite} reflection groups says that
any such finite reflection group decomposes into the direct product
of 
\index{irreducible reflection group}\index{reflection group,
  irreducible}irreducible reflection groups, all of which act
on pairwise orthogonal subspaces. These irreducible reflection
groups do not decompose further. 
There exist four infinite
families of types $I_2(m)$ ($m=1,2,\dots$), 
$A_d$, $B_d=C_d$, $D_d$ ($d=1,2,\dots$) of such groups, and the seven
exceptional groups of types $G_2$, $F_4$, $E_6$, $E_7$, $E_8$, and
$H_3$, $H_4$. (The
indices mark the dimension of the vector space on which they act
faithfully.) In addition, for most of these irreducible finite
reflection groups there exists an
\index{affine reflection group}\index{reflection group, affine}{\em affine\/} 
reflection group, which is infinite.
The finite reflection group is generated by all reflections
with respect to the hyperplanes
which run through a given point (we assume that this is the origin).
The affine reflection
group is generated by a larger set of hyperplanes, which includes the
aforementioned hyperplanes plus certain translates of them.
The reflection groups corresponding to $G_2$ are the same as those for
$I_2(6)$, therefore we need not consider $G_2$. 

\index{Grabiner, David J.}Grabiner and
\index{Magyar, Peter}Magyar \cite[p.~247]{GrMrAA} have determined all
possible step sets (up to dilation) 
for each of the irreducible reflection groups
such that the technical conditions of Theorem~\ref{CKT7.3} are satisfied.
Not for all types do there exist such step sets.
It should be noted however that the ``empty'' step 
$(0,0,\dots,0)$ can always be added
to any possible step set.
The following list describes all possible instances of
Theorem~\ref{CKT7.3} when applied to an irreducible finite or affine
reflection group. The results for lattice paths in chambers of
affine reflection groups have been made explicit by 
\index{Grabiner, David J.}Grabiner \cite{GrabAD}.

\medskip
\index{$I_2(m)$}{\em Types $H_3$, $H_4$, $F_4$, $E_8$, $I_2(m)$}: 
There are no possible step sets.

\medskip
\index{$A_{d-1}$}{\em Type $A_{d-1}$}: 
The set of reflecting hyperplanes is $\mathcal 
R=\{x_i-x_j=0:\ 1\le
i<j\le d\}$. Obviously, the reflection with respect to $x_i-x_j=0$
acts by interchanging the $i$-th and $j$-th coordinate.
So, the associated finite reflection group is the group of
\index{permutation group}\index{permutation}permutations 
of the coordinates $x_1,x_2,\dots,x_{d}$, which is
isomorphic to the permutation group $\mathfrak S_{d}$.
A typical chamber is $C=\{(x_1,x_2,\dots,x_{d}):\ x_1> x_2>\dots>x_{d}\}$.

Possible step sets are the sets
$$
\SS_k:=\{w\cdot(1,\dots,1,0,\dots,0):w\in\mathfrak S_{d}\},
\quad k=1,2,\dots,d,
$$
(with $k$ occurrences of $1$), all compatible with each other,
as well as
$$
\SS_k^\pm:=\{w\cdot(\pm1,\dots,\pm1,0,\dots,0):w\in\mathfrak S_{d}\},
\quad k=1\text{ and }k=d,
$$
(with $k$ occurrences of $\pm1$), which can not be mixed together.

Theorem~\ref{CKT7.1} is a direct consequence of Theorem~\ref{CKT7.3}
with $W=\mathfrak S_d$ and $\SS=\SS_1$.

The second standard application is the one for $\SS=\SS_d^\pm.$

\begin{theorem} \label{CKT7.1b}%
Let $A=(a_1,a_2,\dots,a_d)$ and $E=(e_1,e_2,\dots,e_d)$ be points in
$\Z^d$, with all $a_i$'s of the same parity, all $e_i$'s of the
same parity, $a_1> a_2>\dots> a_d$ and $e_1> e_2>\dots>
e_d$. The number of all lattice paths from $A$ to $E$
consisting of $m$ steps from $\SS_d^\pm$ and staying in the region
$x_1> x_2>\dots> x_d$ equals
\begin{equation} \label{CKe7.2b} 
\vv{L(A\to E; \SS_d^\pm\mid x_1> x_2>\dots> x_d)}=
\det_{1\le i,j\le d}\(\binom {m} {\frac {m+e_i-a_j} {2}}\).
\end{equation}
\end{theorem}

\medskip
We should point out that the lattice paths in Theorem~\ref{CKT7.1b} are
in bijection with configurations in the 
\index{lock step model}\index{model, lock step}{\em lock step model},
a frequently studied
\index{vicious walker model}\index{model, vicious walker}{\em vicious
  walker model}. On the other hand, the lattice paths 
in Theorem~\ref{CKT7.1} are
in bijection with configurations in another popular vicious walker model,
the so-called
\index{random turns model}\index{model, random turns}{\em random turns
  model}. We refer the reader to \cite[Sec.~2]{KratBT} for more
detailed comments on these connections.

\medskip
The associated affine reflection group, the affine reflection group
of type $\tilde A_{d-1}$, is generated by the
reflections with respect to the
hyperplanes $\mathcal R=\{x_i-x_j=k:\ 1\le i<j\le d,\ k\in\Z\}$. The
elements of this group are called 
\index{affine permutation}\index{permutation,
affine}{\em affine permutations\/}. They act by permuting the
coordinates $x_1,x_2,\dots,x_{d}$ and adding a vector
$(k_1,k_2,\dots,k_{d})$ with $k_1+k_2+\dots+k_{d}=0$.
A typical chamber\footnote{Actually, the chambers of affine
reflection groups are usually called 
\index{alcove}{\em alcoves}.} 
is $C=\{(x_1,x_2,\dots,x_d):\ x_1> x_2>\dots>x_{d}>x_1-1\}$.
For enumeration purposes, we inflate this chamber, see \eqref{CKeq:alcA}
below. 

The probably first explicitly stated enumeration result for lattice
paths in an affine chamber is the result below due to
\index{Filaseta, Michael}Filaseta \cite{FilaAA}, although it was
not formulated in that way.

\begin{theorem}%
Let $A=(a_1,a_2,\dots,a_d)$ and $E=(e_1,e_2,\dots,e_d)$ be points in
$\Z^d$ with $a_1> a_2>\dots> a_d$ and $e_1> e_2>\dots>
e_d$. The number of all paths from $A$ to $E$
consisting of steps from $\SS_1$ and staying in the chamber
\begin{equation} \label{CKeq:alcA} 
\{(x_1,x_2,\dots,x_d):x_1>x_2>\dots>x_d>x_1-N\}
\end{equation}
of type $\tilde A_{d-1}$, equals
\begin{multline} \label{CKe7.15} 
\vv{L(A\to E;\SS_1\mid x_1> x_2>\dots> x_d> x_1-N)}\\
=\Big(\sum _{i=1}
^{d}(e_i-a_i)\Big)!\sum _{k_1+\dots +k_d=0} ^{} 
\det_{1\le i,j\le d}\(\frac {1} {(e_i-a_j+k_iN)!}\).
\end{multline}
\end{theorem}

See \cite{KrMoAB} for a $q$-analogue.
It should be noted that
Theorem~\ref{CKT1.3} follows from the special case of the
above theorem where $d=2$.

For the step set $\SS_1^\pm$ consisting of positive {\it and\/}
negative unit steps in the direction of some coordinate axis,
we obtain the following result.

\begin{theorem}\label{CKthm:2}%
Let $m$ and $N$ be positive integers. Furthermore,
let $A=(a_1,a_2,\dots, a_d)$ and 
$E=(e_1,e_2,\dots,e_d)$ be vectors of integers in the
chamber \eqref{CKeq:alcA} of type $\tilde A_{d-1}$.
Then the number of lattice paths from $A$ to $E$ with exactly $m$
steps from $\SS_1^\pm$,
which stay in the alcove \eqref{CKeq:alcA}, is given by
the coefficient of $x^m/m!$ in
\begin{equation} \label{CKeq:A+-e}
\sum _{k_1+\dots+k_d=0} ^{}\det_{1\le i,j\le
d}\(I_{e_j-a_i+Nk_i}(2x)\),
\end{equation}
where $I_\al(x)$ is the modified Bessel function of the first kind
$$I_\al(x)=\sum _{j=0} ^{\infty}\frac {(x/2)^{2j+\al}}
{j!\,(j+\al)!}.$$
\end{theorem}

The last result for type $A_{d-1}$ which we state is the one for
paths in an affine chamber of type $\tilde A_{d-1}$ with
steps from $\SS_d^\pm$.

\begin{theorem}[{{\cite[Eq.~(35)]{GrabAD}}}] \label{CKthm:3}
Let $m$ and $N$ be positive integers. Furthermore,
let $A=(a_1,a_2,\dots,a_d)$ and 
$E=(e_1,e_2,\dots,e_d)$ be vectors of integers 
in the chamber \eqref{CKeq:alcA}
of type $\tilde A_{d-1}$ such that all $a_i$'s have the same
parity, and all $e_i$'s have the same parity.
Then the number of lattice paths
from $A$ to $E$ with exactly $m$ steps from $\SS_d^\pm$,
which stay in the chamber \eqref{CKeq:alcA}, is given by
\begin{multline} \label{CKeq:Ad}
\vv{L_m(A\to E;\SS_d^\pm\mid x_1> x_2>\dots> x_d> x_1-N)}\\
=\sum _{k_1+\dots+k_d=0} ^{}\det_{1\le i,j\le
d}\(\binom m {\frac {m+e_i-a_j} {2}+Nk_j}\).
\end{multline}
\end{theorem}

\medskip
\index{$B_d$}\index{$C_d$}{\em Types $B_d$, $C_d$}: 
The finite reflection groups of types $B_d$ and $C_d$ are identical.
The set of reflecting hyperplanes is $\mathcal 
R=\{\pm x_i\pm x_j=0:\ 1\le
i<j\le d\}\cup\{x_i=0:1\le i\le d\}$. 
Obviously, the reflection with respect to $x_i-x_j=0$
acts by interchanging the $i$-th and $j$-th coordinate, 
the reflection with respect to $x_i+x_j=0$
acts by interchanging the $i$-th and $j$-th coordinate and changing
the sign of both, while the 
reflection with respect to $x_i=0$ acts by changing sign of the $i$-th
coordinate. 

Here, the possible step sets are only $\SS_1^\pm$ and $\SS_d^ \pm$.

The associated finite reflection group is the group of
\index{permutation, signed}\index{signed permutation}{\em signed
permutations\/} of the coordinates $x_1,x_2,\dots,x_{d}$, which acts
by permuting and changing signs of (some of) the coordinates
$x_1,x_2,\dots,x_d$. 
It is frequently called the
\index{hyperoctahedral group}\index{group,
  hyperoctahedral}{\em hyperoctahedral group}, since it is the symmetry
group of a $d$-dimensional octahedron.
It is furthermore
isomorphic to the semidirect product $(\Z/2\Z)^d\rtimes\mathfrak S_d$. 
A typical chamber is $C=\{(x_1,x_2,\dots,x_d):\ x_1> x_2>\dots>x_{d}>0\}$.
We have the following enumeration result for lattice paths
staying in this chamber.


\begin{theorem} \label{CKT7.1c}%
Let $A=(a_1,a_2,\dots,a_d)$ and $E=(e_1,e_2,\dots,e_d)$ be points in
$\Z^d$, with all $a_i$'s of the same parity, all $e_i$'s of the
same parity, $a_1> a_2>\dots> a_d>0$ and $e_1> e_2>\dots>
e_d>0$. The number of all lattice paths from $A$ to $E$
consisting of $m$ steps from $\SS_d^\pm$ and staying in the region
$x_1> x_2>\dots> x_d>0$ equals
\begin{equation} \label{CKe7.2c} 
\vv{L_m(A\to E; \SS_d^\pm\mid x_1> x_2>\dots> x_d)}=
\det_{1\le i,j\le d}\(\binom {m} {\frac {m+e_i-a_j} {2}}
-\binom {m} {\frac {m+e_i+a_j} {2}}\).
\end{equation}
\end{theorem}

\medskip
The associated affine reflection group
now comes in two flavours, types $\tilde B_d$ and $\tilde C_d$.
A typical chamber of type $\tilde C_d$ is 
$C=\{(x_1,x_2,\dots,x_d):\ 1>x_1> x_2>\dots>x_{d}>0\}$,
while a typical chamber of type $\tilde B_d$ is 
$C=\{(x_1,x_2,\dots,x_d):\ x_1> x_2>\dots>x_{d}>0
\text{ and }x_1+x_2<1\}$,

Next we quote the two results from \cite{GrabAD} on the enumeration
of lattice paths in chambers of type $\tilde C_d$. 

\begin{theorem}[{{\cite[Eq.~(23)]{GrabAD}}}] \label{CKthm:6}
Let $m$ and $N$ be positive integers. Furthermore,
let $A=(a_1,a_2,\dots,a_d)$ and 
$E=(e_1,e_2,\dots,e_d)$ be vectors of integers in the
chamber
\begin{equation} \label{CKeq:alcC}
\{(x_1,x_2,\dots,x_n):N>x_1>x_2>\dots>x_d>0\}
\end{equation}
of type $\tilde C_d$.
Then the number of lattice paths from $A$ to $E$ with exactly $m$
steps from $\SS_1^\pm$,
which stay in the chamber \eqref{CKeq:alcC}, is given by
the coefficient of $x^m/m!$ in
\begin{equation} \label{CKeq:C+-e}
\det_{1\le i,j\le d}\(\frac {1} {N}\sum _{r=0} ^{2N-1}
\sin\frac {\pi r e_i} {N}\cdot\sin\frac {\pi r a_j} {N}\cdot
\exp\(2x\cos\frac {\pi r} {N}\)\).
\end{equation}
\end{theorem}

The result for lattice paths with steps from $\SS_d^\pm$ is the following.

\begin{theorem}[{{\cite[Eq.~(18)]{GrabAD}}}] \label{CKthm:7}
Let $m$ and $N$ be positive integers. Furthermore,
let $A=(a_1,a_2,\dots,a_d)$ and 
$E=(e_1,e_2,\dots,e_d)$ be vectors of integers in the
chamber \eqref{CKeq:alcC} of type $\tilde C_d$ such that all $a_i$'s
are of the same parity, and all $e_i$'s are of the same parity.
Then the number of lattice paths
from $A$ to $E$ with exactly $m$ steps from $\SS_d^\pm$,
which stay in the chamber \eqref{CKeq:alcC}, is given by
\begin{equation} \label{CKeq:Cd}
\det_{1\le i,j\le d}\(\frac {2^{m-1}} {N}\sum _{r=0} ^{4N-1}
\sin\frac {\pi r\la_t} {N}\cdot\sin\frac {\pi r\et_h} {N}\cdot
\cos^m\frac {\pi r} {2N}\).
\end{equation}
\end{theorem}

Enumeration results for lattice paths in a chamber of
type $\tilde B_d$ can be
also derived from Theorem~\ref{CKT7.3}. We omit to state them
here, but instead refer to \cite{GrabAD} and \cite[Theorems~8 and
  9]{KratBT}. 

\medskip
\index{$D_d$}{\em Type $D_d$}: 
The set of reflecting hyperplanes is $\mathcal 
R=\{\pm x_i\pm x_j=0:\ 1\le
i<j\le d\}$. Obviously, $D_d$ is a subset of $B_d$ or
$C_d$. The action of the reflection with respect to a hyperplane $\pm
x_i\pm x_j=0$ was already explained there.

The associated finite reflection group is the group of
\index{permutation, signed}\index{signed permutation}signed
permutations of the coordinates $x_1,x_2,\dots,x_{d}$ {\em with an
even number of sign changes\/}. It acts
by permuting the coordinates
$x_1,x_2,\dots,x_d$ and changing an even number of signs thereof. 
A typical chamber is $C=\{(x_1,x_2,\dots,x_d):\ x_1> x_2>\dots>x_{d-1}>\vv{x_d}\}$.

The associated affine reflection group is generated by the
reflections with respect to the
hyperplanes $\mathcal R=\{\pm x_i\pm x_j=k:\ 1\le i<j\le d,\ k\in\Z\}$. 
The elements of this group act by permuting the
coordinates $x_1,x_2,\dots,x_{d}$, changing an even number of signs
thereof, and adding a vector
$(k_1,k_2,\dots,k_{d})$ with $k_1+k_2+\dots+k_{d}\equiv
0\ (\text{mod }2)$.
A typical chamber is $C=\{(x_1,x_2,\dots,x_d):x_1>x_2>\dots>x_{d-1}>\vert
x_d\vert ,\text { and }
x_1+x_2<1\}$.

We omit the explicit statement of enumeration results for
types $D_d$ and $\tilde D_d$ which one may derive from
Theorem~\ref{CKT7.3}, and instead refer to \cite{GrabAD} and 
\cite[Theorems~10 and 11]{KratBT}.



\medskip
\index{$E_6$}\index{$E_7$}{\em Types $E_6$ and $E_7$}: 
There are possible step sets (see \cite[p.~247]{GrMrAA}),
but since this does not yield interesting enumeration results,
we refrain from discussing these two cases further.

\medskip
A non-example for the application of the reflection principle
has been discussed in Section~\ref{CKsec:var}, see Theorem~\ref{CKT7.2}.

\section{$q$-Counting of lattice paths and 
Rogers--Ra\-ma\-nu\-jan identities}\label{CKsec:RR}

In this section, we discuss some $q$-analogues of earlier (plain)
enumeration results, and we briefly present work showing the
close link between lattice path enumeration and the celebrated
Rogers--Ramanujan identities.

\begin{figure}
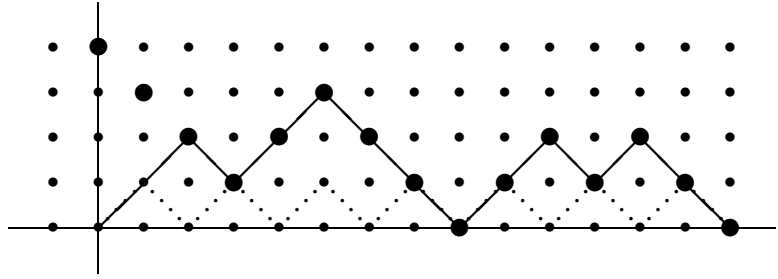
 
$$
\Gitter(15,5)(-1,0)
\Koordinatenachsen(15,5)(-1,0)
\Pfad(0,0),33433444334344\endPfad
\SPfad(0,0),34343434343434\endSPfad
\DickPunkt(0,4)
\DickPunkt(1,3)
\DickPunkt(2,2)
\DickPunkt(3,1)
\DickPunkt(4,2)
\DickPunkt(5,3)
\DickPunkt(6,2)
\DickPunkt(7,1)
\DickPunkt(8,0)
\DickPunkt(9,1)
\DickPunkt(10,2)
\DickPunkt(11,1)
\DickPunkt(12,2)
\DickPunkt(13,1)
\DickPunkt(14,0)
\hbox{\hskip6.5cm}
$$
\caption{A Dyck path}
\label{CKF3.D}
\end{figure}

As we have already seen in the introduction, one source of $q$-analogues is
area counting of lattice paths. This idea has also been used to
construct a $q$-analogue of 
\index{$q$-Catalan number}\index{number,
  $q$-Catalan}Catalan numbers. Given a Dyck path $P$
(see Section~\ref{CKs3.1}) from
$(0,0)$ to $(2n,0)$, let $\tilde a(P):=\frac {1} {2}\big(a(P)-n)$. 
In other words,
$\tilde a(P)$ is half of the area between $P$ and the ``lowest"
Dyck path from $(0,0)$ and $(2n,0)$, that is, the zig-zag path
in which up-steps and down-steps alternate. Alternatively, 
$\tilde a(P)$ counts the squares with side length $\sqrt 2$
(rotated by $45^\circ$) which fit between $P$ and the
zig-zag path.
In Figure~\ref{CKF3.D} this zig-zag path is indicated as the dotted
path. For the Dyck path shown with full lines in the figure,
we have $\tilde a(\,.\,)=6$. This (modified) area statistics is
now used to define the {\it $q$-Catalan number} $C_n(q)$ 
as the generating function for Dyck paths of length $2n$ with respect
to the statistics $\tilde a(\,.\,)$,
\begin{equation} \label{CKeq:qCat}
C_n(q)=\GF{\LL{(0,0)\to(2n,0);\{(1,1),\,(1,-1)\}};q^{\tilde a(P)}}. 
\end{equation}
By decomposing a given Dyck path $P$ uniquely into
$$
P=s_uP_1s_dP_2,
$$
where $s_u$ denotes an up-step, $s_d$ denotes a down-step,
and $P_1$ and $P_2$ are Dyck paths, one obtains the recurrence 
\begin{equation} \label{CKeq:qCatrek}
C_n(q)=\sum_{k=0}^{n-1}q^{k}C_k(q)C_{n-k-1}(q),
\quad n\ge1,
\end{equation}
with initial condition $C_0(q)=1$. These $q$-Catalan numbers
have been originally introduced by
\index{Carlitz, Leonard}Carlitz and
\index{Riordan, John}Riordan \cite{CaRiAA}.
We shall say more about these further below.

A different statistics can be derived from turn enumeration
(cf.\ Section~\ref{CKsec:turn}). In the geometry which we are
considering here, turns are peaks and valleys of a Dyck path.
For a peak at lattice point $S$, denote by $x(S)$ the number
of steps along the path from the origin to $S$.
(Equivalently, $x(S)$ is the ordinate of $S$.) In the Dyck path
in Figure~\ref{CKF3.D}, the peaks are at $(2,2)$, $(5,3)$, $(10,2)$,
and $(12,2)$. The ordinates are $x\big((2,2)\big)=2$,
$x\big((5,3)\big)=5$,
$x\big((10,2)\big)=10$,
$x\big((12,2)\big)=12$,
The 
\index{major index}\index{index, major}{\em major index} of a
Dyck path $P$, denoted by $\maj(P)$, 
is the sum of all values $x(S)$ over all peaks $S$ of
$P$. For the Dyck path in Figure~\ref{CKF3.D}, we have
$\maj(\,.\,)=2+5+10+12=29$.
\index{F\"urlinger, Johannes}\index{Hofbauer, Josef}F\"urlinger
and Hofbauer \cite{FuHoAA} used this statistic to define
alternative 
\index{$q$-Catalan number}\index{number,
  $q$-Catalan}$q$-Catalan numbers, namely
\begin{equation} \label{CKeq:qCat2}
c_n(q)=\GF{\LL{(0,0)\to(2n,0);\{(1,1),\,(1,-1)\}};q^{\maj(P)}} .
\end{equation}
They showed that
\begin{equation} \label{CKeq:qCat-qbin}
c_n(q)=\frac {1-q} {1-q^{n+1}}\qbinom {2n}nq, 
\end{equation}
the ``natural" $q$-analogue of the Catalan number in view of its
explicit formula\break $\frac {1} {n+1}\binom {2n}n$.
More on these $q$-Catalan numbers and further work in this
direction can be found in
\cite{FuHoAA,KratAF,KrMoAC}. These ideas have been extended to
Schr\"oder paths and numbers by 
\index{Bonin, Joseph E.}\index{Shapiro, Louis W.}\index{Simion, Rodica}%
Bonin, Shapiro and Simion in \cite{BoSSAA},

Returning to the $q$-Catalan numbers of Carlitz and Riordan, we see
that by the choice of $b_i=0$, $i=0,1,\dots$, $\la_i=q^ {i-1}z$,
$i=1,2,\dots$, in Theorem~\ref{CKT3.4}, we obtain a continued fraction
for the generating function of $q$-Catalan numbers $C_n(q)$, namely
\begin{equation} \label{CKeq:Cnq}
\sum _{n=0} ^{\infty}C_n(q)z^n
=\cfrac 1{1-
\cfrac {z}{1-
\cfrac {qz}{1-
\cfrac {q^2z}{1-
\cfrac {q^3z}{1-\ddots}}}}}\ . 
\end{equation}

If one substitutes $z=-q$ in this continued fraction, then it
becomes the reciprocal of the celebrated
\index{Ramanujan Iyengar, Srinivasa}%
Ramanujan continued fraction (cf.\ \cite[Ch.~7]{AndrAF})
\begin{equation} \label{CKeq:Ram}
1+\cfrac{q}
		{1+\cfrac{
			q^2}
				{1+\cfrac{
					q^3}
						{1+\cfrac{
					q^4}
\ddots}}}
=\frac {\sum_{n=0}^\infty \frac {q^{n^2}} {(q;q)_n}} 
{\sum_{n=0}^\infty \frac {q^{n(n+1)}} {(q;q)_n}} ,
\end{equation}
where $(\alpha;q)_n:=(1-\alpha)(1-\alpha q)\cdots (1-\alpha q^{n-1})$,
$n\ge1$, and $(\alpha;q)_0:=1$.

Numerator and denominator on the right-hand side of this
identity feature in the equally celebrated
\index{Rogers, Leonard James}\index{Ramanujan Iyengar, Srinivasa}%
Rogers--Ramanujan identities (cf.\ also \cite[Ch.~7]{AndrAF})
\begin{equation} \label{CKeq:Ram1} 
{\sum_{n=0}^\infty \frac {q^{n^2}} {(q;q)_n}} 
=
\frac {1} {(q;q^5)_\infty\,(q^4;q^5)_\infty}
\end{equation}
and
\begin{equation} \label{CKeq:Ram2} 
{\sum_{n=0}^\infty \frac {q^{n(n+1)}} {(q;q)_n}} 
=
\frac {1} {(q^2;q^5)_\infty\,(q^3;q^5)_\infty}.
\end{equation}
The fact that we came across the left-hand sides of these
identities by starting with lattice path counting problems
may indicate that the Rogers--Ramanujan identities themselves
may be linked with lattice path enumeration.
\index{Bressoud, David Marius}Bressoud \cite{BresAJ} was the 
first to actually set up such a link. Since then, 
this connection has been explored much further and extended
in various directions, particularly so in the physics
literature, see \cite{BePaAA,CiglAX,CoMaAA,MathAA,WelsAA}
and the references therein.

\section{Self-avoiding walks}\label{CKsec:SAW}

A path (walk) in a lattice in $d$-dimensional Euclidean space is called
\index{self-avoiding walk}\index{path, self-avoiding}%
\index{walk, self-avoiding}{\em self-\break avoiding} if it visits each 
point of the lattice at most once. One cannot expect useful
formulas for the exact enumeration of self-avoiding paths
(except in extremely simple lattices).
This is the reason why, with a few exceptions, research in this
area concentrates on {\it asymptotic} counting: how many
self-avoiding walks are there in a particular lattice,
consisting of $n$ steps from a given step set, 
asymptotically as $n$ tends to infinity? This is a notoriously
difficult question, which has been investigated mainly in the
physics and probability literature. In fact, the self-avoiding walk
constitutes a fascinating, vast subject area, which would need a chapter by
itself. We refer the reader to the standard book \cite{MaSlAA},
and to the more recent volumes \cite{GuttAB,JRenAA} which contain more recent
material on or relating to self-avoiding walks.

\bibliographystyle{plain}
\bibliography{encylatt}

}
\end{document}